\documentclass{tac}

\usepackage[utf8]{inputenc} 
\usepackage[T1]{fontenc}
\usepackage[english]{babel}
\usepackage{amsfonts, amssymb, amsmath}
\usepackage{mathtools} 
\usepackage{enumitem}
\usepackage{tensor} 

\usepackage[x11names]{xcolor}
\usepackage[hyperfootnotes=false, unicode=true]{hyperref}
\hypersetup{
	pdfauthor={Seerp Roald Koudenburg},
	pdftitle={Algebraic Kan extensions in double categories},
  pdfkeywords={double monad, algebraic Kan extension, free bicommutative Hopf monoid},
  linkbordercolor={NavajoWhite1},
  citebordercolor={LightCyan2},
  urlbordercolor={DarkSeaGreen1}
}

\usepackage{tikz} 
\usetikzlibrary{decorations.markings, matrix, arrows, external, fit, calc}

\tikzstyle{map} = [->, font=\scriptsize]
\tikzstyle{linj} = [left hook->, font=\scriptsize]
\tikzstyle{rinj} = [right hook->, font=\scriptsize]
\tikzstyle{sur} = [->>, font=\scriptsize]
\tikzstyle{cell} = [double,double equal sign distance,-implies, shorten >= 3.5pt, shorten <= 3.5pt, font=\scriptsize]
\tikzstyle{cell125} = [cell, shorten >= 1pt, shorten <= 1pt]
\tikzstyle{eq} = [double,double equal sign distance]

\tikzstyle{iso} = [above, sloped, inner sep=1.5pt]
\tikzstyle{nat} = [above, sloped, inner sep=2pt]
\tikzstyle{desc} = [fill=white, inner sep=2pt]

\tikzstyle{textbaseline} = [baseline={([yshift=-2.8pt]current bounding box.center)}]

\tikzstyle{barred} = [decoration={markings, mark=at position 0.5 with {\draw[-] (0,-1.5pt) -- (0,1.5pt);}}, postaction ={decorate}]

\tikzstyle{math} = [matrix of math nodes, row sep=1em, column sep=1em, text height=1.5ex, text depth=0.25ex, nodes in empty cells]
\tikzstyle{math125em} = [math, row sep=1.25em, column sep=1.25em]
\tikzstyle{math175em} = [math, row sep=1.75em, column sep=1.75em]
\tikzstyle{math2em} = [math, row sep=2em, column sep=2em]
\tikzstyle{tab} = [math, row sep=1.75em, column sep=1.25em]

\makeatletter
\def\slashedarrowfill@#1#2#3#4#5{%
  $\m@th\thickmuskip0mu\medmuskip\thickmuskip\thinmuskip\thickmuskip
   \relax#5#1\mkern-7mu%
   \cleaders\hbox{$#5\mkern-2mu#2\mkern-2mu$}\hfill
   \mathclap{#3}\mathclap{#2}%
   \cleaders\hbox{$#5\mkern-2mu#2\mkern-2mu$}\hfill
   \mkern-7mu#4$%
}
\def\rightslashedarrowfill@{%
  \slashedarrowfill@\relbar\relbar\mapstochar\rightarrow}
\newcommand\xslashedrightarrow[2][]{%
  \ext@arrow 0055{\rightslashedarrowfill@}{#1}{#2}}
\makeatother

\def\slashedrightarrow{\xslashedrightarrow{}}

\newtheoremrm{definition}{Definition}

\providecommand{\corref}[1]{Corollary~\ref{#1}}
\providecommand{\defref}[1]{Definition~\ref{#1}}
\providecommand{\exref}[1]{Example~\ref{#1}}

\providecommand{\lemref}[1]{Lemma~\ref{#1}}
\providecommand{\propref}[1]{Proposition~\ref{#1}}

\providecommand{\thmref}[1]{Theorem~\ref{#1}}

\providecommand{\of}{\circ}
\providecommand{\iso}{\cong}

\providecommand{\brar}{\slashedrightarrow}
\providecommand{\xrar}[1]{\xrightarrow{#1}}
\providecommand{\xlar}[1]{\xleftarrow{#1}}

\providecommand{\Rar}{\Rightarrow}
\providecommand{\xRar}[1]{\xRightarrow{#1}}

\providecommand{\into}{\hookrightarrow}

\providecommand{\eps}{\varepsilon}

\DeclareMathOperator{\dash}{--}
\providecommand{\ndash}{\nobreakdash-}

\providecommand{\tens}{\otimes}

\providecommand{\ul}[1]{\underline{#1}{}}
\providecommand{\ull}[1]{\ul{\ul{#1}}{}}

\providecommand{\mf}[1]{\mathfrak{#1}}
\providecommand{\mb}[1]{\mathbb{#1}}

\providecommand{\brcs}[1]{\lbrace #1 \rbrace}

\providecommand{\brks}[1]{\lbrack #1 \rbrack}
\providecommand{\bigbrks}[1]{\bigl\lbrack #1 \bigr\rbrack}

\providecommand{\pars}[1]{\left(#1\right)}
\providecommand{\bigpars}[1]{\bigl(#1\bigr)}
\providecommand{\Bigpars}[1]{\Bigl(#1\Bigr)}
\providecommand{\lns}[1]{\lvert#1\rvert}

\providecommand{\angles}[1]{\langle#1\rangle}

\providecommand{\gen}[1]{\angles{#1}}

\providecommand{\ol}[1]{\overline{#1}}

\providecommand{\set}[1]{\brcs{#1}}

\providecommand{\djunion}{\sqcup}

\providecommand{\sym}{\Sigma}

\providecommand{\card}[1]{\lns{#1}}

\providecommand{\pb}[2]{\tensor[_{#1}]{\times}{_{#2}}}

\providecommand{\natarrow}{\Rightarrow}

\providecommand{\map}[3]{#1\colon#2\to#3}

\providecommand{\nat}[3]{#1\colon#2\natarrow#3}
\providecommand{\cell}[3]{#1\colon#2\Rightarrow#3}

\providecommand{\hmap}[3]{#1\colon#2\slashedrightarrow#3}

\providecommand{\inv}[1]{{#1}^{-1}}

\DeclareMathOperator{\id}{id}

\DeclareMathOperator{\ob}{ob}

\providecommand{\ladj}{\dashv}

\providecommand{\op}[1]{#1^\textup{op}}
\providecommand{\co}[1]{#1^\textup{co}}

\providecommand{\catvar}[1]{\mathcal{#1}}

\providecommand{\Dbl}{\mathsf{Dbl}}

\providecommand{\2}{\mathsf 2}
\renewcommand{\C}{\catvar C}
\providecommand{\D}{\catvar D}
\providecommand{\E}{\catvar E}
\renewcommand{\L}{\catvar L}
\providecommand{\K}{\catvar K}

\providecommand{\V}{\catvar V}

\providecommand{\Set}{\mathsf{Set}}
\providecommand{\Cat}{\mathsf{Cat}}
\providecommand{\enCat}[1]{#1\text-\Cat}
\providecommand{\twoCat}{2\text{-}\Cat}

\providecommand{\Dbl}{\mathsf{Dbl}}
\providecommand{\lDbl}{\Dbl_\textup l}
\providecommand{\oDbl}{\Dbl_\textup o}
\providecommand{\noDbl}{\Dbl_\textup{no}}
\providecommand{\nlDbl}{\Dbl_\textup{nl}}

\providecommand{\Tens}{\bigotimes}

\providecommand{\inhom}[1]{\brks{#1}}

\providecommand{\PAlg}[1]{\mathsf{Alg}(#1)}
\providecommand{\sMonCat}{\mathsf{sMonCat}}

\providecommand{\Mat}[1]{#1\text-\mathsf{Mat}}
\providecommand{\Span}[1]{\mathsf{Span}(#1)}
\providecommand{\Mod}[1]{\mathsf{Mod}(#1)}
\providecommand{\Prof}{\mathsf{Prof}}
\providecommand{\enProf}[1]{#1\text-\Prof}
\providecommand{\inProf}[1]{\Prof(#1)}

\providecommand{\Alg}[2]{\mathsf{Alg}_\textup{#1}\!\pars{#2}}
\providecommand{\MonProf}[2]{#1\text-\mathsf{MonProf}_{\textup{#2}}}
\providecommand{\sMonProf}[2]{#1\text-\mathsf{sMonProf}_{\textup{#2}}}
\providecommand{\hc}{\odot}

\providecommand{\tab}[1]{\gen{#1}}

\title{Algebraic Kan extensions in double categories}
\author{Seerp Roald Koudenburg}
\address{Cortenhoeve 14\\2411 JM Bodegraven\\The Netherlands}
\eaddress{roaldkoudenburg@gmail.com}
\copyrightyear{2015}
\thanks{Many of the results in this paper first appeared as part of my PhD thesis ``Algebraic weighted colimits'' that was written under the guidance of Simon Willerton. I would like to thank Simon for his support and encouragement. I thank the anonymous referee for helpful suggestions and the prompt review of this paper, and also the University of Sheffield for its financial support of my PhD studies.}
\keywords{double monad, algebraic Kan extension, free bicommutative Hopf monoid}
\amsclass{18D05, 18C15, 18A40, 16T05}

\begin{document}
	\maketitle
	\begin{abstract}
		We study Kan extensions in three weakenings of the Eilenberg-Moore double category associated to a double monad, that was introduced by Grandis and Paré. To be precise, given a normal oplax double monad $T$ on a double category $\K$, we consider the double categories consisting of pseudo $T$-algebras, `weak' vertical $T$-morphisms, horizontal $T$-morphisms and $T$-cells, where `weak' means either `lax', `colax' or `pseudo'. Denoting these double categories by $\Alg wT$, where w = l, c or ps accordingly, our main result gives, in each of these cases, conditions ensuring that (pointwise) Kan extensions can be lifted along the forgetful double functor $\Alg wT \to \K$. As an application we recover and generalise a result by Getzler, on the lifting of pointwise left Kan extensions along symmetric monoidal enriched functors. As an application of Getzler's result we prove, in suitable symmetric monoidal categories, the existence of bicommutative Hopf monoids that are freely generated by cocommutative comonoids.
	\end{abstract}

	\section{Introduction}
	When given a symmetric monoidal functor $\map jAB$ one, instead of considering `ordinary' Kan extensions along $j$, often considers `symmetric monoidal' Kan extensions along $j$. Precisely, while we can consider Kan extensions along $j$ in the $2$-category $\Cat$ of categories, functors and transformations, it is often more useful to consider such extensions in the $2$-category $\mathsf{sMonCat}$ of symmetric monoidal categories, symmetric monoidal functors and monoidal transformations.
	
	For example consider algebras of `PROPs': informally, a PROP $\mb P$ is a certain kind of symmetric monoidal category that describes a type of algebraic structure, on the objects $A$ of any symmetric monoidal category, that involves operations of the form $A^{\tens m} \to A^{\tens n}$. There is, for instance, a PROP that describes monoids, and one that describes Hopf monoids. An algebra $A$ of $\mb P$, in a symmetric monoidal category $M$, is then an object $A$ in $M$ that is equipped with the algebraic structure described by $\mb P$---formally, $A$ is simply a symmetric monoidal functor $\map A{\mb P}M$. Presenting algebraic structures as symmetric monoidal functors like this has the advantage that, often, freely generated such structures can be constructed as left Kan extensions in $\mathsf{sMonCat}$. For example, the PROP $\mb C$ of cocommutative comonoids embeds into the PROP $\mb H$ of bicommutative Hopf monoids, and the bicommutative Hopf monoids in $M$ that are freely generated by cocommutative comonoids correspond precisely to left Kan extensions along the embedding $\mb C \into \mb H$. We will show in \S6 that, as a consequence of our main result, all such Kan extensions exist under reasonable conditions on $M$.
	
	As a second example, Getzler shows in \cite{Getzler09} that `operads' in $M$---which describe algebraic structures involving operations of the form $A^{\tens n} \to A$---can also be regarded as symmetric monoidal functors $\mathsf T \to M$, where $\mathsf T$ is a certain symmetric monoidal category that describes the algebraic structure of operads. Several generalisations of operads, such as `cyclic operads' and `modular operads' can be similarly presented. The main result of \cite{Getzler09} gives conditions ensuring that left Kan extensions can be `lifted' along the forgetful $2$-functor $\mathsf{sMonCat} \to \Cat$\footnote{In fact it gives conditions ensuring that left Kan extensions can be lifted along the forgetful $2$-functor $\V\text -\mathsf{sMonCat} \to \enCat\V$, from the $2$-category of symmetric monoidal categories enriched in a suitable closed symmetric monoidal category $\V$, to the $2$-category of $\V$-categories.} which, as a consequence, gives a coherent way of freely generating many types of generalised operad. To describe this `lifting' of Kan extensions more precisely, let us consider symmetric monoidal functors $\map jAB$ and $\map dAM$: we say that the ordinary left Kan extension $\map lBM$ of $d$ along $j$, in $\Cat$, can be lifted to $\mathsf{sMonCat}$ whenever $l$ admits a `canonical' symmetric monoidal structure that makes it into the left Kan extension of $d$ along $j$ in $\mathsf{sMonCat}$.
	
	Using the language of double categories, the horizontal dual of the main result of this paper, which is stated below, can be thought of as generalising Getzler's result of `lifting symmetric monoidal Kan extensions' to the broader idea of `lifting algebraic Kan extensions'. In the remainder of this introduction we will informally explain some of the details of the main result and at the same time describe the contents of this paper. We will however not explain the condition (p) except for remarking that, to recover Getzler's result, we will apply the main result to the double monad whose algebras are symmetric monoidal categories and in that case (p) holds as soon as the tensor product of $M$ preserves colimits in each variable.
	
	\begin{theorem}[Horizontal dual of \thmref{main result}.] \label{horizontal dual main theorem}
  	Let $T$ be a normal pseudo double monad on a double category $\K$, and assume that $T$ is pointwise left exact. Let `weak' mean either `colax', `lax' or `pseudo'. Given pseudo $T$\ndash algebras $A$, $B$ and $M$, consider the following conditions on a horizontal $T$\ndash morphism \mbox{$\hmap JAB$} and a weak vertical $T$\ndash morphism $\map dAM$:
  	\begin{enumerate}
  		\item[(p)]	the algebraic structure of $M$ preserves the pointwise left Kan extension of $d$ along $J$;
  		\item[(e)]	the structure cell of $J$ is pointwise left $d$-exact;
  		\item[(l)]	the forgetful double functor $\Alg wT \to \K$ lifts the pointwise left Kan extension of $d$ along $J$.
  	\end{enumerate}
  	The following hold:
  	\begin{enumerate}[label=\textup{(\alph*)}]
  		\item if `weak' means `colax' then (e) implies (l);
  		\item if `weak' means `lax' then (p) implies (l);
  		\item if `weak' means `pseudo' then any two of (p), (e) and (l) imply the third.
  	\end{enumerate}
  \end{theorem}
	
	We start in \S2 by recalling the relevant terminology on double categories. Briefly, the features of double categories that distinguish them from bicategories are that, besides objects, they consist of two types of morphisms, `vertical' ones denoted $\map fAB$ and `horizontal' ones denoted $\hmap JAB$, as well as cells that are shaped like squares, which can be composed both vertically and horizontally. Amongst others, we will recall the notions of `restriction' and `extension' of a horizontal morphism along vertical morphisms, as well as the notion of `tabulation' of a horizontal morphism---the latter can be thought of as generalising the notion of comma object in $2$-categories. Following this we recall the notion of `left Kan extension' in any double category $\K$ which, as can be deduced from the way the main result is stated, defines the extension of a vertical morphism $\map dAM$ along a horizontal morphism $\hmap JAB$; the resulting extension, if it exists, is a vertical morphism of the form $\map lBM$ which, like in the case of $2$-categories, is defined by a cell satisfying a universal property. In fact, the notion of Kan extension in $\K$ generalises that of Kan extension in the $2$-category $V(\K)$, which is the `vertical part' of $\K$, as soon as $\K$ has all restrictions. Double categories that have all restrictions are called `equipments'.
	
	We continue by recalling the stronger notion of `pointwise' Kan extension and, to get some feeling for it, study in detail such extensions in the equipment $\enProf\V$ consisting of categories enriched in a suitable monoidal category $\V$, $\V$-functors, `$\V$-profunctors' and their transformations. In doing so we are naturally lead to an extension of the notions of `weighted limit' and `Kan extension along $\V$-functors' to the setting in which $\V$ is just a monoidal category, where classically (see e.g.\ \cite{Kelly82}) $\V$ is assumed to be closed symmetric monoidal. At the end of \S2 we recall the notion of `pointwise left exactness', which is crucial in the statement of the main result: a cell of a double category is called pointwise left exact if, whenever it is vertically postcomposed with a cell defining a pointwise left Kan extension, the resulting composite defines again a such an extension. For example any $\V$-natural transformation of $\V$-functors that satisfies the `left Beck-Chevalley condition', in the sense of e.g.\ \cite{Guitart80}, gives rise to a pointwise left exact cell in $\enProf\V$.
	
	In \S3 we recall the description of the $2$-category consisting of double categories, so-called `normal pseudo' double functors, and double transformations; by a normal pseudo double monad $T = (T, \mu, \eta)$ on a double category $\K$, as in the main result, we simply mean a monad on $\K$ in this $2$-category. We remark that, like all double transformations, the multiplication $\mu$ and the unit $\eta$ consist of a natural family of vertical morphisms, one for each object in $\K$, as well as a natural family of cells, one for each horizontal morphism in $\K$. We call the double monad $T$ pointwise left exact, as is required in the main result, whenever each of these cells is pointwise left exact.
	
	Any normal pseudo double monad $T$ on $\K$ induces a strict $2$-monad on the vertical part $V(\K)$ of $\K$, so that we can consider pseudo $V(T)$-algebras in $V(\K)$, in the usual $2$-categorical sense of e.g.\ \cite{Street74}, as well as `weak' $V(T)$-morphisms between them, where we take `weak' to mean either `lax', `colax' or `pseudo'. Moreover, we shall consider a notion of `horizontal $T$-morphism' between pseudo $V(T)$\ndash algebras, which generalises slightly a notion introduced in \cite{Grandis-Pare04}, and show that, for each choice of `weak', pseudo $V(T)$\ndash algebras, weak $V(T)$-morphisms, horizontal $T$-morphisms, together with an appropriate notion of `$T$-cell', form a double category $\Alg wT$, where the subscript $\textup w \in \set{\textup c, \textup l, \textup{ps}}$ according to the choice of weakness. We remark that, without going into details, a horizontal $T$-morphism $A \brar B$ is a horizontal morphism $\hmap JAB$ that is equipped with a cell $\bar J$ defining its algebraic structure. If the pointwise left Kan extension of a map $\map dAM$ along $J$ exists in $\K$ then its defining cell can be vertically postcomposed with $\bar J$ and, in these terms, condition (e) of the main result means that the resulting composite defines again a pointwise left Kan extension.
	
	In \S4 we consider Kan extensions in $\Alg wT$. We show that some restrictions (of horizontal morphisms) can be lifted along the double functor $\Alg wT \to \K$, that forgets the algebraic structure, and that all tabulations can be lifted. We conclude that, if $\K$ is an equipment that has so-called `opcartesian' tabulations, then pointwise Kan extensions in $\Alg wT$ can be defined in terms of ordinary Kan extensions, in a way that is analogous to Street's definition of pointwise Kan extension in $2$-categories, that was introduced in \cite{Street74}.
	
	Finally, \S5 is devoted to stating and proving the main result. We show that Getzler's result can be recovered and, in \S6, that the existence of freely generated bicommutative Hopf monoids can be obtained as an application, as promised.
		
	\section{Kan extensions in double categories}
	Throughout this paper the terminology and notation of \cite{Koudenburg14a} are used, which we recall here.
	
	\subsection{Double categories.}
	By a double category we mean a weakly internal category in the $2$-category $\Cat$ of categories, functors and natural transformations, as follows.
	\begin{definition} \label{double category}
		A \emph{double category} $\K$ consists of a diagram of functors
		\begin{displaymath}
			\begin{tikzpicture}[textbaseline]
    		\matrix(m)[math, column sep=2em]{\K_\textup c \pb RL \K_\textup c & \K_\textup c & \K_\textup v \\};
    		\path[map]  (m-1-1) edge[transform canvas={yshift=6pt}] node[above] {$\pi_1$} (m-1-2)
    												edge[transform canvas={yshift=-1.5pt}] node[above, inner sep=0.75pt] {$\hc$} (m-1-2)
    												edge[transform canvas={yshift=-6pt}] node[below] {$\pi_2$} (m-1-2)
        		        (m-1-2) edge[transform canvas={yshift=5pt}] node[above] {$L$} (m-1-3)
        		                edge[transform canvas={yshift=-5pt}] node[below] {$R$} (m-1-3)
        		        (m-1-3) edge[transform canvas={yshift=-1.5pt}] node[above, inner sep=0.75pt] {$1$} (m-1-2);
  		\end{tikzpicture}
		\end{displaymath}
		 (where $\K_\textup c \pb RL \K_\textup c$ is the pullback of $R$ and $L$, with projections $\pi_1$ and $\pi_2$), such that
		\begin{displaymath}
  		L \of \hc = L \of \pi_1, \qquad R \of \hc = R \of \pi_2 \qquad \text{and} \qquad L \of 1 = \id = R \of 1,
  	\end{displaymath}
		together with natural isomorphisms
		\begin{displaymath}
			\mf a\colon (J \hc H) \hc K \iso J \hc (H \hc K), \quad \mf l\colon 1_A \hc M \iso M \quad \text{and} \quad \mf r\colon M \hc 1_B \iso M,
		\end{displaymath}
		for all $(J, H, K) \in \K_\textup c \pb RL \K_\textup c \pb RL \K_\textup c$ and $M \in \K_\textup v$ with $LM = A$ and $RM = B$. The natural isomorphisms $\mf a$, $\mf l$ and $\mf r$ are required to satisfy the usual coherence axioms for a monoidal category or bicategory (see e.g.\ Section VII.1 of \cite{MacLane98}), while their images under both $R$ and $L$ must be identities.
	\end{definition}
	
	The category $\K_\textup v$ consists of the objects and \emph{vertical morphisms} of $\K$, while the objects and morphisms of $\K_\textup c$ form the \emph{horizontal morphisms} and \emph{cells} of $\K$. We denote a horizontal morphism $J \in \K_\textup c$ with $LJ = A$ and $RJ = B$ as a barred arrow $\hmap JAB$, while a cell $\map\phi JK \in \K_\textup c$ with $L\phi = \map fAC$ and $R\phi = \map gBD$ will be depicted as
	\begin{equation} \label{cell}
	  \begin{tikzpicture}[textbaseline]
	    \matrix(m)[math175em]{A & B \\ C & D \\};
	    \path[map]  (m-1-1) edge[barred] node[above] {$J$} (m-1-2)
	                        edge node[left] {$f$} (m-2-1)
	                (m-1-2) edge node[right] {$g$} (m-2-2)
	                (m-2-1) edge[barred] node[below] {$K$} (m-2-2);
	    \path[transform canvas={shift={($(m-1-2)!(0,0)!(m-2-2)$)}}] (m-1-1) edge[cell] node[right] {$\phi$} (m-2-1);
	  \end{tikzpicture}
	\end{equation}
	and denoted by $\cell\phi JK$; we call $J$ and $K$ the \emph{horizontal source} and \emph{target} of $\phi$, while $f$ and $g$ are called its \emph{vertical source} and \emph{target}. A cell whose vertical source and target are identities is called \emph{horizontal}.
	
	The compositions of $\K_\textup v$ and $\K_\textup c$, which are associative and unital, define vertical compositions for the vertical morphisms and cells of $\K$, both of which we denote by $\of$, and whose identities are denoted $\map{\id_A}AA$ and $\cell{\id_J}JJ$ (which is a horizontal cell). The functors $\map\hc{\K_\textup c \pb RL \K_\textup c}{\K_\textup c}$ and $\map 1{\K_\textup v}{\K_\textup c}$ define horizontal compositions for the horizontal morphisms and cells of $\K$, which are associative up to invertible horizontal cells $\mf a\colon (J \hc H) \hc K \iso J \hc (H \hc K)$, that are called \emph{associators}, and unital up to invertible horizontal cells $\mf l\colon 1_A \hc M \iso M$ and $\mf r\colon M \hc 1_B \iso M$, called \emph{unitors}. A cell $\phi$ as in \eqref{cell} that has units $J = 1_A$ and $K = 1_C$ as horizontal source and target is called \emph{vertical}; we will often denote it by the more descriptive $\cell\phi fg$.
	
	To make our drawings of cells more readable we will depict both vertical identities and horizontal units simply as $\begin{tikzpicture}[textbaseline] \path (0.2,0) edge[eq] (0.7,0); \node at (0,0) {$A$}; \node at (0.9,0) {$A$};\end{tikzpicture}$. Likewise when writing down, or depicting, compositions of cells we will often leave out the associators and unitors.
	
	Every double category $\K$ contains both a \emph{vertical $2$-category} $V(\K)$, consisting of its objects, vertical morphisms and vertical cells, as well as a \emph{horizontal bicategory} $H(\K)$, consisting of its objects, horizontal morphisms and horizontal cells. For details see Definition~1.8 of \cite{Koudenburg14a}. Like $2$-categories, any double category $\K$ has both a \emph{vertical dual} $\op \K$, that is given by taking $(\op \K)_\textup v = \op{(\K_\textup v)}$ and $(\op \K)_\textup c = \op{(\K_\textup c)}$, and a \emph{horizontal dual} $\co \K$, that is obtained by swapping the functors $L$ and $\map R{\K_\textup c}{\K_\textup v}$; for details see Definition~1.7 of \cite{Koudenburg14a}.
	
	\begin{example} \label{example: profunctors}
		The archetypal double category is that of profunctors. Denoted $\Prof$, it has small categories as objects and functors as vertical morphisms, while its horizontal morphisms  $\hmap JAB$ are \emph{profunctors}, that is functors of the form \mbox{$\map J{\op A \times B}\Set$}. A cell $\phi$ in $\Prof$, of the form \eqref{cell}, is a natural transformation \mbox{$\nat\phi J{K(f,g)}$} where $K(f, g) = K \of (\op f \times g)$. The horizontal composite \mbox{$J \hc H$} of profunctors $\hmap JAB$ and $\hmap HBE$ is given by choosing a coequaliser \mbox{$(J \hc H)(x, z)$} of each pair of functions
		\begin{equation} \label{composite of unenriched profunctors}
			\coprod\limits_{\map v{y_1}{y_2} \in B} J(x, y_1) \times H(y_2, z) \rightrightarrows \coprod\limits_{y \in B} J(x, y) \times H(y, z),
		\end{equation}
		that are induced by postcomposing the maps in $J(x, y_1)$ with $\map v{y_1}{y_2}$ and precomposing the maps in $H(y_2, z)$ with $v$. The unit profunctor $\hmap{1_A}AA$ is given by the hom-sets $1_A(x_1, x_2) = A(x_1, x_2)$.	We shall describe a $\V$-enriched variant of $\Prof$, where $\V$ is a suitable monoidal category, in detail in \exref{example: enriched profunctors}.
	\end{example}
	
	\begin{example} \label{example: spans}
		A \emph{span} $\hmap JAB$ in a category $\E$ is a diagram $A \xlar{d_0} J \xrar{d_1} B$ in $\E$. Spans can be composed as soon as $\E$ has pullbacks; with this composition objects and morphisms of $\E$, together with spans in $\E$ and their morphisms, form a double category $\Span\E$. Details can be found in Example 1.5 of \cite{Koudenburg14a}.
	\end{example}
	
	\begin{example} \label{example: matrices}
		Given a category $\V$ and sets $A$ and $B$, a \emph{$\V$-matrix} $\hmap JAB$ is simply a family of $\V$-objects $J(x,y)$, one for each pair of objects $x \in A$ and $y \in B$. If $\V$ is equipped with a monoidal structure $(\tens, 1)$, and has coproducts that are preserved by $\tens$ on both sides, then sets, functions between sets and $\V$-matrices form a double category $\Mat\V$ as follows. A cell $\phi$ in $\Mat\V$, of the form \eqref{cell}, is given by a family of $\V$-maps $\map{\phi_{x,y}}{J(x,y)}{K(fx,gy)}$, while the horizontal composite $J \hc H$ of $\hmap JAB$ and $\hmap HBE$ is given by `matrix multiplication':
		\begin{displaymath}
			(J \hc H)(x, z) = \coprod_{y \in B} J(x, y) \tens H(y, z).
		\end{displaymath}
		The unit matrix $\hmap{1_A}AA$ is given by $1_A(x,x) = 1$ and $1_A(x_1, x_2) = \emptyset$, the initial object of $\V$, whenever $x_1 \neq x_2$.
	\end{example}
	
	\subsection{Equipments.}
	Important to the theory of double categories are the notions of cartesian and opcartesian cells. A cartesian cell defines the restriction of a horizontal morphism along a pair of vertical morphisms and, dually, opcartesian cells define extensions, as follows.
	\begin{definition} \label{definition: cartesian and opcartesian cells}
		The cell $\phi$ on the left below is called \emph{cartesian} if any cell $\psi$, as in the middle, factors uniquely through $\phi$ as shown.
		\begin{displaymath}
  		\begin{tikzpicture}[textbaseline]
    		\matrix(m)[math175em]{A & B \\ C & D \\};
    		\path[map]  (m-1-1) edge[barred] node[above] {$J$} (m-1-2)
                        		edge node[left] {$f$} (m-2-1)
		                (m-1-2) edge node[right] {$g$} (m-2-2)
    		            (m-2-1) edge[barred] node[below] {$K$} (m-2-2);
    		\path[transform canvas={shift={($(m-1-1)!0.5!(m-2-1)$)}}] (m-1-2) edge[cell] node[right] {$\phi$} (m-2-2);
  		\end{tikzpicture} \quad\quad \begin{tikzpicture}[textbaseline]
    		\matrix(m)[math175em]{X & Y \\ A & B \\ C & D \\};
    		\path[map]  (m-1-1) edge[barred] node[above] {$H$} (m-1-2)
        		                edge node[left] {$h$} (m-2-1)
        		        (m-1-2) edge node[right] {$k$} (m-2-2)
        		        (m-2-1) edge node[left] {$f$} (m-3-1)
        		        (m-2-2) edge node[right] {$g$} (m-3-2)
        		        (m-3-1) edge[barred] node[below] {$K$} (m-3-2);
    		\path[transform canvas={shift={($(m-2-1)!0.5!(m-1-1)$)}}] (m-2-2) edge[cell] node[right] {$\psi$} (m-3-2);
  		\end{tikzpicture} = \begin{tikzpicture}[textbaseline]
    		\matrix(m)[math175em]{X & Y \\ A & B \\ C & D \\};
    		\path[map]  (m-1-1) edge[barred] node[above] {$H$} (m-1-2)
        		                edge node[left] {$h$} (m-2-1)
            		    (m-1-2) edge node[right] {$k$} (m-2-2)
            		    (m-2-1) edge[barred] node[below] {$J$} (m-2-2)
            		            edge node[left] {$f$} (m-3-1)
            		    (m-2-2) edge node[right] {$g$} (m-3-2)
            		    (m-3-1) edge[barred] node[below] {$K$} (m-3-2);
    		\path[transform canvas={shift=(m-2-1))}]
        		        (m-1-2) edge[cell] node[right] {$\psi'$} (m-2-2)
        		        (m-2-2) edge[transform canvas={yshift=-0.3em}, cell] node[right] {$\phi$} (m-3-2);
  		\end{tikzpicture} \quad\quad \begin{tikzpicture}[textbaseline]
    		\matrix(m)[math175em]{A & B \\ C & D \\ X & Y \\};
    		\path[map]  (m-1-1) edge[barred] node[above] {$J$} (m-1-2)
        		                edge node[left] {$f$} (m-2-1)
        		        (m-1-2) edge node[right] {$g$} (m-2-2)
        		        (m-2-1) edge node[left] {$h$} (m-3-1)
        		        (m-2-2) edge node[right] {$k$} (m-3-2)
        		        (m-3-1) edge[barred] node[below] {$L$} (m-3-2);
    		\path[transform canvas={shift={($(m-2-1)!0.5!(m-1-1)$)}}] (m-2-2) edge[cell] node[right] {$\chi$} (m-3-2);
  		\end{tikzpicture} = \begin{tikzpicture}[textbaseline]
    		\matrix(m)[math175em]{A & B \\ C & D \\ X & Y \\};
    		\path[map]  (m-1-1) edge[barred] node[above] {$J$} (m-1-2)
        		                edge node[left] {$f$} (m-2-1)
            		    (m-1-2) edge node[right] {$g$} (m-2-2)
            		    (m-2-1) edge[barred] node[below] {$K$} (m-2-2)
            		            edge node[left] {$h$} (m-3-1)
            		    (m-2-2) edge node[right] {$k$} (m-3-2)
            		    (m-3-1) edge[barred] node[below] {$L$} (m-3-2);
    		\path[transform canvas={shift=(m-2-1))}]
        		        (m-1-2) edge[cell] node[right] {$\phi$} (m-2-2)
        		        (m-2-2) edge[transform canvas={yshift=-0.3em}, cell] node[right] {$\chi'$} (m-3-2);
  		\end{tikzpicture}
		\end{displaymath}
		Vertically dual, the cell $\phi$ is called \emph{opcartesian} if any cell $\chi$ as on the right factors uniquely through $\phi$ as shown.
	\end{definition}
	If a cartesian cell like $\phi$ exists then we call $J$ the \emph{restriction of $K$ along $f$ and $g$}, and write $K(f, g) = J$; if $K = 1_C$ then we write $C(f, g) = 1_C(f,g)$. By their universal property any two cartesian cells defining the same restriction factor through each other as invertible horizontal cells. Moreover, since the vertical composite of two cartesian cells is again cartesian, and since vertical identities $\id_K$ are cartesian, it follows that restrictions are pseudofunctorial, in the sense that $K(f, g)(h, k) \iso K(f \of h, g \of k)$ and $K(\id, \id) \iso K$. Dually, if an opcartesian cell like $\phi$ exists then we call $K$ the \emph{extension of $J$ along $f$ and $g$}; like restrictions, extensions are unique up to isomorphism and pseudofunctorial.  We shall usually not name cartesian and opcartesian cells, but simply depict them like the two cells below. 
	
	For each vertical morphism $\map fAC$ the restriction $f_* = \hmap{C(f, \id)}AC$, if it exists, is called the \emph{companion} of $f$; it is defined by a cartesian cell as on left below. Dually the extension of $1_A$ along $f$ and $\id_A$, if it exists, is called the \emph{conjoint} of $f$; it is denoted by $f^*$ and defined by an opcartesian cell as on the right.
	\begin{displaymath}
		\begin{tikzpicture}[textbaseline]
    		\matrix(m)[math175em]{A & C \\ C & C \\};
    		\path[map]  (m-1-1) edge[barred] node[above] {$f_*$} (m-1-2)
                        		edge node[left] {$f$} (m-2-1);
		    \path				(m-1-2) edge[eq] (m-2-2)
    		            (m-2-1) edge[eq] (m-2-2);
    		\draw ($(m-1-1)!0.5!(m-2-2)$) node {cart};
  		\end{tikzpicture} \qquad\qquad\qquad\qquad\qquad \begin{tikzpicture}[textbaseline]
    		\matrix(m)[math175em]{A & A \\ C & A \\};
    		\path				(m-1-1) edge[eq] (m-1-2)
                    (m-1-2) edge[eq] (m-2-2);
        \path[map]	(m-2-1) edge[barred] node[below] {$f^*$} (m-2-2)
    		            (m-1-1) edge node[left] {$f$} (m-2-1);
    		\draw ($(m-1-1)!0.5!(m-2-2)$) node {opcart};
  		\end{tikzpicture}
	\end{displaymath}
	\begin{example} \label{cartesian and opcartesian cells for profunctors}
		In \exref{example: profunctors} we have already used the notation $K(f, g)$ to denote the profunctor $\hmap{K \of (\op f \times g)}AB$. It is readily seen that the cell $\cell\eps{K(f,g)}K$ given by the identity transformation on $K\of (\op f \times g)$ is cartesian in the double category $\Prof$, so that $K \of (\op f \times g)$ is indeed the restriction of $K$ along $f$ and $g$.
		
		We can take the conjoint $f^*$ of a functor $\map fAC$ to be the restriction $C(\id, f)$: the opcartesian cell defining it is given by the actions \mbox{$\map f{A(x_1, x_2)}{C(fx_1, fx_2)}$} of $f$ on the hom-sets. That $f^* \iso C(\id, f)$ is no coincidence is explained below.
	\end{example}
	
	Generalising the situation above, in every double category the conjoint $f^*$ of a vertical morphism $\map fAC$ can be equivalently defined as the restriction $C(\id, f)$. This is because the vertical identity cell $1_f$ factors uniquely through the opcartesian cell defining $f^*$ as a cartesian cell that defines $C(\id, f)$, and conversely. Horizontally dual, the same relation exists between the companion $f_*$ and the extension of $1_A$ along $\id_A$ and $f$.
	
	Thus companions and conjoints are defined by cartesian, or equivalently opcartesian, cells. Conversely the existence of all companions and conjoints implies the existence of all restrictions and extensions: for any $\hmap KCD$ the composite on the left below is cartesian while, for any $\hmap JBA$, the composite on the right is opcartesian. For details see Theorem 4.1 of \cite{Shulman08} or Lemma 2.8 of \cite{Koudenburg14a}.
	\begin{equation} \label{cartesian and opcartesian cells in terms of companions and conjoints}
		\begin{tikzpicture}[textbaseline]
			\matrix(m)[math175em]{A & C & D & B \\ C & C & D & D \\};
			\path[map]	(m-1-1) edge[barred] node[above] {$f_*$} (m-1-2)
													edge node[left] {$f$} (m-2-1)
									(m-1-2) edge[barred] node[above] {$K$} (m-1-3)
									(m-1-3) edge[barred] node[above] {$g^*$} (m-1-4)
									(m-1-4) edge node[right] {$g$} (m-2-4)
									(m-2-2) edge[barred] node[below] {$K$} (m-2-3);
			\path				(m-1-2) edge[eq] (m-2-2)
									(m-1-3) edge[eq] (m-2-3)
									(m-2-1) edge[eq] (m-2-2)
									(m-2-3) edge[eq] (m-2-4);
			\draw				($(m-1-1)!0.5!(m-2-2)$) node {cart}
									($(m-1-3)!0.5!(m-2-4)$) node {cart};
		\end{tikzpicture} \qquad\qquad \begin{tikzpicture}[textbaseline]
			\matrix(m)[math175em]{B & B & A & A \\ D & B & A & C \\};
			\path[map]	(m-1-1) edge node[left] {$g$} (m-2-1)
									(m-1-2) edge[barred] node[above] {$J$} (m-1-3)
									(m-2-3) edge[barred] node[below] {$f_*$} (m-2-4)
									(m-1-4) edge node[right] {$f$} (m-2-4)
									(m-2-1) edge[barred] node[below] {$g^*$} (m-2-2)
									(m-2-2) edge[barred] node[below] {$J$} (m-2-3);
			\path				(m-1-2) edge[eq] (m-2-2)
									(m-1-3) edge[eq] (m-2-3)
									(m-1-1) edge[eq] (m-1-2)
									(m-1-3) edge[eq] (m-1-4);
			\draw				($(m-1-1)!0.5!(m-2-2)$) node {opcart}
									($(m-1-3)!0.5!(m-2-4)$) node {opcart};
		\end{tikzpicture}
	\end{equation}
	In summary, the following conditions on a double category $\K$ are equivalent: $\K$ has all companions and conjoints; $\K$ has all restrictions; $\K$ has all extensions.
	
	\begin{definition}
		An \emph{equipment} is a double category that satisfies the conditions above. 
	\end{definition}

	\begin{example}
		Since the double category $\Prof$ has all restrictions, as we saw in \exref{cartesian and opcartesian cells for profunctors}, it is an equipment. The double categories $\Span\E$ (\exref{example: spans}) and $\Mat\V$ (\exref{example: matrices}) are equipments as well. Indeed the extension of a span $A \xlar{d_0} J \xrar{d_1} B$ in $\E$, along morphisms $\map fAC$ and $\map gBD$, is given by the span $C \xlar{f \of d_0} J \xrar{g \of d_1} D$, while the restriction $K(f, g)$ of a $\V$-matrix $\hmap KCD$ is given by the family of $\V$-objects $K(f,g)(x,y) = K(fx, gy)$.
	\end{example}
	
	\subsection{Monoids and bimodules.}
	Next we recall the notions of monoid and bimodule in double categories, following Section 11 of \cite{Shulman08}. These notions are useful: for example, as we recall below, monoids and bimodules in $\Span \E$ are categories and profunctors internal in $\E$ while monoids and bimodules in $\Mat\V$ are $\V$-enriched categories and $\V$\ndash profunctors.
	\begin{definition}
		Let $\K$ be a double category.
		\begin{enumerate}[label=-]
			\item	A \emph{monoid} in $\K$ consists of a quadruple $A = (A_0, A, \mu, \eta)$ where $\hmap A{A_0}{A_0}$ is a horizontal morphism in $\K$ and $\cell \mu{A \hc A}A$ and $\cell \eta{1_{A_0}}A$ are horizontal cells satisfying the usual coherence axioms for monoids.
			\item Given monoids $A$ and $C$, a \emph{morphism of monoids} $\map fAC$ consists of a vertical morphism $\map{f_0}{A_0}{C_0}$ and a cell $f$, as on the left below, such that $\mu_C \of (f \hc f) = f \of \mu_A$ and $f \of \eta_A = \eta_C \of 1_{f_0}$.
		\end{enumerate}
		\begin{displaymath}
			\begin{tikzpicture}[baseline]
				\matrix(m)[math175em]{A_0 & A_0 \\ C_0 & C_0 \\};
				\path[map]	(m-1-1) edge[barred] node[above] {$A$} (m-1-2)
														edge node[left] {$f_0$} (m-2-1)
										(m-1-2) edge node[right] {$f_0$}  (m-2-2)
										(m-2-1) edge[barred] node[below] {$C$} (m-2-2);
				\path[transform canvas={shift=($(m-1-1)!0.5!(m-2-1)$)}]	(m-1-2) edge[cell] node[right] {$f$} (m-2-2);
			\end{tikzpicture} \qquad\qquad \begin{tikzpicture}[baseline]
				\matrix(m)[math175em]{A & B \\ C & D \\};
				\path[map]	(m-1-1) edge[barred] node[above] {$J$} (m-1-2)
														edge node[left] {$f$} (m-2-1)
										(m-1-2) edge node[right] {$g$}  (m-2-2)
										(m-2-1) edge[barred] node[below] {$K$} (m-2-2);
				\path[transform canvas={shift=($(m-1-1)!0.5!(m-2-1)$)}]	(m-1-2) edge[cell] node[right] {$\phi$} (m-2-2);
			\end{tikzpicture} \qquad\qquad\begin{tikzpicture}[baseline]
				\matrix(m)[math175em]{A_0 & B_0 \\ C_0 & D_0 \\};
				\path[map]	(m-1-1) edge[barred] node[above] {$J$} (m-1-2)
														edge node[left] {$f_0$} (m-2-1)
										(m-1-2) edge node[right] {$g_0$}  (m-2-2)
										(m-2-1) edge[barred] node[below] {$K$} (m-2-2);
				\path[transform canvas={shift=($(m-1-1)!0.5!(m-2-1)$)}]	(m-1-2) edge[cell] node[right] {$\phi$} (m-2-2);
			\end{tikzpicture}
		\end{displaymath}
		\begin{enumerate}[label=-]
			\item	Given monoids $A$ and $B$, an \emph{$(A,B)$-bimodule} $\hmap JAB$ consists of a horizontal morphism $\hmap J{A_0}{B_0}$ that is equipped with horziontal cells $\cell \lambda{A \hc J}J$ and $\cell \rho{J \hc B}J$ defining the \emph{actions} of $A$ and $B$ on $J$, which satisfy the usual coherence axioms for bimodules.
			\item Given morphisms of monoids $\map fAC$ and $\map gBD$, and bimodules $\hmap JAB$ and $\hmap KCD$, a \emph{cell} $\phi$ as in the middle above is a cell in $\K$ as on the right, such that $\lambda_K \of (f \hc \phi) = \phi \of \lambda_J$ and $\rho_K \of (\phi \hc g) = \phi \of \rho_J$.
			\item The \emph{horizontal composite} $J \hc_B H$ of bimodules $\hmap JAB$ and $\hmap HBE$ is the following reflexive coequaliser in $H(\K)(A,E)$, if it exists:
			\begin{displaymath}
				\begin{tikzpicture}
					\matrix(m)[math2em]{J \hc B \hc H &[1.25em] J \hc H & J \hc_B H. \\};
					\path[map]	(m-1-1) edge[transform canvas={yshift=4pt}, cell] node[above] {$\rho_J \hc \id$} (m-1-2)
															edge[transform canvas={yshift=-4pt}, cell] node[below] {$\id \hc \lambda_H$} (m-1-2)
											(m-1-2) edge[cell] (m-1-3);
				\end{tikzpicture}
			\end{displaymath}
		\end{enumerate}
	\end{definition}
	The following is Proposition 11.10 of \cite{Shulman08}. Recall that each double category $\K$ contains a bicategory $H(\K)$ consisting of horizontal morphisms and horizontal cells. We say that $\K$ has \emph{local reflexive coequalisers} if the categories $H(\K)(A, B)$ have reflexive coequalisers that are preserved by horizontal composition on both sides.
	\begin{proposition}[Shulman] \label{bimodule double categories}
		If an equipment $\K$ has local reflexive coequalisers then monoids and bimodules in $\K$, together with their morphisms and cells, again form an equipment $\Mod\K$ that has local reflexive coequalisers, whose horizontal composition is given as above.
	\end{proposition}
	
	\begin{example} \label{example: internal profunctors}
		Let $\E$ be a category with pullbacks. Monoids in $\Span \E$ are \emph{internal categories} in $\E$ and morphisms of such monoids are internal functors. Bimodules and their cells in $\Span \E$ are \emph{internal profunctors} in $\E$ and their transformations. Internal categories and internal profunctors are described in some detail in Example~1.6 of \cite{Koudenburg14a}. If $\E$ has reflexive coequalisers preserved by pullback then $\Span\E$ has local reflexive coequalisers, so that internal categories, functors, profunctors and transformations in $\E$ form an equipment $\Mod{\Span \E}$ which is denoted $\inProf\E$. 
	\end{example}
	
	For a suitable monoidal category $\V$ the equipment of bimodules in $\Mat\V$ (\exref{example: matrices}) is that of $\V$-enriched profunctors, which we shall now describe; it will be used as an example throughout. We remark that, in the case that $\V$ is closed symmetric monoidal---hence enriched over itself---, the classical meaning of a $\V$\ndash profunctor $A \brar B$, where $A$ and $B$ are $\V$-categories, is that of a $\V$-functor of the form $\op A \tens B \to \V$. The advantage of defining $\V$-profunctors as bimodules in $\Mat\V$ is that, while this extends the classical definition, in this way it is not necessary for $\V$ to be closed symmetric monoidal---$\V$ being a monoidal category suffices. This allows us, in the next subsection, to extend the classical notions of `weighted limit' and `enriched Kan extension' to settings in which the enriching category $\V$ is not closed symmetric monoidal.
	\begin{example} \label{example: enriched profunctors}
		Let $\V = (\V, \tens, 1)$ be a monoidal category that has coproducts which are preserved by $\tens$ on both sides. A monoid in $\Mat\V$ is a category \emph{enriched} in $\V$, while a morphism of such monoids is a $\V$-functor, both in the usual sense; see e.g.\ Section 1.2 of \cite{Kelly82}. A bimodule $\hmap JAB$ in $\Mat\V$ is a \emph{$\V$\ndash profunctor} in the sense of Section 7 of \cite{Day-Street97}: it consists of a family $J(x,y)$ of $\V$-objects, indexed by pairs of objects $x \in A$ and $y \in B$, that is equipped with associative and unital actions
		\begin{displaymath}
			\map \lambda{A(x_1, x_2) \tens J(x_2, y)}{J(x_1, y)} \quad \text{and} \quad \map \rho{J(x, y_1) \tens B(y_1, y_2)}{J(x, y_2)}
		\end{displaymath}
		satisfying the usual compatibility axiom for bimodules. Given a map\footnote{As is the custom, by a map $\map f{x_1}{x_2}$ in a $\V$-category $A$ we mean a $\V$-map $\map f1{A(x_1, x_2)}$.} $\map f{x_1}{x_2}$ in $A$ we write $\lambda_f$ for the composite
		\begin{displaymath}
			J(x_2, y) \xrar{f \tens \id} A(x_1, x_2) \tens J(x_2, y) \xrar\lambda J(x_1, y);
		\end{displaymath}
		likewise for $\map g{y_1}{y_2}$ in $B$ we write $\map{\rho_g = \rho \of (\id \tens g)}{J(x, y_1)}{J(x, y_2)}$.
		
		If $\V$ is closed symmetric monoidal then $\V$\ndash pro\-func\-tors $\hmap JAB$ can be identified with $\V$-functors of the form $\map J{\op A \tens B}\V$. Indeed the actions of $J$ correspond, under the adjunctions $\dash \tens J(x,y) \ladj \inhom{J(x, y), \dash}$ that are part of the closed structure on $\V$, to families of maps
		\begin{displaymath}
			A(x_1, x_2) \to \inhom{J(x_2, y), J(x_1, y)} \qquad \text{and} \qquad B(y_1, y_2) \to \inhom{J(x,y_1), J(x,y_2)}
		\end{displaymath}
		respectively, which combine to form families of partial $\V$-functors $\map{J(\dash, y)}{\op A}\V$ and $\map{J(x, \dash)}B\V$. The compatibility axiom for bimodules ensures that the latter correspond to a single $\V$-functor $\map J{\op A \tens B}\V$; for details see Section 1.4 of \cite{Kelly82}.
		
		A cell of $\V$-profunctors
		\begin{displaymath}
			\begin{tikzpicture}
				\matrix(m)[math175em]{A & B \\ C & D, \\};
				\path[map]  (m-1-1) edge[barred] node[above] {$J$} (m-1-2)
														edge node[left] {$f$} (m-2-1)
										(m-1-2) edge node[right] {$g$} (m-2-2)
										(m-2-1) edge[barred] node[below] {$K$} (m-2-2);
				\path[transform canvas={shift={($(m-1-2)!(0,0)!(m-2-2)$)}}] (m-1-1) edge[cell] node[right] {$\phi$} (m-2-1);
			\end{tikzpicture}
		\end{displaymath}
		called a \emph{transformation}, consists of a family of $\V$-maps $\map{\phi_{(x, y)}}{J(x,y)}{K(fx, gy)}$ (often simply denoted $\phi$) that are compatible with the actions, in the sense that the identities
		\begin{flalign*}
			&& \map{\lambda \of (f \tens \phi) &= \phi \of \lambda}{A(x_1, x_2) \tens J(x_2, y)}{K(fx_1, gy)}& \\
			\text{and} && \map{\rho \of (\phi \tens g) &= \phi \of \rho}{J(x, y_1) \tens B(y_1, y_2)}{K(fx, gy_2)} &
		\end{flalign*}
		are satisfied, where $f$ and $g$ denote the actions of the $\V$-functors $f$ and $g$ on the hom-objects of $A$ and $B$ respectively. If $\V$ is closed symmetric monoidal then transformations in the above sense can be identified with the usual $\V$-natural transformations between the $\V$-functors $\map J{\op A \tens B}\V$ and $K(f, g) = K \of (\op f \tens g)$.
	
		If $\V$ has reflexive coequalisers preserved by $\tens$ on both sides then $\Mat\V$ has local reflexive coequalisers, so that $\V$-categories, $\V$-functors, $\V$\ndash profunctors and their transformations form a double category $\Mod{\Mat\V}$ which is denoted $\enProf\V$. Its horizontal composite $J \hc H$, of $\V$-profunctors $\hmap JAB$ and \mbox{$\hmap HBE$}, is obtained by choosing coequalisers $(J \hc H)(x, z)$ for the pairs of $\V$-maps
		\begin{equation} \label{composite of profunctors}
			\coprod\limits_{y_1, y_2 \in B} J(x, y_1) \tens B(y_1, y_2) \tens H(y_2, z) \rightrightarrows \coprod\limits_{y \in B} J(x, y) \tens H(y, z),
		\end{equation}
		that are induced by letting $B(y_1, y_2)$ act on $J(x, y_1)$ and $H(y_2, z)$ respectively; compare the unenriched situation \eqref{composite of unenriched profunctors}. The horizontal composite $\phi \hc \chi$ of the transformations on the left below is given by the family of unique factorisations shown on the right, where the $\V$\ndash maps drawn horizontally are the coequalisers defining $(J \hc H)(x, z)$ and $(K \hc L)(fx, hz)$, and the $\V$-map on the left is induced by the tensorproducts $\phi_{(x,y)} \tens \chi_{(y,z)}$.
		\begin{displaymath}
	  	\begin{tikzpicture}[baseline]
				\matrix(m)[math175em]{A & B & E \\ C & D & F \\};
				\path[map]	(m-1-1) edge[barred] node[above] {$J$} (m-1-2)
														edge node[left] {$f$} (m-2-1)
										(m-1-2) edge[barred] node[above] {$H$} (m-1-3)
														edge node[right] {$g$} (m-2-2)
										(m-1-3) edge node[right] {$h$} (m-2-3)
										(m-2-1) edge[barred] node[below] {$K$} (m-2-2)
										(m-2-2) edge[barred] node[below] {$L$} (m-2-3);
				\path[transform canvas={shift=($(m-1-1)!0.5!(m-2-2)$)}]
										(m-1-2) edge[cell] node[right] {$\phi$} (m-2-2)
										(m-1-3) edge[cell] node[right] {$\chi$} (m-2-3);
			\end{tikzpicture} \qquad\qquad \begin{tikzpicture}[baseline]
				\matrix(m)[math175em]{\coprod\limits_{y \in B} J(x, y) \tens H(y, z) & (J \hc H)(x, z) \\
				\coprod\limits_{v \in D} K(fx, v) \tens L(v, hz) & (K \hc L)(fx, hz) \\};
				\path[map]	(m-1-1) edge (m-1-2)
														edge (m-2-1)
										(m-1-2) edge[dashed] node[right] {$(\phi \hc \chi)_{(x,z)}$} (m-2-2)
										(m-2-1) edge (m-2-2);
			\end{tikzpicture}
	  \end{displaymath}
	  The unit $\V$-profunctor $\hmap{1_A}AA$, for a $\V$-category $A$, is given by the hom-objects $1_A(x_1, x_2) = A(x_1, x_2)$; its actions are given by the composition of $A$. Finally, $\enProf\V$ is an equipment in which the restriction $K(f,g)$ of a $\V$\ndash profunctor $\hmap KCD$ along $\V$-functors $\map fAC$ and $\map gBD$ is given by the family of $\V$-objects $K(f,g)(x,y) = K(fx,gy)$, that is equipped with actions induced by those of $K$.
	
		We remark that, for $\V$-functors $f$ and $\map gAC$, the vertical cells $f \Rar g$ in $\enProf\V$ can be identified with $\V$-natural transformations $f \Rar g$, in the classical sense; see Section 1.2 of \cite{Kelly82}. Indeed any vertical cell $\nat\phi fg$ is given by a family of $\V$-maps \mbox{$\map{\phi_{(x_1, x_2)}}{A(x_1, x_2)}{C(fx_1, gx_2)}$} which is, because of its compatibility with the actions of $A$ and $C$, completely determined by the composites
		\begin{displaymath}
			\phi_x = \bigbrks{1 \xrar{\eta_x} A(x, x) \xrar{\phi_{(x,x)}} C(fx, gx)},
		\end{displaymath}
		which make the diagram below commute; that is, they form a $\V$-natural transformation of $\V$-functors $f \Rar g$.
		\begin{displaymath}
			\begin{tikzpicture}
				\matrix(m)[math175em]{ A(x_1, x_2) &[-2em] C(fx_1, fx_2) \tens C(fx_2, gx_2) \\
					C(fx_1, gx_1) \tens C(gx_1, gx_2) & C(fx_1, gx_2) \\ };
				\path[map]	(m-1-1) edge node[above] {$f \tens \phi_{x_2}$} (m-1-2)
														edge node[left] {$\phi_{x_1} \tens g$} (m-2-1)
										(m-1-2) edge node[right] {$\mu$} (m-2-2)
										(m-2-1) edge node[below] {$\mu$} (m-2-2);
			\end{tikzpicture}
		\end{displaymath}
		This identification induces an isomorphism
		\begin{equation} \label{vertical 2-category of V-Prof}
			V(\enProf\V) \iso \enCat\V
		\end{equation}
		between the vertical $2$-category contained in $\enProf\V$ and the $2$-category $\enCat\V$ of $\V$\ndash cat\-e\-go\-ries, $\V$-functors and $\V$-natural transformations.
	\end{example}
	
	\subsection{Kan extensions in double categories.}
	Analogous to that in $2$-categories there is a notion of Kan extension in double categories, that was introduced in \cite{Grandis-Pare08} and which is recalled below. The stronger notion of pointwise Kan extension, which we also consider, was introduced in \cite{Koudenburg14a}. To get some feeling for the latter we shall consider pointwise Kan extensions in the double category $\enProf\V$ in the next subsection.
	\begin{definition}\label{definition: right Kan extension}
		Let $\map dBM$ and $\hmap JAB$ be morphisms in a double category. The cell $\eps$ in the right-hand side below is said to define $r$ as the \emph{right Kan extension of $d$ along $J$} if every cell $\phi$ below factors uniquely through $\eps$ as shown.
		\begin{displaymath}
			\begin{tikzpicture}[textbaseline]
  			\matrix(m)[math175em]{A & B \\ M & M \\};
  			\path[map]  (m-1-1) edge[barred] node[above] {$J$} (m-1-2)
														edge node[left] {$s$} (m-2-1)
										(m-1-2) edge node[right] {$d$} (m-2-2);
				\path				(m-2-1) edge[eq] (m-2-2);
				\path[transform canvas={shift={($(m-1-2)!(0,0)!(m-2-2)$)}}] (m-1-1) edge[cell] node[right] {$\phi$} (m-2-1);
			\end{tikzpicture} = \begin{tikzpicture}[textbaseline]
				\matrix(m)[math175em]{A & A & B \\ M & M & M \\};
				\path[map]	(m-1-1) edge node[left] {$s$} (m-2-1)
										(m-1-2) edge[barred] node[above] {$J$} (m-1-3)
														edge node[right] {$r$} (m-2-2)
										(m-1-3) edge node[right] {$d$} (m-2-3);
				\path				(m-1-1) edge[eq] (m-1-2)
										(m-2-1) edge[eq] (m-2-2)
										(m-2-2) edge[eq] (m-2-3);
				\path[transform canvas={shift=($(m-1-1)!0.5!(m-2-2)$)}]
										(m-1-2) edge[cell] node[right] {$\phi'$} (m-2-2)
										(m-1-3) edge[cell] node[right] {$\eps$} (m-2-3);
			\end{tikzpicture}
		\end{displaymath}
		
		We call the right Kan extension $r$ above \emph{pointwise} if cells $\phi$ of the more general form below also factor uniquely through $\eps$, as shown.
		\begin{displaymath}
			\begin{tikzpicture}[textbaseline]
				\matrix(m)[math175em]{C & A & B \\ M & & M \\};
				\path[map]	(m-1-1) edge[barred] node[above] {$H$} (m-1-2)
														edge node[left] {$s$} (m-2-1)
										(m-1-2) edge[barred] node[above] {$J$} (m-1-3)
										(m-1-3) edge node[right] {$d$} (m-2-3)
										(m-1-2) edge[cell] node[right] {$\phi$} (m-2-2);
				\path				(m-2-1) edge[eq] (m-2-3);
			\end{tikzpicture} = \begin{tikzpicture}[textbaseline]
				\matrix(m)[math175em]{C & A & B \\ M & M & M \\};
				\path[map]	(m-1-1) edge[barred] node[above] {$H$} (m-1-2)
														edge node[left] {$s$} (m-2-1)
										(m-1-2) edge[barred] node[above] {$J$} (m-1-3)
														edge node[right] {$r$} (m-2-2)
										(m-1-3) edge node[right] {$d$} (m-2-3);
				\path				(m-2-1) edge[eq] (m-2-2)
										(m-2-2) edge[eq] (m-2-3);
				\path[transform canvas={shift={($(m-1-2)!0.5!(m-2-3)$)}}] (m-1-1) edge[cell] node[right] {$\phi'$} (m-2-1)
										(m-1-2) edge[cell] node[right] {$\eps$} (m-2-2);
			\end{tikzpicture}
  	\end{displaymath}
	\end{definition}
		
	The following, which is a simple consequence of the universal property of opcartesian cells, shows that the notion of Kan extension in double categories generalises that of Kan extension in $2$-categories (for a definition see e.g.\ Section 3 of \cite{Koudenburg14a}). Remember that any double category $\K$ contains a $2$-category $V(\K)$ of objects, vertical morphisms and vertical cells of $\K$.
	\begin{proposition} \label{right Kan extensions along conjoints as right Kan extensions in V(K)}
		Let $\map dBM$, $\map jBA$ and $\map rAM$ be morphisms in a double category $\K$, and assume that the conjoint $\hmap{j^*}AB$ of $j$ exists. Consider a vertical cell $\eps$ as on the left below, as well as its factorisation through the opcartesian cell defining $j^*$, as shown.
		\begin{displaymath}
			\begin{tikzpicture}[textbaseline]
    		\matrix(m)[math175em]{B & B \\ A & \phantom A \\ M & M \\};
    		\path[map]  (m-1-1) edge node[left] {$j$} (m-2-1)
       			        (m-1-2) edge node[right] {$d$} (m-3-2)
       			        (m-2-1) edge node[left] {$r$} (m-3-1);
      	\path				(m-1-1) edge[eq] (m-1-2)
      							(m-3-1) edge[eq] (m-3-2);
    		\path[transform canvas={shift={($(m-2-1)!0.5!(m-1-1)$)}}] (m-2-2) edge[cell] node[right] {$\eps$} (m-3-2);
  		\end{tikzpicture} = \begin{tikzpicture}[textbaseline]
				\matrix(m)[math175em]{B & B \\ A & B \\ M & M \\};
				\path[map]	(m-1-1) edge node[left] {$j$} (m-2-1)
										(m-2-1) edge[barred] node[below] {$j^*$} (m-2-2)
														edge node[left] {$r$} (m-3-1)
										(m-2-2) edge node[right] {$d$} (m-3-2);
				\path				(m-1-1) edge[eq] (m-1-2)
										(m-1-2) edge[eq] (m-2-2)
										(m-3-1) edge[eq] (m-3-2);
				\path[transform canvas={shift=(m-2-1), yshift=-2pt}]	(m-2-2) edge[cell] node[right] {$\eps'$} (m-3-2);
				\draw				($(m-1-1)!0.5!(m-2-2)$)	node {\textup{opcart}};
			\end{tikzpicture}
		\end{displaymath}
		The vertical cell $\eps$ defines $r$ as the right Kan extension of $d$ along $j$ in $V(\K)$ precisely if its factorisation $\eps'$ defines $r$ as the right Kan extension of $d$ along $j^*$ in $\K$.
	\end{proposition}
	
	The following result for iterated pointwise Kan extensions is a straightforward consequence of \defref{definition: right Kan extension}.
	\begin{proposition} \label{iterated pointwise Kan extensions}
		Consider horizontally composable cells
		\begin{displaymath}
			\begin{tikzpicture}
				\matrix(m)[math175em]{A & B & C \\ M & M & M \\};
				\path[map]	(m-1-1) edge[barred] node[above] {$J$} (m-1-2)
														edge node[left] {$s$} (m-2-1)
										(m-1-2) edge[barred] node[above] {$H$} (m-1-3)
														edge node[right] {$r$} (m-2-2)
										(m-1-3) edge node[right] {$d$} (m-2-3);
				\path	(m-2-1) edge[eq] (m-2-2)
							(m-2-2) edge[eq] (m-2-3);
				\path[transform canvas={shift={($(m-1-2)!0.5!(m-2-3)$)}}] (m-1-1) edge[cell] node[right] {$\gamma$} (m-2-1)
				(m-1-2) edge[cell] node[right] {$\eps$} (m-2-2);
			\end{tikzpicture}
		\end{displaymath}
		in a double category, and suppose that $\eps$ defines $r$ as the pointwise right Kan extension of $d$ along $H$. Then $\gamma$ defines $s$ as the pointwise right Kan extension of $r$ along $J$ precisely if $\gamma \hc \eps$ defines $s$ as the pointwise right Kan extension of $d$ along $J \hc H$.
	\end{proposition}
	
	For completeness we record the definition of pointwise left Kan extension in a double category $\K$, which coincides with that of pointwise right Kan extension in its horizontal dual $\co\K$.
	\begin{definition} \label{definition: pointwise left Kan extension}
  	Let $\map dAM$ and $\hmap JAB$ be morphisms in a double category. The cell $\eta$ in the right-hand side below defines $l$ as the \emph{pointwise left Kan extension of $d$ along $J$} if every cell $\phi$ below factors uniquely through $\eta$ as shown.
  	\begin{displaymath}
  		\begin{tikzpicture}[textbaseline]
				\matrix(m)[math175em]{A & B & C \\ M & & M \\};
				\path[map]	(m-1-1) edge[barred] node[above] {$J$} (m-1-2)
														edge node[left] {$d$} (m-2-1)
										(m-1-2) edge[barred] node[above] {$H$} (m-1-3)
										(m-1-3) edge node[right] {$k$} (m-2-3)
										(m-1-2) edge[cell] node[right] {$\phi$} (m-2-2);
				\path				(m-2-1) edge[eq] (m-2-3);
			\end{tikzpicture} = \begin{tikzpicture}[textbaseline]
				\matrix(m)[math175em]{A & B & C \\ M & M & M \\};
				\path[map]	(m-1-1) edge[barred] node[above] {$J$} (m-1-2)
														edge node[left] {$d$} (m-2-1)
										(m-1-2) edge[barred] node[above] {$H$} (m-1-3)
														edge node[right] {$l$} (m-2-2)
										(m-1-3) edge node[right] {$k$} (m-2-3);
				\path				(m-2-1) edge[eq] (m-2-2)
										(m-2-2) edge[eq] (m-2-3);
				\path[transform canvas={shift={($(m-1-2)!0.5!(m-2-3)$)}}] (m-1-1) edge[cell] node[right] {$\eta$} (m-2-1)
										(m-1-2) edge[cell] node[right] {$\phi'$} (m-2-2);
			\end{tikzpicture}
  	\end{displaymath}
  \end{definition}
	
	\subsection{Pointwise Kan extensions in \texorpdfstring{$\enProf \V$}{V-Prof}.}
	Here we study pointwise Kan extensions in the double category $\enProf \V$ and show that they can be described in terms of `$\V$-weighted limits'. While such limits are classically defined in the case that $\V$ is a closed symmetric monoidal category, we shall consider a more general definition for which a monoidal structure on $\V$ alone suffices; this enables us to treat all cases of $\enProf\V$.
	
	For the rest of this subsection we take $\V$ to be a cocomplete monoidal category whose tensor product preserves colimits on both sides, so that $\V$-categories, $\V$\ndash functors, $\V$\ndash profunctors and their transformations form a double category $\enProf\V$; see \exref{example: enriched profunctors}. We write $1$ for the unit $\V$-category, that has a single object $*$ and hom-object $1(*, *) = 1$. We identify $\V$-functors $1 \to A$ with objects in $A$ and $\V$\ndash profunctors \mbox{$1 \brar 1$} with $\V$-objects; transformations of such profunctors are identified with $\V$\ndash maps. 
	
	By a \emph{$\V$-weight} $J$ on a $\V$-category $B$ we mean a $\V$\ndash pro\-func\-tor $\hmap J1B$.
	\begin{definition} \label{definition: weighted limit}
		Let $\map dBM$ be a $\V$-functor, $J$ a $\V$-weight on $B$ and $r$ an object of $M$. A transformation $\eps$ in $\enProf\V$, as in the right-hand side below, is said to define $r$ as the \emph{$J$-weighted limit of $d$} if every transformation $\phi$ below factors uniquely through $\eps$ as shown.
		\begin{displaymath}
			\begin{tikzpicture}[textbaseline]
				\matrix(m)[math175em]{1 & 1 & B \\ M & & M \\};
				\path[map]	(m-1-1) edge[barred] node[above] {$H$} (m-1-2)
														edge node[left] {$s$} (m-2-1)
										(m-1-2) edge[barred] node[above] {$J$} (m-1-3)
										(m-1-3) edge node[right] {$d$} (m-2-3)
										(m-1-2) edge[cell] node[right] {$\phi$} (m-2-2);
				\path				(m-2-1) edge[eq] (m-2-3);
			\end{tikzpicture} = \begin{tikzpicture}[textbaseline]
				\matrix(m)[math175em]{1 & 1 & B \\ M & M & M \\};
				\path[map]	(m-1-1) edge[barred] node[above] {$H$} (m-1-2)
														edge node[left] {$s$} (m-2-1)
										(m-1-2) edge[barred] node[above] {$J$} (m-1-3)
														edge node[right] {$r$} (m-2-2)
										(m-1-3) edge node[right] {$d$} (m-2-3);
				\path				(m-2-1) edge[eq] (m-2-2)
										(m-2-2) edge[eq] (m-2-3);
				\path[transform canvas={shift={($(m-1-2)!0.5!(m-2-3)$)}}] (m-1-1) edge[cell] node[right] {$\phi'$} (m-2-1)
										(m-1-2) edge[cell] node[right] {$\eps$} (m-2-2);
			\end{tikzpicture}
  	\end{displaymath}
	\end{definition}
	The definition above extends the classical definition of weighted limit as follows. Note that if $\V$ is closed symmetric monoidal then $\V$-weights on $B$ can be identified with $\V$\ndash functors $\map JB\V$, using the isomorphism $\op 1 \tens B = 1 \tens B \iso B$; see \exref{example: enriched profunctors}. This recovers the classical definition of $\V$-weight, see e.g.\ Section 3.1 of \cite{Kelly82} where such $\V$-functors are called `indexing types'. Using this identification, the transformation $\eps$ above can be regarded as a $\V$-natural transformation between the $\V$-functors $\map JB\V$ and $\map{M(r,d)}B\V$.
	\begin{proposition} \label{classical definition of weighted limit}
		Let $d$, $J$, $r$ and $\eps$ be as above. If $\V$ is complete and closed symmetric monoidal then the transformation $\eps$ defines $r$ as the $J$-weighted limit of $d$, in the above sense, precisely if the pair $(r, \eps)$, where $\nat\eps J{M(r, d)}$ is regarded as a $\V$\ndash natural transformation of $\V$\ndash functors $B \to \V$, forms the `limit of $d$ indexed by $J$' in the sense of Section 3.1 of \cite{Kelly82}.
	\end{proposition}
	\begin{proof}
		In terms of $\V$-natural transformations between $\V$-functors $B \to \V$, the unique factorisations through $\eps$ above can be restated as the following universal property for the composite $\V$-natural transformations
		\begin{displaymath}
			\eps_s = \bigbrks{M(s, r) \tens J \xRar{\id \tens \eps} M(s,r) \tens M(r, d) \xRar\mu M(s,d)},
		\end{displaymath}
		where $s \in M$: for any $\V$-natural transformation \mbox{$\nat\phi{H \hc J}M(s, d)$}, where $H \in \V$, there exists a unique $\V$-map \mbox{$\map{\phi'}H{M(s,r)}$} such that
		\begin{displaymath}
			\phi = \bigbrks{H \hc J \xRar{\phi' \hc \id} M(s, r) \hc J \xRar{\eps_s} M(s,r)}.
		\end{displaymath}
	
		Now the adjunctions $\dash \tens Y \ladj \inhom{Y, \dash}$, that define the closed structure on $\V$, induce a bijective correspondence between $\V$-natural transformations \mbox{$H \hc J \Rar M(s,d)$}, of $\V$\ndash func\-tors $B \to \V$, and $\V$-natural transformations $H \Rar \inhom{J, M(s,d)}$, from the $\V$-object $H$ to the $\V$\ndash functor $\map{\inhom{J, M(s,d)}}{\op B \tens B}\V$ given on objects by $(y_1, y_2) \mapsto \inhom{Jy_1, M(s, dy_2)}$; the latter in the sense of Section~2.1 of \cite{Kelly82}. Since this correspondence is natural in $H$ the universality of the transformations $\nat{\eps_s}{M(s, r) \tens J}{M(s,d)}$ above translates into the universality of the corresponding transformations $M(s,r) \Rar \inhom{J, M(s,d)}$; in other words they define $M(s,r)$ as the end $\int_{y \in B}\inhom{Jy, M(s,dy)}$, for each $s \in M$. The latter is equivalent to the meaning of the definition of `$(r, \eps)$ is the limit of $d$ indexed by $J$' that is given in Section 3.1 of \cite{Kelly82}; see the sentence therein containing formula (3.3).
	\end{proof}
	
	Next we describe the pointwise Kan extensions of $\enProf\V$ in terms of weighted limits and, as a consequence, obtain an `enriched variant' of \propref{right Kan extensions along conjoints as right Kan extensions in V(K)}. Recall that vertical cells in $\enProf\V$ can be identified with $\V$-natural transformations of $\V$\ndash functors, under the isomorphism \eqref{vertical 2-category of V-Prof}.
	\begin{proposition} \label{pointwise Kan extensions in terms of weighted limits}
		Consider $\V$-functors $\map rAM$ and $\map dBM$ as well as a $\V$\ndash profunctor $\hmap JAB$.
		\begin{displaymath}
			\begin{tikzpicture}[textbaseline]
  			\matrix(m)[math175em]{A & B \\ M & M \\};
  			\path[map]  (m-1-1) edge[barred] node[above] {$J$} (m-1-2)
														edge node[left] {$r$} (m-2-1)
										(m-1-2) edge node[right] {$d$} (m-2-2);
				\path				(m-2-1) edge[eq] (m-2-2);
				\path[transform canvas={shift={($(m-1-2)!(0,0)!(m-2-2)$)}}] (m-1-1) edge[cell] node[right] {$\eps$} (m-2-1);
			\end{tikzpicture} \qquad\qquad \begin{tikzpicture}[textbaseline]
				\matrix(m)[math175em]{1 & B \\ A & B \\ M & M \\};
				\path[map]	(m-1-1) edge[barred] node[above] {$J(x, \id)$} (m-1-2)
														edge node[left] {$x$} (m-2-1)
										(m-2-1) edge[barred] node[below] {$J$} (m-2-2)
														edge node[left] {$r$} (m-3-1)
										(m-2-2) edge node[right] {$d$} (m-3-2);
				\path				(m-1-2) edge[eq] (m-2-2)
										(m-3-1) edge[eq] (m-3-2);
				\path[transform canvas={shift=(m-2-1), yshift=-2pt}]	(m-2-2) edge[cell] node[right] {$\eps$} (m-3-2);
				\draw				($(m-1-1)!0.5!(m-2-2)$)	node {\textup{cart}};
			\end{tikzpicture} \qquad\qquad \begin{tikzpicture}[textbaseline]
    		\matrix(m)[math175em]{B & B \\ A & \phantom A \\ M & M \\};
    		\path[map]  (m-1-1) edge node[left] {$j$} (m-2-1)
       			        (m-1-2) edge node[right] {$d$} (m-3-2)
       			        (m-2-1) edge node[left] {$r$} (m-3-1);
      	\path				(m-1-1) edge[eq] (m-1-2)
      							(m-3-1) edge[eq] (m-3-2);
    		\path[transform canvas={shift={($(m-2-1)!0.5!(m-1-1)$)}}] (m-2-2) edge[cell] node[right] {$\psi$} (m-3-2);
  		\end{tikzpicture} = \begin{tikzpicture}[textbaseline]
				\matrix(m)[math175em]{B & B \\ A & B \\ M & M \\};
				\path[map]	(m-1-1) edge node[left] {$j$} (m-2-1)
										(m-2-1) edge[barred] node[below] {$j^*$} (m-2-2)
														edge node[left] {$r$} (m-3-1)
										(m-2-2) edge node[right] {$d$} (m-3-2);
				\path				(m-1-1) edge[eq] (m-1-2)
										(m-1-2) edge[eq] (m-2-2)
										(m-3-1) edge[eq] (m-3-2);
				\path[transform canvas={shift=(m-2-1), yshift=-2pt}]	(m-2-2) edge[cell] node[right] {$\eps$} (m-3-2);
				\draw				($(m-1-1)!0.5!(m-2-2)$)	node {\textup{opcart}};
			\end{tikzpicture}
		\end{displaymath}
		For a transformation $\eps$ in $\enProf\V$ as on the left above, the following are equivalent:
		\begin{enumerate}[label=(\alph*)]
			\item $\eps$ defines $r$ as the pointwise right Kan extension of $d$ along $J$;
			\item for each $x \in A$ the composite in the middle above defines $rx$ as the $J(x, \id)$\ndash weighted limit of $d$.
		\end{enumerate}
		In particular, pointwise right Kan extensions along a $\V$-weight $\hmap J1B$ coincide with $J$-weighted limits.
		
		Finally assume that $\V$ is complete and closed symmetric monoidal. If $J = j^*$, for a $\V$-functor $\map jBA$, and $\eps$ is the factorisation of a $\V$-natural transformation $\nat\psi{rj}d$ through the opcartesian cell defining $j^*$, as on the right above, then condition (b) coincides with the meaning of the definition of `$\psi$ exhibits $r$ as the right Kan extension of $d$ along $j$' that is given in Section 4.1 of \cite{Kelly82}.
	\end{proposition}
	\begin{proof}
		For each $x \in A$ we denote by $\eps_x$ the composite in the middle above.
		
		(a) $\Rar$ (b). By \thmref{pointwise Kan extensions in terms of Kan extensions} below condition (a) implies that $\eps_x$ defines $rx$ as the pointwise right Kan extension of $d$ along $J(x, \id)$, for each $x \in A$. By restricting the universal property of \defref{definition: right Kan extension} for $\eps_x$ to cells $\phi$ with $\hmap H11$, we obtain the universal property of \defref{definition: weighted limit} for $\eps_x$, so that (b) follows.
		
		(b) $\Rar$ (a). Consider a transformation $\phi$ in $\enProf\V$, as in the middle composite below; we have to show that it factors uniquely as $\phi = \phi' \hc \eps$.
		\begin{equation} \label{weighted limit factorisation}
			\phi_{(z, x)} = \begin{tikzpicture}[textbaseline]
				\matrix(m)[math175em]{1 & 1 & B \\ C & A & B \\ M & & M \\};
				\path[map]	(m-1-1) edge[barred] node[above] {$H(z, x)$} (m-1-2)
														edge node[left] {$z$} (m-2-1)
										(m-1-2) edge[barred] node[above] {$J(x, \id)$} (m-1-3)
														edge node[right] {$x$} (m-2-2)
										(m-2-1) edge[barred] node[below] {$H$} (m-2-2)
														edge node[left] {$s$} (m-3-1)
										(m-2-2) edge[barred] node[below] {$J$} (m-2-3)
										(m-2-3) edge node[right] {$d$} (m-3-3)
										(m-2-2) edge[cell] node[right] {$\phi$} (m-3-2);
				\path				(m-1-3) edge[eq] (m-2-3)
										(m-3-1) edge[eq] (m-3-3);
				\draw				($(m-1-1)!0.5!(m-2-2)$) node {cart}
										($(m-1-2)!0.5!(m-2-3)$) node[xshift=2.5pt] {cart};
			\end{tikzpicture} = \begin{tikzpicture}[textbaseline]
				\matrix(m)[math175em]{1 & 1 & B \\ M & M & M \\};
				\path[map]	(m-1-1) edge[barred] node[above] {$H(z, x)$} (m-1-2)
														edge node[left] {$sz$} (m-2-1)
										(m-1-2) edge[barred] node[above] {$J(x, \id)$} (m-1-3)
														edge node[right, inner sep=1.5pt] {$rx$} (m-2-2)
										(m-1-3) edge node[right] {$d$} (m-2-3);
				\path				(m-2-1) edge[eq] (m-2-2)
										(m-2-2) edge[eq] (m-2-3);
				\path[map, transform canvas={shift=($(m-1-2)!0.5!(m-2-3)$), xshift=-10pt}]	(m-1-1) edge[cell] node[right] {$\phi'_{(z, x)}$} (m-2-1);
				\path[map, transform canvas={shift=($(m-1-2)!0.5!(m-2-3)$)}]	(m-1-2) edge[cell] node[right] {$\eps_x$} (m-2-2);
			\end{tikzpicture}
		\end{equation}
		Because the cells $\eps_x$ define weighted limits, the restrictions $\phi_{(z, x)}$ of $\phi$ above, for each pair $z \in C$ and $x \in A$, factor uniquely through $\eps_x$ as cells $\phi'_{(z, x)}$ as shown. These factorisations are simply $\V$-maps $\map{\phi'_{(z,x)}}{H(z,x)}{M(sz, rx)}$, and it remains to show that, as a family, they combine to form a transformation $\nat{\phi'}H{M(s,r)}$. Indeed, the unique factorisations above then imply that $\phi$ factors uniquely as $\phi = \phi' \hc \eps$, as needed.
		
		Thus we have to show that the $\V$-maps $\phi'_{(z,x)}$ are compatible with the left and right actions of $C$, $A$ and $M$. In terms of cells in $\enProf\V$, the compatibility with the left actions means that the identities
		\begin{displaymath}
			\begin{tikzpicture}[textbaseline]
				\matrix(m)[math175em, column sep=2em]{1 & 1 & 1 \\ 1 & & 1 \\ M & & M \\};
				\path[map]	(m-1-1) edge[barred] node[above] {$C(z_1, z_2)$} (m-1-2)
										(m-1-2) edge[barred] node[above] {$H(z_2, x)$} (m-1-3)
										(m-2-1) edge[barred] node[below] {$H(z_1, x)$} (m-2-3)
														edge node[left] {$sz_1$} (m-3-1)
										(m-2-3) edge node[right] {$rx$} (m-3-3)
										(m-1-2) edge[cell] node[right] {$\lambda$} (m-2-2)
										(m-2-2) edge[cell, transform canvas={yshift=-3pt}] node[right] {$\phi'_{(z_1, x)}$} (m-3-2);
				\path				(m-1-1) edge[eq] (m-2-1)
										(m-1-3) edge[eq] (m-2-3)
										(m-3-1) edge[eq] (m-3-3);
			\end{tikzpicture} = \begin{tikzpicture}[textbaseline]
				\matrix(m)[math175em, column sep=2em]{1 & 1 & 1 \\ M & M & M \\};
				\path[map]	(m-1-1) edge[barred] node[above] {$C(z_1, z_2)$} (m-1-2)
														edge node[left] {$sz_1$} (m-2-1)
										(m-1-2) edge[barred] node[above] {$H(z_2, x)$} (m-1-3)
														edge node[left, inner sep=0.5pt] {$sz_2$} (m-2-2)
										(m-1-3) edge node[right] {$rx$} (m-2-3);
				\path				(m-2-1) edge[eq] (m-2-2)
										(m-2-2) edge[eq] (m-2-3);
				\path[map, transform canvas={shift=($(m-1-2)!0.5!(m-2-3)$)}]	(m-1-1) edge[cell] node[left] {$s$} (m-2-1)
										(m-1-2) edge[cell, transform canvas={xshift=-11.5pt}] node[right] {$\phi'_{(z_2, x)}$} (m-2-2);
			\end{tikzpicture}
		\end{displaymath}
		hold, where $\lambda$ is the $\V$-map $C(z_1, z_2) \tens H(z_2, x) \to H(z_1, x)$ given by the left action of $C$ on $H$ and $s$ is the $\V$-map $C(z_1, z_2) \to M(sz_1, sz_2)$ given by the action of the $\V$-functor $s$ on hom-objects. Because factorisations through $\eps_x$ are unique we may equivalently show that this identity holds after composing both sides on the right with $\eps_x$. That it does is shown below where, to save space, only the non-identity cells are depicted while objects and morphisms are left out. The first and last identities here follow from the factorisations \eqref{weighted limit factorisation} while the second identity follows from the compatibility of $\phi$ with the action of $C$.
		\begin{displaymath}
			\begin{tikzpicture}[textbaseline, x=1.75em, y=1.75em, font=\scriptsize]
				\draw	(0,1) -- (3,1) -- (3,0) -- (0,0) -- (0,2) -- (2,2) -- (2,0);
				\draw[shift={(0,0.5)}]	(1,0) node {$\phi'_{(z_1, x)}$}
							(1,1) node {$\lambda$}
							(2.5,0) node {$\eps_x$};
			\end{tikzpicture} \mspace{15mu} = \mspace{15mu} \begin{tikzpicture}[textbaseline, x=1.75em, y=1.75em, font=\scriptsize]
				\draw (0,1) -- (3,1) -- (3,0) -- (0,0) -- (0,2) -- (2,2) -- (2,1);
				\draw[shift={(0,0.5)}]	(1.5,0) node {$\phi_{(z_1, x)}$}
							(1,1) node {$\lambda$};
			\end{tikzpicture} \mspace{15mu} = \mspace{15mu} \begin{tikzpicture}[textbaseline, x=1.75em, y=1.75em, font=\scriptsize]
				\draw (0,0) -- (3,0) -- (3,1) -- (0,1) -- (0,0)
							(1,1) -- (1,0);
				\draw[shift={(0,0.5)}]	(0.5,0) node {$s$}
							(2,0) node {$\phi_{(z_2, x)}$};
			\end{tikzpicture} \mspace{15mu} = \mspace{15mu} \begin{tikzpicture}[textbaseline, x=1.75em, y=1.75em, font=\scriptsize]
				\draw (0,0) -- (3.75,0) -- (3.75,1) -- (0,1) -- (0,0)
							(1,1) -- (1,0)
							(2.75,1) -- (2.75,0);
				\draw[shift={(0.5,0.5)}]	(0,0) node {$s$}
							(1.375,0) node {$\phi'_{(z_2, x)}$}
							(2.75,0) node {$\eps_x$};
			\end{tikzpicture}
		\end{displaymath}
		
		That the family of $\V$-maps $\phi'_{(z,x)}$ is compatible with the right actions as well follows similarly, from the following equation of transformations (which are of the form \mbox{$H(z,x_1) \hc A(x_1, x_2) \hc J(x_2, \id) \Rar 1_M$}) and the fact that factorisations through $\eps_{x_2}$ are unique.
		\begin{displaymath}
			\begin{tikzpicture}[textbaseline, x=1.75em, y=1.75em, font=\scriptsize]
			\draw	(0,1) -- (3,1) -- (3,0) -- (0,0) -- (0,2) -- (2,2) -- (2,0);
				\draw[shift={(0,0.5)}]	(1,0) node {$\phi'_{(z, x_2)}$}
							(1,1) node {$\rho$}
							(2.5,0) node {$\eps_{x_2}$};
			\end{tikzpicture} \mspace{3mu} = \mspace{3mu} \begin{tikzpicture}[textbaseline, x=1.75em, y=1.75em, font=\scriptsize]
				\draw (0,1) -- (3,1) -- (3,0) -- (0,0) -- (0,2) -- (2,2) -- (2,1);
				\draw[shift={(0,0.5)}]	(1.5,0) node {$\phi_{(z, x_2)}$}
							(1,1) node {$\rho$};
			\end{tikzpicture} \mspace{3mu} = \mspace{3mu} \begin{tikzpicture}[textbaseline, x=1.75em, y=1.75em, font=\scriptsize]
				\draw (3,1) -- (0,1) -- (0,0) -- (3,0) -- (3,2) -- (1,2) -- (1,1);
				\draw[shift={(0,0.5)}]	(1.5,0) node {$\phi_{(z, x_1)}$}
							(2,1) node {$\lambda$};
			\end{tikzpicture} \mspace{3mu} = \mspace{3mu} \begin{tikzpicture}[textbaseline, x=1.75em, y=1.75em, font=\scriptsize]
				\draw (3.75,1) -- (0,1) -- (0,0) -- (3.75,0) -- (3.75,2) -- (1.75,2) -- (1.75,0);
				\draw[shift={(0.5,0.5)}]	(0.375,0) node {$\phi'_{(z, x_1)}$}
							(2.25,0) node {$\eps_{x_1}$}
							(2.25,1) node {$\lambda$};
			\end{tikzpicture} \mspace{3mu} = \mspace{3mu} \begin{tikzpicture}[textbaseline, x=1.75em, y=1.75em, font=\scriptsize]
				\draw (0,0) -- (3.75,0) -- (3.75,1) -- (0,1) -- (0,0)
							(1.75,1) -- (1.75,0)
							(2.75,1) -- (2.75,0);
				\draw[shift={(0,0.5)}]	(0.875,0) node {$\phi'_{(z,x_1)}$}
							(2.25,0) node {$r$}
							(3.25,0) node {$\eps_{x_2}$};
			\end{tikzpicture}
		\end{displaymath}
		The identities here follow from the factorisations \eqref{weighted limit factorisation}; the fact that the domain $H \hc_A J$ of $\phi$ coequalises the actions of $A$ on $H$ and on $J$, see \eqref{composite of profunctors}; the factorisations \eqref{weighted limit factorisation} again; the compatibility of $\eps$ with the left actions.
		
		To prove the final assertion we assume that $J = j^*$ for a $\V$-functor $\map jBA$, and that $\eps$ is the unique factorisation of a $\V$-natural transformation $\nat\psi{rj}d$ through the opcartesian cell defining $j^*$. As soon as we consider $\psi$ as a vertical cell of $\enProf\V$, under the isomorphism \eqref{vertical 2-category of V-Prof}, it is easy to check that $\eps$ must be given by the $\V$-maps
		\begin{displaymath}
			A(x, jy) \xrar r M(rx, rjy) \xrar{\rho_{\psi_y}} M(rx,dy),
		\end{displaymath}
		for pairs $x \in A$ and $y \in B$. It then follows from \propref{classical definition of weighted limit} that, if $\V$ is complete and closed symmetric monoidal, then condition (b) coincides with condition (ii) of Theorem~4.6 of \cite{Kelly82}, which lists equivalent meanings of the definition of `$\psi$ exhibits $r$ as the right Kan extension of $d$ along $j$'.
	\end{proof}
	
	\subsection{Pointwise Kan extensions in terms of Kan extensions.}
	The main goal of \cite{Koudenburg14a} was to give a condition on equipments $\K$ ensuring that an analogue of \propref{right Kan extensions along conjoints as right Kan extensions in V(K)} holds for pointwise Kan extensions, where by `pointwise Kan extension in the $2$-category $V(\K)$', we mean the classical notion given by Street in \cite{Street74}. This condition is given in terms of the notion of tabulation, as follows.
	\begin{definition} \label{definition: tabulation}
		Given a horizontal morphism $\hmap JAB$ in a double category $\K$, the \emph{tabulation} $\tab J$ of $J$ consists of an object $\tab J$ equipped with a cell $\pi$ as on the left below, satisfying the following $1$-dimensional and $2$-dimensional universal properties.
		\begin{displaymath}
			\begin{tikzpicture}[textbaseline]
  			\matrix(m)[tab]{\tab J & \tab J \\ A & B \\};
  			\path[map]  (m-1-1) edge node[left] {$\pi_A$} (m-2-1)
										(m-1-2) edge node[right] {$\pi_B$} (m-2-2)
										(m-2-1) edge[barred] node[below] {$J$} (m-2-2);
				\path				(m-1-1) edge[eq] (m-1-2);
				\path[transform canvas={shift={($(m-1-2)!(0,0)!(m-2-2)$)}}] (m-1-1) edge[cell] node[right] {$\pi$} (m-2-1);
			\end{tikzpicture} \qquad\qquad\qquad\qquad\qquad \begin{tikzpicture}[textbaseline]
  			\matrix(m)[math175em]{X & X \\ A & B \\};
  			\path[map]  (m-1-1) edge node[left] {$\phi_A$} (m-2-1)
										(m-1-2) edge node[right] {$\phi_B$} (m-2-2)
										(m-2-1) edge[barred] node[below] {$J$} (m-2-2);
				\path				(m-1-1) edge[eq] (m-1-2);
				\path[transform canvas={shift={($(m-1-2)!(0,0)!(m-2-2)$)}}] (m-1-1) edge[cell] node[right] {$\phi$} (m-2-1);
			\end{tikzpicture}
		\end{displaymath}
		Given any other cell $\phi$ in $\K$ as on the right above, the $1$-dimensional property states that there exists a unique vertical morphism $\map{\phi'}X{\tab J}$ such that $\pi \of 1_{\phi'} = \phi$.
		
		The $2$-dimensional property is the following. Suppose we are given a further cell $\psi$ as in the identity below, which factors through $\pi$ as $\map{\psi'}Y{\tab J}$, like $\phi$ factors as $\phi'$. Then for any pair of cells $\xi_A$ and $\xi_B$ as below, so that the identity holds, there exists a unique cell $\xi'$ as on the right below such that $1_{\pi_A} \of \xi' = \xi_A$ and $1_{\pi_B} \of \xi' = \xi_B$.
		\begin{displaymath}
			 \begin{tikzpicture}[textbaseline]
				\matrix(m)[math175em]{X & Y & Y \\ A & A & B \\};
				\path[map]	(m-1-1) edge[barred] node[above] {$H$} (m-1-2)
														edge node[left] {$\phi_A$} (m-2-1)
										(m-1-2) edge node[right, inner sep=1pt] {$\psi_A$} (m-2-2)
										(m-1-3) edge node[right] {$\psi_B$} (m-2-3)
										(m-2-2) edge[barred] node[below] {$J$} (m-2-3);
				\path				(m-1-2) edge[eq] (m-1-3)
										(m-2-1) edge[eq] (m-2-2);
				\path[transform canvas={shift={($(m-1-2)!0.5!(m-2-3)$)}}] (m-1-1) edge[cell] node[right] {$\xi_A$} (m-2-1)
										(m-1-2) edge[cell] node[right] {$\psi$} (m-2-2);
			\end{tikzpicture} = \begin{tikzpicture}[textbaseline]
				\matrix(m)[math175em]{X & X & Y \\ A & B & B \\};
				\path[map]	(m-1-1) edge node[left] {$\phi_A$} (m-2-1)
										(m-1-2) edge[barred] node[above] {$H$} (m-1-3)
														edge node[right, inner sep=1pt] {$\phi_B$} (m-2-2)
										(m-1-3) edge node[right] {$\psi_B$} (m-2-3)
										(m-2-1) edge[barred] node[below] {$J$} (m-2-2);
				\path				(m-1-1) edge[eq] (m-1-2)
										(m-2-2) edge[eq] (m-2-3);
				\path[transform canvas={shift={($(m-1-2)!0.5!(m-2-3)$)}}] (m-1-1) edge[cell] node[right] {$\phi$} (m-2-1)
										(m-1-2) edge[cell] node[right] {$\xi_B$} (m-2-2);
			\end{tikzpicture} \qquad \qquad \begin{tikzpicture}[textbaseline]
  			\matrix(m)[math175em]{X & Y \\ \tab J & \tab J \\};
  			\path[map]  (m-1-1) edge[barred] node[above] {$H$} (m-1-2)
														edge node[left] {$\phi'$} (m-2-1)
										(m-1-2) edge node[right] {$\psi'$} (m-2-2);
				\path				(m-2-1) edge[eq] (m-2-2);
				\path[transform canvas={shift={($(m-1-2)!(0,0)!(m-2-2)$)}}] (m-1-1) edge[cell] node[right] {$\xi'$} (m-2-1);
			\end{tikzpicture}
		\end{displaymath}
		A tabulation is called \emph{opcartesian} whenever its defining cell $\pi$ is opcartesian.
	\end{definition}
	
	\begin{example} \label{example: tabulations of profunctors}
		The tabulation $\tab J$ of a profunctor $\map J{\op A \times B}\Set$ is the category that has triples $(x, u, y)$ as objects, where $(x, y) \in A \times B$ and $\map uxy$ in $J(x,y)$, while a map $(x, u, y) \to (x', u', y')$ is a pair $\map{(p, q)}{(x,y)}{(x', y')}$ in $A \times B$ making the diagram
	\begin{displaymath}
		\begin{tikzpicture}[textbaseline]
			\matrix(m)[math175em]{x & y \\ x' & y' \\};
			\path[map]	(m-1-1) edge node[above] {$u$} (m-1-2)
													edge node[left] {$p$} (m-2-1)
									(m-1-2) edge node[right] {$q$} (m-2-2)
									(m-2-1) edge node[below] {$u'$} (m-2-2);
		\end{tikzpicture}
	\end{displaymath}
	commute in $J$. The functors $\pi_A$ and $\pi_B$ are the projections and the natural transformation $\nat \pi{1_{\tab J}}J$ maps the pair $(p, q)$ to the diagonal $u' \of p = q \of u$. It is easy to check that $\pi$ satisfies both the $1$-dimensional and $2$-dimensional universal property above, and that it is opcartesian.
	\end{example}
	
	\begin{example}
		Generalising the previous example, it was shown in Proposition 5.15 of \cite{Koudenburg14a} that the double category $\inProf\E$, of profunctors internal to $\E$, admits all opcartesian tabulations.
	\end{example}
	
	The main result, Theorem 5.11, of \cite{Koudenburg14a} is the following.
	\begin{theorem} \label{pointwise Kan extensions in terms of Kan extensions}
		In an equipment $\K$ consider the cell $\eps$ on the left below.
		\begin{displaymath}
			\begin{tikzpicture}[textbaseline]
  			\matrix(m)[math175em]{A & B \\ M & M \\};
  			\path[map]  (m-1-1) edge[barred] node[above] {$J$} (m-1-2)
														edge node[left] {$r$} (m-2-1)
										(m-1-2) edge node[right] {$d$} (m-2-2);
				\path				(m-2-1) edge[eq] (m-2-2);
				\path[transform canvas={shift={($(m-1-2)!(0,0)!(m-2-2)$)}}] (m-1-1) edge[cell] node[right] {$\eps$} (m-2-1);
			\end{tikzpicture} \qquad\qquad\qquad\qquad\qquad \begin{tikzpicture}[textbaseline]
				\matrix(m)[math175em]{C & B \\ A & B \\ M & M \\};
				\path[map]	(m-1-1) edge[barred] node[above] {$J(f, \id)$} (m-1-2)
														edge node[left] {$f$} (m-2-1)
										(m-2-1) edge[barred] node[below] {$J$} (m-2-2)
														edge node[left] {$r$} (m-3-1)
										(m-2-2) edge node[right] {$d$} (m-3-2);
				\path				(m-1-2) edge[eq] (m-2-2)
										(m-3-1) edge[eq] (m-3-2);
				\path[transform canvas={shift=(m-2-1), yshift=-2pt}]	(m-2-2) edge[cell] node[right] {$\eps$} (m-3-2);
				\draw				($(m-1-1)!0.5!(m-2-2)$)	node {\textup{cart}};
			\end{tikzpicture} 
		\end{displaymath}
		For the following conditions the implications (a) $\Leftrightarrow$ (b) $\Rightarrow$ (c) hold, while (c) $\Rightarrow$ (a) holds as soon as $\K$ has opcartesian tabulations.
		\begin{enumerate}[label=(\alph*)]
			\item The cell $\eps$ defines $r$ as the pointwise right Kan extension of $d$ along $J$;
			\item for all $\map fCA$ the composite on the right above defines $r \of f$ as the pointwise right Kan extension of $d$ along $J(f, \id)$;
			\item for all $\map fCA$ the composite on the right above defines $r \of f$ as the right Kan extension of $d$ along $J(f, \id)$.
		\end{enumerate}
	\end{theorem}
	
	As a consequence the analogue of \propref{right Kan extensions along conjoints as right Kan extensions in V(K)} for pointwise Kan extensions below follows, which is Proposition 5.12 of \cite{Koudenburg14a}.
	\begin{proposition} \label{pointwise right Kan extensions along conjoints as pointwise right Kan extensions in V(K)}
		Let $\map dBM$, $\map jBA$ and $\map rAM$ be morphisms in an equipment $\K$ that has opcartesian tabulations. Consider a vertical cell $\eps$ as on the left below, as well as its factorisation through the opcartesian cell defining $j^*$, as shown.
		\begin{displaymath}
			\begin{tikzpicture}[textbaseline]
    		\matrix(m)[math175em]{B & B \\ A & \phantom A \\ M & M \\};
    		\path[map]  (m-1-1) edge node[left] {$j$} (m-2-1)
       			        (m-1-2) edge node[right] {$d$} (m-3-2)
       			        (m-2-1) edge node[left] {$r$} (m-3-1);
      	\path				(m-1-1) edge[eq] (m-1-2)
      							(m-3-1) edge[eq] (m-3-2);
    		\path[transform canvas={shift={($(m-2-1)!0.5!(m-1-1)$)}}] (m-2-2) edge[cell] node[right] {$\eps$} (m-3-2);
  		\end{tikzpicture} = \begin{tikzpicture}[textbaseline]
				\matrix(m)[math175em]{B & B \\ A & B \\ M & M \\};
				\path[map]	(m-1-1) edge node[left] {$j$} (m-2-1)
										(m-2-1) edge[barred] node[below] {$j^*$} (m-2-2)
														edge node[left] {$r$} (m-3-1)
										(m-2-2) edge node[right] {$d$} (m-3-2);
				\path				(m-1-1) edge[eq] (m-1-2)
										(m-1-2) edge[eq] (m-2-2)
										(m-3-1) edge[eq] (m-3-2);
				\path[transform canvas={shift=(m-2-1), yshift=-2pt}]	(m-2-2) edge[cell] node[right] {$\eps'$} (m-3-2);
				\draw				($(m-1-1)!0.5!(m-2-2)$)	node {\textup{opcart}};
			\end{tikzpicture}
		\end{displaymath}
		The vertical cell $\eps$ defines $r$ as the pointwise right Kan extension of $d$ along $j$ in $V(\K)$, in the sense of \cite{Street74}, precisely if its factorisation $\eps'$ defines $r$ as the pointwise right Kan extension of $d$ along $j^*$ in $\K$.
	\end{proposition}
	
	\subsection{Exact cells.}
	The final notions that we need to recall are those of `exact' and `initial' cell.	
	\begin{definition} \label{definition: exact cell}
		In a double category consider a cell $\phi$, as on the left below, and a vertical morphism $\map dDM$. We call $\phi$ \emph{(pointwise) right $d$-exact} if for any cell $\eps$, as on the right, that defines $r$ as the (pointwise) right Kan extension of $d$ along $K$, the composite $\eps \of \phi$ defines $r \of f$ as the (pointwise) right Kan extension of $d \of g$ along $J$. If the converse holds as well then we call $\phi$ \emph{(pointwise) $d$-initial}.
		\begin{displaymath}
			\begin{tikzpicture}[textbaseline]
				\matrix(m)[math175em]{A & B \\ C & D \\};
				\path[map]  (m-1-1) edge[barred] node[above] {$J$} (m-1-2)
														edge node[left] {$f$} (m-2-1)
										(m-1-2) edge node[right] {$g$} (m-2-2)
										(m-2-1) edge[barred] node[below] {$K$} (m-2-2);
				\path[transform canvas={shift={($(m-1-2)!(0,0)!(m-2-2)$)}}] (m-1-1) edge[cell] node[right] {$\phi$} (m-2-1);
			\end{tikzpicture} \qquad\qquad\qquad\qquad \begin{tikzpicture}[textbaseline]
  			\matrix(m)[math175em]{C & D \\ M & M \\};
  			\path[map]  (m-1-1) edge[barred] node[above] {$K$} (m-1-2)
														edge node[left] {$r$} (m-2-1)
										(m-1-2) edge node[right] {$d$} (m-2-2);
				\path				(m-2-1) edge[eq] (m-2-2);
				\path[transform canvas={shift={($(m-1-2)!(0,0)!(m-2-2)$)}}] (m-1-1) edge[cell] node[right] {$\eps$} (m-2-1);
			\end{tikzpicture}
		\end{displaymath}
		If $\phi$ is (pointwise) right $d$-exact for all vertical morphisms $\map dDM$, where $M$ varies, then it is called \emph{(pointwise) right exact}. Likewise $\phi$ is called \emph{(pointwise) initial} whenever it is (pointwise) $d$-initial for all $\map dDM$. (Pointwise) left exact and (pointwise) final cells are defined likewise.
	\end{definition}
	The uniqueness of Kan extensions implies that the notions of right $d$-exactness  and $d$-initialness do not depend on the choice of the cell $\eps$ that defines the Kan extension along $d$.
	\begin{example}
		In Example 4.3 of \cite{Koudenburg14a} it was shown that every `initial' functor $\map fAC$ induces an initial cell in $\Prof$, while in Example 4.6 of the same paper any natural transformation $\nat\phi{fj}{kg}$ between composites of functors, that satisfies the `right Beck-Chevalley condition', was shown to give rise to a pointwise right exact cell. We shall return to this condition later, at the end of \S5.
	\end{example}
	
	The following, which combines Proposition 4.2 and Corollary 4.5 of \cite{Koudenburg14a}, describes classes of initial and right exact cells.
	\begin{proposition} \label{pointwise exact cells}
		In a double category consider a cell $\phi$ as on the left below, and assume that the companions $\hmap{f_*}AC$ and $\hmap{g_*}BD$ of $f$ and $g$ exists. It follows that the opcartesian and cartesian cell in the composite on the right exist, see \eqref{cartesian and opcartesian cells in terms of companions and conjoints}, and we write $\phi_*$ for the unique factorisation of $\phi$ through these cells, as shown.
		\begin{displaymath}
			\begin{tikzpicture}[textbaseline]
				\matrix(m)[math175em]{A & B \\ C & D \\};
				\path[map]	(m-1-1) edge[barred] node[above] {$J$} (m-1-2)
														edge node[left] {$f$} (m-2-1)
										(m-1-2) edge node[right] {$g$} (m-2-2)
										(m-2-1) edge[barred] node[below] {$K$} (m-2-2);
				\path[transform canvas={shift={($(m-1-2)!(0,0)!(m-2-2)$)}}] (m-1-1) edge[cell] node[right] {$\phi$} (m-2-1);
			\end{tikzpicture} = \begin{tikzpicture}[textbaseline]
				\matrix(m)[math175em]{A & & B \\ A & B & D \\ A & & D \\ C & & D \\};
				\path[map]	(m-1-1) edge[barred] node[above] {$J$} (m-1-3)
										(m-1-3) edge node[right] {$g$} (m-2-3)
										(m-2-1) edge[barred] node[below] {$J$} (m-2-2)
										(m-2-2) edge[barred] node[below] {$g_*$} (m-2-3)
										(m-3-1) edge[barred] node[below] {$K(f, \id)$} (m-3-3)
														edge node[left] {$f$} (m-4-1)
										(m-4-1) edge[barred] node[below] {$K$} (m-4-3);
				\path				(m-1-1) edge[eq] (m-2-1)
										(m-2-1) edge[eq] (m-3-1)
										(m-2-3) edge[eq] (m-3-3)
										(m-3-3) edge[eq] (m-4-3)
										(m-2-2) edge[cell] node[right] {$\phi_*$} (m-3-2);
				\draw				($(m-1-1)!0.5!(m-2-3)$) node {\textup{opcart}}
										($(m-3-1)!0.5!(m-4-3)$) node[yshift=-3pt] {\textup{cart}};
			\end{tikzpicture}
		\end{displaymath}
		The following hold:
		\begin{enumerate}[label=\textup{(\alph*)}]
			\item if $f = \id_A$ and $\phi$ is opcartesian then $\phi$ is both initial and pointwise initial;
			\item $\phi$ is pointwise right exact if the following equivalent conditions hold: $\phi_* \of \textup{opcart}$ is opcartesian; $\textup{cart} \of \phi_*$ is cartesian; $\phi_*$ is invertible.
		\end{enumerate}
	\end{proposition}	
	
	\section{Eilenberg-Moore double categories}
	In this section we recall the notion of a `normal oplax double monad' $T$ on a double category and consider several weakenings of the `Eilenberg-Moore double category' associated to $T$, that was introduced in Section 7.1 of \cite{Grandis-Pare04}.
	
	More precisely, by a `normal oplax double monad' we shall mean a monad $T$ in the $2$-category $\noDbl$ of double categories, `normal oplax double functors' and their transformations. Since the assignment $\K \mapsto V(\K)$, that maps a double category $\K$ to its vertical $2$-category $V(\K)$, extends to a $2$\ndash functor $\noDbl \to \twoCat$, any normal oplax double monad $T$ on $\K$ induces a strict $2$-monad $V(T)$ on $V(\K)$. After generalising slightly the notions of `horizontal $T$-morphism' and `$T$-cell', that were introduced in \cite{Grandis-Pare04}, we will, for any choice of `weak' $\in \set{\text{colax, lax, pseudo}}$, show that lax $V(T)$-algebras, weak $V(T)$-morphisms, horizontal $T$-morphisms and $T$-cells form a double category $\Alg wT$.
	
	\subsection{Double functors and their transformations.}
	We start by recalling the notions of double functor and double transformation; references include Sections 7.2 and 7.3 of \cite{Grandis-Pare99} and Section 6 of \cite{Shulman08}. 
	\begin{definition} \label{definition: double functor}
  	Let $\K$ and $\L$ be double categories. A \emph{lax double functor} $\map F\K\L$ consists of a pair of functors $\map{F_\textup v}{\K_\textup v}{\L_\textup v}$ and $\map{F_\textup c}{\K_\textup c}{\L_\textup c}$ (both will be denoted $F$), such that $L F_\textup c = F_\textup v L$ and $R F_\textup c = F_\textup v R$, together with natural transformations
  	\begin{displaymath}
    	\begin{tikzpicture}
    	  \matrix(m)[math, column sep=2em, row sep=2.25em]
    	    { \K_\textup c \times_{\K_\textup v} \K_\textup c & \K_\textup c & \K_\textup v \\
    	      \L_\textup c \times_{\L_\textup v} \L_\textup c & \L_\textup c & \L_\textup v, \\};
    	  \path[map]  (m-1-1) edge node[above] {$\hc_\K$} (m-1-2)
    	                      edge node[left] {$F_\textup c \times_{F_\textup v} F_\textup c$} (m-2-1)
    	              (m-1-3) edge node[above] {$1$} (m-1-2)
    	                      edge node[right] {$F_\textup v$} (m-2-3)
    	              (m-2-1) edge node[below] {$\hc_\L$} (m-2-2)
    	              (m-1-2) edge node[left] {$F_\textup c$} (m-2-2)
    	              (m-2-3) edge node[below] {$1$} (m-2-2);
    	  \path[map]  (m-2-1) edge[cell, shorten >= 18pt, shorten <= 18pt] node[above left, inner sep=1.5pt] {$F_\hc$} (m-1-2)
    	              (m-2-3) edge[cell, shorten >= 12.5pt, shorten <= 12.5pt] node[above right, inner sep=1.5pt] {$F_1$} (m-1-2);;
    	\end{tikzpicture}
  	\end{displaymath}
  	whose components are horizontal cells, that satisfy the usual associativity and unit axioms for monoidal functors, see e.g.\ Section XI.2 of \cite{MacLane98}. The transformations $F_\hc$ and $F_1$ are called the \emph{compositor} and \emph{unitor} of $F$.
  	
  	Vertically dual, in the definition of an \emph{oplax double functor} $\map F\K\L$ the directions of the compositor $F_\hc$ and unitor $F_1$ are reversed. A \emph{pseudo double functor} is a lax (or, equivalently, an oplax) double functor whose compositor and unitor are invertible; a lax, oplax or pseudo double functor whose unitor $F_1$ is the identity is called \emph{normal}.
	\end{definition}
	
	We shall mostly be interested in normal oplax double functors. Unpacking the above, such a double functor $\map F\K\L$ maps the objects, vertical and horizontal morphisms, as well as cells of $\K$ to those of $\L$, in a way that preserves horizontal and vertical sources and targets. Moreover, vertical composition of morphisms and cells is preserved, as are the horizontal units: $F1_A = 1_{FA}$ and $F1_f = 1_{Ff}$ for each $A$ and $\map fAC$ in $\K$, while horizontal composition is preserved only up to natural coherent cells
	\begin{displaymath}
		\begin{tikzpicture}[textbaseline]
			\matrix(m)[math175em]{FA & & FE \\ FA & FB & FE \\};
			\path[map]	(m-1-1) edge[barred] node[above] {$F(J \hc H)$} (m-1-3)
									(m-2-1) edge[barred] node[below] {$FJ$} (m-2-2)
									(m-2-2) edge[barred] node[below] {$FH$} (m-2-3)
									(m-1-2) edge[cell] node[right] {$F_\hc$} (m-2-2);
			\path				(m-1-1) edge[eq] (m-2-1)
									(m-1-3) edge[eq] (m-2-3);
		\end{tikzpicture}
	\end{displaymath}
	for each composable pair $\hmap JAB$ and $\hmap HBE$ in $\K$. We will often use the unit axioms for $F$: they state that, for any horizontal morphism $\hmap JAB$, the composite
	\begin{equation} \label{unit axiom for normal oplax double functors}
		F(1_A \hc J) \xRar{F_\hc} F1_A \hc FJ = 1_{FA} \hc FJ \xRar{\mf l} FJ
	\end{equation}
	coincides with $F\mf l$, and similar for $J \hc 1_B$.
	
	\begin{example} \label{example: span functors}
		Any functor $\map F\D\E$ between categories with pullbacks induces a normal oplax double functor $\map{\Span F}{\Span\D}{\Span\E}$, simply by applying $F$ to the spans of $\D$. The compositor of $\Span F$ is induced by the universal property of pullbacks; in particular, if $F$ preserves pullbacks then $\Span F$ is a normal pseudo double functor.
	\end{example}
	
	\begin{definition} \label{definition: transformation}
  	A \emph{double transformation} $\nat\xi FG$ between lax double functors $F$ and $\map G\K\L$ is given by natural transformations $\nat{\xi_\textup v}{F_\textup v}{G_\textup v}$ and $\nat{\xi_\textup c}{F_\textup c}{G_\textup c}$ (both will be denoted $\xi$), with $L\xi_\textup c = \xi_\textup v L$ and $R\xi_\textup c = \xi_\textup vR$, such that the following diagrams commute, where $(J, H) \in \K_\textup c \pb RL \K_\textup c$ and $A \in \K_\textup v$.
  	\begin{displaymath}
			\begin{tikzpicture}[baseline]
				\matrix(m)[math175em]{ FJ \hc FH & F(J \hc H) \\ GJ \hc GH & G(J \hc H) \\};
				\path	(m-1-1) edge[cell] node[above] {$F_\hc$} (m-1-2)
											edge[cell] node[left] {$\xi_J \hc \xi_H$} (m-2-1)
							(m-1-2) edge[cell] node[right] {$\xi_{J \hc H}$} (m-2-2)
							(m-2-1) edge[cell] node[below] {$G_\hc$} (m-2-2);
			\end{tikzpicture} \qquad \qquad \begin{tikzpicture}[baseline]
				\matrix(m)[math175em]{ 1_{FA} & F1_A \\ 1_{GA} & G1_A \\};
				\path	(m-1-1) edge[cell] node[above] {$F_1$} (m-1-2)
											edge[cell] node[left] {$1_{\xi_A}$} (m-2-1)
							(m-1-2) edge[cell] node[right] {$\xi_{1_A}$} (m-2-2)
							(m-2-1) edge[cell] node[below] {$G_1$} (m-2-2);
			\end{tikzpicture}
		\end{displaymath}
		
		Analogously, in the definition of a double transformation $\nat\xi FG$ between oplax double functors $F$ and $\map G\K\L$ the directions of the coherence cells in the diagrams above are reversed; in both cases we shall call the commuting of these diagrams respectively the composition and unit axiom for $\xi$.
	\end{definition}
	
	Double categories, lax double functors and double transformations form a $2$\ndash cat\-e\-gory which we denote $\lDbl$; analogously $\oDbl$ denotes the $2$-category of double categories, oplax double functors and double transformations. We denote by $\nlDbl \subset \lDbl$ and $\noDbl \subset \oDbl$ the sub-$2$-categories consisting of normal double functors. Notice that the unit axiom for a double transformation $\nat\xi FG$ between normal double functors reduces to the identity $\xi_{1_A} = 1_{\xi_A}$, for all objects $A$ of $\K$.
	
	\begin{example}
		Every transformation $\nat\xi FG$, between functors of categories that have pullbacks, induces a double transformation $\nat{\Span\xi}{\Span F}{\Span G}$ between normal oplax double functors, that is given by \mbox{$A \mapsto \xi_A$} on objects and $(A \leftarrow J \rightarrow B) \mapsto \xi_J$ on spans. We conclude that the assignment $\E \mapsto \Span\E$, that maps a category $\E$ with pullbacks to the double category $\Span\E$ of spans in $\E$, extends to a $2$-functor $\map{\mathsf{Span}}{\mathsf{pbCat}}\noDbl$, where $\mathsf{pbCat}$ denotes the $2$-category of categories with pullbacks, all functors between such categories and their transformations.
	\end{example}
	
	The following propositions record some useful properties of double functors and their transformations. Remember that every pseudo double category $\K$ contains a vertical $2$-category $V(\K)$; in the proposition below $\twoCat$ denotes the $2$-category of $2$-categories, strict $2$-functors and $2$\ndash transformations.
	\begin{proposition}\label{normal functors induce vertical functors}
	  The assignment $\K \mapsto V(\K)$ extends to a strict $2$-functor
	  \begin{displaymath}
		  \map V\noDbl \twoCat.
		\end{displaymath}
	\end{proposition}
	\begin{proof}[(sketch).]
	  The image $\map{V(F)}{V(\K)}{V(\L)}$ of a double functor $\map F\K\L$ is simply the restriction of $F$ to objects, vertical morphisms and vertical cells; that $V(F)$ is a strict $2$-functor follows from the unit axiom for $F$. Likewise the vertical part $\nat{\xi_\textup v}{F_\textup v}{G_\textup v}$ of any double transformation $\nat\xi FG$ forms a $2$-transformation $\nat{V(\xi)}{V(F)}{V(G)}$.
	\end{proof}
	
	The following is Proposition 6.8 of \cite{Shulman08}.
	\begin{proposition}[Shulman] \label{preserving of cartesian and opcartesian cells}
	Every lax double functor between equipments preserves cartesian cells, while every oplax double functor between equipments preserves opcartesian cells.
	\end{proposition}
	
	Having introduced double functors and their transformations we turn to double monads.
	\begin{definition}
		An \emph{oplax double monad} on a double category $\K$ is a monad $T = (T, \mu, \eta)$ on $\K$ in the $2$-category $\oDbl$, consisting of an oplax double endofunctor $\map T\K\K$, a multiplication $\nat\mu{T^2}T$ and a unit $\nat\eta{\id_\K}T$, which satisfy the associativity axiom $\mu \of T\mu = \mu \of \mu T$ and unit axiom $\mu \of T\eta = \id_T = \mu \of \eta T$.
		
		Dually a \emph{lax double monad} is a monad in $\lDbl$. An oplax, or lax, double monad is called \emph{normal} or \emph{pseudo} whenever its underlying double endofunctor is a normal or pseudo double functor.
	\end{definition}
	
	Notice that, \propref{normal functors induce vertical functors}, any normal oplax double monad $T$ on $\K$ restricts to a strict $2$-monad $V(T)$ on the vertical $2$-category $V(\K)$. The following example describes the `free monoid'-monad on $\Mat\V$. In \exref{example: free strict monoidal enriched category monad} below we will see that it induces the `free strict monoidal $\V$-category'-monad on $\enProf\V$, which we shall use throughout.
	
	For many other examples of lax double monads we refer to Section 3 of \cite{Cruttwell-Shulman10}, where `monads on virtual double categories' are considered. Analogous to the generalisation of the notion of monoidal category to that of `multicategory', a `virtual double category' consists of objects, vertical morphisms and horizontal morphisms, like a double category, but has cells of the form $(J_1, \dotsc, J_n) \Rar K$, where $(J_1, \dotsc, J_n)$ is a sequence of horizontal morphisms $\hmap{J_1}{A_0}{A_1}, \dotsc, \hmap{J_n}{A_{n-1}}{A_n}$. Virtual double categories that have all `composites' and `units', in the sense of Section 5 of \cite{Cruttwell-Shulman10}, can be identified with double categories in our sense, and `monads' on such virtual double categories correspond to the lax double monads in our sense.
	\begin{example} \label{example: free monoid double monad}
		Let $\V = (\V, \tens, 1, \mf s)$ be a symmetric monoidal category that has coproducts  preserved by $\tens$ on both sides, so that $\V$-matrices form a double category $\Mat\V$; see \exref{example: matrices}. The pseudo double endofunctor $T$ underlying the \emph{`free monoid'-monad} on $\Mat\V$ assigns to a set $A$ the free monoid $TA = \coprod_{n \geq 0} A^n$ and, accordingly, maps a function $\map fAC$ to $Tf = \coprod_{n \geq 0} f^n$. The elements of $TA$ are (possibly empty) sequences $\ul x = (x_1, \dotsc, x_n)$ of elements in $A$; we write $\lns{\ul x} = n$. The function $Tf$ simply applies $f$ coordinatewise.
		
		The assignments $A \mapsto TA$ and $f \mapsto Tf$ extend to a pseudo double endofunctor $T$ on $\Mat\V$ as follows. Its image of a $\V$-matrix $\hmap JAB$ is given by
		\begin{displaymath}
			TJ(\ul x, \ul y) = \begin{cases}
				\Tens_{i = 1}^n J(x_i, y_i) & \text{if $\card{\ul x} = n = \card{\ul y}$;} \\
				\emptyset	& \text{otherwise.}
			\end{cases}
		\end{displaymath}
		Likewise the image $T\phi$ of the cell on the left below is given by the family of $\V$-maps on the right. 
		\begin{displaymath}
	  	\begin{tikzpicture}[textbaseline]
	  	  \matrix(m)[math175em]{A & B \\ C & D \\};
	  	  \path[map]  (m-1-1) edge[barred] node[above] {$J$} (m-1-2)
	  	                      edge node[left] {$f$} (m-2-1)
	  	              (m-1-2) edge node[right] {$g$} (m-2-2)
	  	              (m-2-1) edge[barred] node[below] {$K$} (m-2-2);
	  	  \path[transform canvas={shift={($(m-1-2)!(0,0)!(m-2-2)$)}}] (m-1-1) edge[cell] node[right] {$\phi$} (m-2-1);
	  	\end{tikzpicture} \qquad\qquad (T\phi)_{(\ul x, \ul y)} = \begin{cases}
				\Tens_{i = 1}^n \phi_{(x_i, y_i)} & \text{if $\card{\ul x} = n = \card{\ul y}$;} \\
				\id_\emptyset	& \text{otherwise.}
			\end{cases}
		\end{displaymath}
		Next the compositor $T_\hc \colon TJ \hc TH \iso T(J \hc H)$ is given, on sequences $\ul x$ and $\ul z$ of equal length $n$, by the canonical isomorphisms
		\begin{multline*}
			(TJ \hc TH)(\ul x, \ul z) = \mspace{-7mu}\coprod_{\ul y \in TB} \mspace{-5mu} TJ(\ul x, \ul y) \tens TH(\ul y, \ul z) \iso \mspace{-7mu}\coprod_{\ul y \in B^n} \mspace{-2mu} \Bigpars{\Tens_{i=1}^n J(x_i, y_i)} \tens \Bigpars{\Tens_{i=1}^n H(y_i, z_i)} \\
			\iso \coprod_{\ul y \in B^n} \Tens_{i=1}^n \bigpars{J(x_i, y_i) \tens H(y_i, z_i)} \iso \Tens_{i=1}^n \coprod_{y \in B} \bigpars{J(x_i, y) \tens H(y, z_i)} = T(J \hc H)(\ul x, \ul z),
		\end{multline*}
		where the first and the third isomorphism follow from the fact that $\tens$ preserves coproducts on both sides, while the second isomorphism is given by reordering the factors $J(x_i, y_i)$ and $H(y_i, z_i)$; of course $(T_\hc)_{(\ul x, \ul z)} = \id_\emptyset$ if $\card{\ul x} \neq \card{\ul z}$. On a set $A$ the unitor $T_1\colon 1_{TA} \iso T1_A$ is given by the isomorphisms $1_{TA}(\ul x, \ul y) = 1 \iso 1^{\tens n} = (T1_A)(\ul x, \ul y)$ if $\ul x = \ul y$ and $\id_\emptyset$ otherwise. This completes the definition of the pseudo double endofunctor $T$ on $\Mat\V$.
		
		The multiplication $\nat\mu{T^2}T$ and unit $\nat\eta\id T$, that make $T$ into a double monad, restrict to sets as concatenation of double sequences and insertion as one-element sequences respectively. On a $\V$-matrix $\hmap JAB$ the multiplication $\cell{\mu_J}{T^2J}{TJ}$ is given as follows: if $\ull x = (\ul x_1, \dotsc, \ul x_n)$ and $\ull y = (\ul y_1, \dotsc, \ul y_n)$ are double sequences of equal length, such that $\card{\ul x_i} = m_i = \card{\ul y_i}$ for all $i = 1, \dotsc, n$, then $(\mu_J)_{(\ull x, \ull y)}$ is the $\V$-isomorphism
		\begin{equation} \label{concatenation}
			(T^2J)(\ull x, \ull y) = \Tens_{i=1}^n \Bigpars{\Tens_{j=1}^{m_i} J(x_{ij}, y_{ij})} \iso \mspace{-9mu} \Tens_{i=1}^{m_1 + \dotsb + m_n} \mspace{-9mu} J\bigpars{(\mu\ull x)_i, (\mu\ull y)_i} = (TJ)(\mu\ull x, \mu\ull y);
		\end{equation}
		in all other cases it is the unique $\V$-map $\emptyset \to (TJ)(\mu\ull x, \mu\ull y)$. Finally, the unit $\cell{\eta_J}J{TJ}$ is given by the isomorphisms $J(x, y) \iso (TJ)(\eta x, \eta y)$; this completes the description of the pseudo double monad $T$.
		
		We remark that $T$ can be modified into the following lax double monad on $\Mat\V$, which we denote $T_\Sigma$, by introducing `twists' in the images $TJ$, as follows:
		\begin{displaymath}
			T_\Sigma J(\ul x, \ul y) = \begin{cases}
				\coprod_{\sigma \in \Sigma_n} \Tens_{i = 1}^n J(x_{\sigma i}, y_i) & \text{if $\card{\ul x} = n = \card{\ul y}$;} \\
				\emptyset	& \text{otherwise,}
			\end{cases}
		\end{displaymath}
		where $\Sigma_n$ denotes the group of permutations of the set $\set{1, \dotsc, n}$. The double monad $T_\Sigma$ induces the `free symmetric strict monoidal $\V$-category'-monad on $\enProf\V$; see \exref{example: free strict monoidal enriched category monad}.
	\end{example}
	
	Remember that monoids and bimodules in an equipment $\K$ form again an equipment $\Mod\K$ whenever $\K$ has local reflexive coequalisers, that is each category $H(\K)(A, B)$ has reflexive coequalisers which are preserved by $\hc$ on both sides; see \propref{bimodule double categories}. In the proposition below, which is Proposition~11.12 of \cite{Shulman08}, $\mathsf{qEquip}_\textup l \subset \lDbl$ and $\mathsf{qEquip}_\textup{nl} \subset \nlDbl$ denote the full sub-$2$\ndash categories consisting of equipments that have local reflexive coequalisers.
	\begin{proposition}[Shulman] \label{mod is a 2-functor}
		The assignment $\K \mapsto \Mod\K$ extends to a $2$-functor $\map{\mathsf{Mod}}{\mathsf{qEquip}_\textup l}{\mathsf{qEquip}_\textup{nl}}$, whose image of a pseudo double functor \mbox{$\map F\D\E$} is again pseudo whenever $F$ preserves local reflexive coequalisers.
	\end{proposition}
	
	By applying $\mathsf{Mod}$ to the `free monoid'-monad on $\Mat\V$, of \exref{example: free monoid double monad}, we obtain the `free strict monoidal $\V$-category'-monad on $\enProf\V$, as follows; this double monad, whose algebras are monoidal $\V$-categories (see \exref{example: monoidal enriched categories}), will be our main example.
	\begin{example} \label{example: free strict monoidal enriched category monad}
		Let $\V$ be a cocomplete symmetric monoidal category whose tensor product preserves colimits on both sides. We denote the image $\Mod T$ of the `free monoid'\ndash monad on $\Mat\V$, of \exref{example: free monoid double monad}, again by $T$; it is the \emph{`free strict monoidal $\V$\ndash category'-monad}, which is given as follows. Its image of a $\V$-category $A$ is the free strict monoidal $\V$-category generated by $A$, that is $TA = \coprod_{n \geq 0} A^{\tens n}$; likewise its image of a $\V$-functor $\map fAC$ is given by $Tf = \coprod_{n \geq 0} f^{\tens n}$. In detail, the objects of $TA$ are (possibly empty) sequences $\ul x = (x_1, \dotsc, x_n)$ of objects in $A$, while its hom-objects $(TA)(\ul x, \ul y)$ are given by the tensor products $\Tens_{i=1}^n A(x_i, y_i)$ if $\lns{\ul x} = \lns{\ul y}$, and $\emptyset$ otherwise. The functor $Tf$ simply applies $f$ coordinatewise.
		
		The assignments $A \mapsto TA$ and $f \mapsto Tf$ extend to a double functor $T$ on $\enProf\V$ as follows. Its image of a profunctor $\hmap JAB$ is given by
		\begin{displaymath}
			TJ(\ul x, \ul y) = \begin{cases}
				\Tens_{i = 1}^n J(x_i, y_i) & \text{if $\card{\ul x} = n = \card{\ul y}$;} \\
				\emptyset	& \text{otherwise,}
			\end{cases}
		\end{displaymath}
		on which the hom-objects of $TA$ and $TB$ act coordinatewise. It is clear that $T$ is normal, that is $T1_A = 1_{TA}$ for all $\V$-categories $A$. The invertible compositors $T_\hc\colon TJ \hc TH \iso T(J \hc H)$, for composable $\V$-profunctors $J$ and $H$, are induced by those of the `free monoid'\ndash monad on $\Mat\V$; that they are again invertible follows, by \propref{mod is a 2-functor}, from the fact that the latter preserves local reflexive coequalisers. This is a direct consequence of the fact that a tensor product of reflexive coequalisers in $\V$ forms again a reflexive coequaliser; see e.g.\ Corollary~A1.2.12 of \cite{Johnstone02} for the dual result. This completes the description of the normal pseudo double endofunctor $T$. The multiplication and unit, that make $T$ into a double monad, are induced by those of the `free monoid'-monad in the obvious way.
		
		Recall that, at the end of \exref{example: free monoid double monad}, we briefly considered a lax double monad $T_\Sigma$ on $\Mat\V$, that was obtained by introducing twists in the images $TJ$ of the `free monoid'-monad. Applying $\mathsf{Mod}$ to $T_\Sigma$, we obtain the \emph{`free symmetric strict monoidal $\V$-cat\-e\-gory'\ndash monad} on $\enProf\V$, whose algebras are symmetric monoidal $\V$-categories; see \exref{example: symmetric monoidal enriched category}. Again denoting it $T_\Sigma$, its image on a $\V$-category $A$ is given as follows: the objects of $T_\Sigma A$ are (possibly empty) sequences of objects in $A$, while its hom-objects are given by
		\begin{displaymath}
			T_\Sigma A(\ul x, \ul y) = \begin{cases}
				\coprod_{\sigma \in \Sigma_n} \Tens_{i = 1}^n A(x_{\sigma i}, y_i) & \text{if $\card{\ul x} = n = \card{\ul y}$;} \\
				\emptyset	& \text{otherwise.}
			\end{cases}
		\end{displaymath}
		Its action on $\V$-profunctors is given analogously; that $T_\Sigma$ is a normal pseudo double monad on $\enProf\V$ is shown in Proposition 3.23 of \cite{Koudenburg13}.
	\end{example}
	
	In closing this subsection we briefly consider monoidal double categories, and study pointwise Kan extensions along tensor products of horizontal morphisms. The following is Definition 9.1 of \cite{Shulman08}.
	\begin{definition}
		A \emph{monoidal double category} is a double category $\K$ equipped with a pseudo double functor $\map\tens{\K \times \K}\K$ and an object $1 \in \K$, together with invertible double transformations $\tens \of (\tens \times \id) \iso \tens \of (\id \times \tens)$ and $1 \tens \dash \iso \id \iso \dash \tens 1$ which satisfy the usual coherence axioms.
	\end{definition}
	
	We abbreviate by $1 = \hmap{1_1}11$ the horizontal unit for $1$. For an unraveling of the above definition we refer to Section 9 of \cite{Shulman08}; for us it suffices to notice the `interchange' isomorphisms
	\begin{displaymath}
		\tens_\hc\colon (J \tens H) \hc (K \tens L) \iso (J \hc K) \tens (H \hc L),
	\end{displaymath}
	which form the invertible compositors of $\tens$.
	
	\begin{example} \label{example: monoidal double category of enriched profunctors}
		Let $\V$ be a cocomplete monoidal category whose tensor product preserves colimits on both sides, so that $\V$-profunctors form a double category $\enProf\V$ as we saw in \exref{example: enriched profunctors}. If $\V$ is symmetric monoidal then we can form tensor products of $\V$-categories and $\V$-functors; see e.g.\ Section 1.4 of \cite{Kelly82}. It is straightforward to extend these tensor products to a monoidal structure on $\enProf\V$, in which the tensor product of $\V$-profunctors $\hmap{J \tens H}{A \tens C}{B \tens D}$ is given by the $\V$-objects $(J \tens H)\bigpars{(x, z)(y, w)} = J(x, y) \tens H(z, w)$, and whose unit is the $\V$-category $1$ consisting of a single object $*$ and hom-object $1(*,*) = 1$.
	\end{example}
	
	We consider pointwise Kan extensions along tensor products of horizontal morphisms. The following is a direct consequence of \propref{iterated pointwise Kan extensions}.
	\begin{proposition} \label{iterated pointwise Kan extensions in V-Prof}
		Consider horizontally composable cells
		\begin{displaymath}
			\begin{tikzpicture}
				\matrix(m)[math175em]{A \tens C & B \tens C & B \tens D \\ M & M & M \\};
				\path[map]	(m-1-1) edge[barred] node[above] {$J \tens 1_C$} (m-1-2)
														edge node[left] {$s$} (m-2-1)
										(m-1-2) edge[barred] node[above] {$1_B \tens H$} (m-1-3)
														edge node[right] {$r$} (m-2-2)
										(m-1-3) edge node[right] {$d$} (m-2-3);
				\path	(m-2-1) edge[eq] (m-2-2)
							(m-2-2) edge[eq] (m-2-3);
				\path[transform canvas={shift={($(m-1-2)!0.5!(m-2-3)$)}}] (m-1-1) edge[cell] node[right] {$\gamma$} (m-2-1)
				(m-1-2) edge[cell] node[right] {$\eps$} (m-2-2);
			\end{tikzpicture}
		\end{displaymath}
		in a monoidal double category and suppose that $\eps$ defines $r$ as the pointwise right Kan extension of $d$ along $1_B \tens H$. Then $\gamma$ defines $s$ as the pointwise right Kan extension of $r$ along $J \tens 1_C$ precisely if the composite
		\begin{displaymath}
			J \tens H \overset{\inv{\mf r} \tens \inv{\mf l}} \iso (J \hc 1_B) \tens (1_C \hc H) \overset{\inv{\tens_\hc}} \iso (J \tens 1_C) \hc (1_B \tens H) \xRar{\gamma \hc \eps} 1_M
		\end{displaymath}
		defines $s$ as the pointwise right Kan extension of $d$ along $J \tens H$.
	\end{proposition}
	
	Lastly we consider Kan extensions along $\V$-profunctors of the form $J \tens 1_C$, in the monoidal double category $\enProf\V$.
	\begin{displaymath}
		\begin{tikzpicture}[baseline]
			\matrix(m)[math175em]{A \tens C & B \tens C \\ M & M \\};
			\path[map]	(m-1-1) edge[barred] node[above] {$J \tens 1_C$} (m-1-2)
													edge node[left] {$s$} (m-2-1)
									(m-1-2) edge node[right] {$d$} (m-2-2);
			\path				(m-2-1) edge[eq] (m-2-2);
			\path[transform canvas={shift=($(m-1-2)!0.5!(m-2-2)$)}]	(m-1-1) edge[cell] node[right] {$\gamma$} (m-2-1);
		\end{tikzpicture} \qquad\qquad\qquad\qquad \begin{tikzpicture}[baseline]
			\matrix(m)[math175em]{A & B \\ M & M \\};
			\path[map]	(m-1-1) edge[barred] node[above] {$J$} (m-1-2)
													edge node[left] {$s(\dash, z)$} (m-2-1)
									(m-1-2) edge node[right] {$d(\dash, z)$} (m-2-2);
			\path				(m-2-1) edge[eq] (m-2-2);
			\path[transform canvas={shift=($(m-1-2)!0.5!(m-2-2)$)}]	(m-1-1) edge[cell] node[right] {$\gamma_z$} (m-2-1);
		\end{tikzpicture}
	\end{displaymath}
	Given a transformation $\gamma$ in $\enProf\V$ as on the left above and an object $z \in C$, we write $\gamma_z$ for the transformation on the right that is given by the composite below; here we have identified $z$ with the corresponding $\V$-functor $\map z1C$.
	\begin{displaymath}
		J \iso J \tens 1 \xRar{\id \tens 1_z} J \tens 1_C \xRar\gamma 1_M.
	\end{displaymath}
	\begin{proposition} \label{pointwise Kan extensions along unit tensored V-profunctors}
		Let $\V$ be a cocomplete symmetric monoidal category whose tensor product preserves colimits on both sides, so that $\V$-profunctors form a monoidal double category $\enProf\V$. A transformation $\gamma$, as on the left above, defines $s$ as the pointwise right Kan extension of $d$ along $J \tens 1_C$ precisely if the transformations $\gamma_z$ on the right, for each $z \in C$, define $s(\dash, z)$ as the pointwise right Kan extension of $d(\dash, z)$ along $J$. An analogous result holds for pointwise right Kan extensions along the $\V$-profunctor $1_C \tens J$.
	\end{proposition}
	\begin{proof}
		For any $x \in A$ we write $\delta_x$ for the cartesian cell defining the restriction $J(x, \id)$. Likewise for any $z \in C$ we write $\eps_z$ for the cartesian cell defining the companion $\hmap{z_*}1C$ of $\map z1C$; remember from the discussion following \exref{cartesian and opcartesian cells for profunctors} that the unit cell $\cell{1_z}1{1_C}$ factors as $1_z = \eps_z \of \eta_z$, where $\cell{\eta_z}1{z_*}$ is an opcartesian cell that defines $z_*$ as the extension of $\hmap 111$ along $\id_1$ and $z$. 
		
		Notice that, by \propref{preserving of cartesian and opcartesian cells}, the tensor product $\cell{\delta_x \tens \eps_z}{J(x, \id) \tens z_*}{J \tens 1_C}$ is a cartesian cell defining the restriction of $J \tens 1_C$ along $\map{x \tens z}{1 \tens 1}{A \tens C}$. It follows from \propref{pointwise Kan extensions in terms of weighted limits} that we may equivalently prove that, for each pair $(x, z)$, the composite $\gamma_{(x, z)} = \gamma \of (\delta_x \tens \eps_z)$ defines $s(x,z)$ as the $J(x, \id) \tens z_*$-weighted limit of $d$ if and only if the composite $\gamma_z \of \delta_x$ defines $s(x, z)$ as the $J(x, \id)$-weighted limit of $d(\dash, z)$.
		
		To see this we use that $\id_{J(x, \id)} \tens \eta_z$ is pointwise initial, because tensor products of opcartesian cells are opcartesian (\propref{preserving of cartesian and opcartesian cells}) and by \propref{pointwise exact cells}(a). By regarding weighted limits as pointwise Kan extensions (\propref{pointwise Kan extensions in terms of weighted limits}), it follows that each $\gamma_{(x, z)}$ defines $s(x, z)$ as the $J(x, \id) \tens z_*$-weighted limit of $d$ precisely if the composite $\gamma_{(x, z)} \of (\id \tens \eta_z)$ defines $s(x, z)$ as the $J(x, \id) \tens 1$-weighted limit of $d(\dash, z)$. To complete the proof we notice that $\gamma_{(x, z)} \of (\id \tens \eta_z) = \gamma \of (\delta_x \tens 1_z)$ which, after composition with $J(x, \id) \iso J(x, \id) \tens 1$, coincides with $\gamma_z \of \delta_x$.
	\end{proof}
	
	\subsection{Eilenberg-Moore double categories.}
	In this subsection we define, for a normal oplax double monad $T$, several weakenings of the `Eilenberg-Moore double category' associated to $T$, that was introduced by Grandis and Par\'e in \cite{Grandis-Pare04}. We start by recalling from \cite{Street74} the definitions of lax algebras for strict $2$-monads, as well as several notions of weak morphism between such algebras, and the cells between those. The strict $2$-monads of interest to us are those of the form $V(T)$, where $T$ is a normal oplax double monad.
	
	Consider a strict $2$-monad $T = (T, \mu, \eta)$ on a $2$-category $\C$, consisting of a strict $2$\ndash functor $\map T\C\C$ and $2$-natural transformations $\nat\mu{T^2}T$ and $\nat\eta {\id_\C}T$ satisfying the usual axioms: $\mu \of T\mu = \mu \of \mu T$ and $\mu \of T\eta = \id_T = \mu \of \eta T$.
	\begin{definition} \label{definition: lax algebra}
  	A \emph{lax $T$\ndash algebra} is a quadruple $A = (A, a, \bar a, \hat a)$ consisting of an object $A$ equipped with a morphism $\map a{TA}A$, its \emph{structure map}, and cells $\cell{\bar a}{a \of Ta}{a \of \mu_A}$ and $\cell{\hat a}{\id_A}{a \of \eta_A}$, its \emph{associator} and \emph{unitor}, that satisfy the following three coherence conditions.
  	\begin{displaymath}
    	\begin{tikzpicture}[textbaseline]
      	\matrix(m)[math, column sep=0.75em, row sep=2em]
	      { & T^2 A &[-0.9em] \phantom{T^3A} & TA & \phantom{T^3A} \\
  	      T^3 A & \phantom{T^3A} & T^2 A & \phantom{T^3A} & A \\
  	      & T^2 A & & TA & \\ };
  	    \path[map]  (m-1-2) edge node[above] {$Ta$} (m-1-4)
  	                (m-1-4) edge node[above right] {$a$} (m-2-5)
  	                (m-2-1) edge node[above left] {$T^2 a$} (m-1-2)
  	                        edge node[below] {$T\mu_A$} (m-2-3)
  	                        edge node[below left] {$\mu_{TA}$} (m-3-2)
  	                (m-2-3) edge node[above left, inner sep=1.5pt] {$Ta$} (m-1-4)
  	                        edge node[below left, inner sep=1.5pt] {$\mu_A$} (m-3-4)
  	                (m-3-2) edge node[below] {$\mu_A$} (m-3-4)
  	                (m-3-4) edge node[below right] {$a$} (m-2-5)
  	                (m-1-2) edge[cell, shorten >=8pt, shorten <=8pt] node[below left] {$T\bar a$} (m-2-3)
  	                ($(m-2-3)!0.5!(m-2-5)+(0,0.8em)$) edge[cell] node[right] {$\bar a$} ($(m-2-3)!0.5!(m-2-5)-(0,0.8em)$);
  	  \end{tikzpicture} = \begin{tikzpicture}[textbaseline]
    	  \matrix(m)[math, column sep=0.75em, row sep=2em]
    	  { & T^2 A & \phantom{T^3A} &[-0.9em] TA & \phantom{T^3A} \\
    	    T^3 A & \phantom{T^3A} & TA & \phantom{T^3A} & A \\
    	    & T^2 A & & TA & \\ };
    	  \path[map]  (m-1-2) edge node[above] {$Ta$} (m-1-4)
    	                      edge node[below left, inner sep=1.5pt] {$\mu_A$} (m-2-3)
    	              (m-1-4) edge node[above right] {$a$} (m-2-5)
    	              (m-2-1) edge node[above left] {$T^2 a$} (m-1-2)
    	                      edge node[below left] {$\mu_{TA}$} (m-3-2)
    	              (m-2-3) edge node[below] {$a$} (m-2-5)
    	              (m-3-2) edge node[above left, inner sep=1.5pt] {$Ta$} (m-2-3)
    	                      edge node[below] {$\mu_A$} (m-3-4)
    	              (m-3-4) edge node[below right] {$a$} (m-2-5)
    	              (m-1-4) edge[cell, shorten >=8pt, shorten <=8pt] node[below right] {$\bar a$} (m-2-3)
    	              (m-2-3) edge[cell, shorten >=8pt, shorten <=8pt] node[below left] {$\bar a$} (m-3-4);
    	\end{tikzpicture}
  	\end{displaymath}
  	\begin{displaymath}
  	  \begin{tikzpicture}[textbaseline]
	      \matrix(m)[math, column sep=0.75em, row sep=2em]
  	    { \phantom{T^3A} & A &[-0.9em] \phantom{T^3A} & \phantom{T^3A} & \phantom{T^3A} \\
  	      TA & \phantom{T^3A} & & & A \\
  	      & T^2 A & & TA & \\ };
  	    \path[map]  (m-1-2) edge[bend left=22] node[above] {$\id$} (m-2-5)
  	                (m-2-1) edge node[above left] {$a$} (m-1-2)
  	                        edge[bend left=22] node[above right] {$\id$} (m-3-4)
  	                        edge node[below left] {$\eta_{TA}$} (m-3-2)
  	                (m-3-2) edge node[below] {$\mu_A$} (m-3-4)
  	                (m-3-4) edge node[below right] {$a$} (m-2-5);
  	  \end{tikzpicture} = \begin{tikzpicture}[textbaseline]
  	    \matrix(m)[math, column sep=0.75em, row sep=2em]
  	    { \phantom{T^3A} & A & \phantom{T^3A} &[-0.9em] & \phantom{T^3A} \\
  	      TA & \phantom{T^3A} & TA & \phantom{T^3A}  & A \\
  	      & T^2 A & & TA & \\ };
  	    \path[map]  (m-1-2) edge[bend left=22] node[above] {$\id$} (m-2-5)
  	                        edge node[below left, inner sep=1.5pt] {$\eta_A$} (m-2-3)
  	                (m-2-1) edge node[above left] {$a$} (m-1-2)
  	                        edge node[below left] {$\eta_{TA}$} (m-3-2)
  	                (m-2-3) edge node[below] {$a$} (m-2-5)
  	                (m-3-2) edge node[above left, inner sep=1.5pt] {$Ta$} (m-2-3)
  	                        edge node[below] {$\mu_A$} (m-3-4)
  	                (m-3-4) edge node[below right] {$a$} (m-2-5)
  	                (m-1-4) edge[cell, shorten >= 6pt, shorten <= 14pt] node[below right] {$\hat a$} (m-2-3)
  	                (m-2-3) edge[cell, shorten >= 8pt, shorten <= 8pt] node[below left] {$\bar a$} (m-3-4);
  	  \end{tikzpicture}
  	\end{displaymath}
  	\begin{displaymath}
  	  \begin{tikzpicture}[textbaseline]
  	    \matrix(m)[math, column sep=0.75em, row sep=2em]
  	    { \phantom{T^3A} &  &[-0.9em] \phantom{T^3A} & TA & \phantom{T^3A} \\
  	      TA & \phantom{T^3A} & T^2 A & \phantom{T^3A} & A \\
  	      & & & TA & \\ };
  	    \path[map]  (m-1-4) edge node[above right] {$a$} (m-2-5)
  	                (m-2-1) edge[bend left=22] node[above left] {$\id$} (m-1-4)
  	                        edge node[below] {$T\eta_A$} (m-2-3)
  	                        edge[bend right=22] node[below left] {$\id$} (m-3-4)
  	                (m-2-3) edge node[above left, inner sep=1.5pt] {$Ta$} (m-1-4)
  	                        edge node[below left, inner sep=1.5pt] {$\mu_A$} (m-3-4)
  	                (m-3-4) edge node[below right] {$a$} (m-2-5)
  	                (m-1-2) edge[cell, shorten >= 6pt, shorten <= 14pt] node[below left] {$T\hat a$} (m-2-3)
  	                ($(m-2-3)!0.5!(m-2-5)+(0,0.8em)$) edge[cell] node[right] {$\bar a$} ($(m-2-3)!0.5!(m-2-5)-(0,0.8em)$);
  	  \end{tikzpicture} = \begin{tikzpicture}[textbaseline]
	      \matrix(m)[math, column sep=0.75em, row sep=2em]
  	    { \phantom{T^3A} & \phantom{T^3A} & \phantom{T^3A} &[-0.9em] TA & \phantom{T^3A} \\
  	      TA & & & \phantom{T^3A} & A \\
  	      & & & TA & \\ };
  	    \path[map]  (m-1-4) edge node[above right] {$a$} (m-2-5)
  	                (m-2-1) edge[bend left=22] node[above left] {$\id$} (m-1-4)
  	                        edge[bend right=22] node[below left] {$\id$} (m-3-4)
  	                (m-3-4) edge node[below right] {$a$} (m-2-5);
  	  \end{tikzpicture}
  	\end{displaymath}
  	
  	Dually, in the notion of a \emph{colax $T$-algebra} $A = (A, a, \bar a, \hat a)$ the directions of the associator $\cell{\bar a}{a \of \mu_A}{a \of Ta}$ and the unitor $\cell{\hat a}{a \of \eta_A}{\id_A}$ are reversed. A lax or colax $T$-algebra $A$ is called a \emph{pseudo} $T$\ndash algebra if the cells $\bar a$ and $\hat a$ are isomorphisms; if they are identities then $A$ is called \emph{strict}.
  \end{definition}
  \begin{definition} \label{definition: weak algebra morphisms}
  	Given lax $T$-algebras $A = (A, a, \bar a, \hat a)$ and $C = (C, c, \bar c, \hat c)$, a \emph{lax $T$\ndash morphism} $A \to C$ is a morphism $\map fAC$ that is equipped with a \emph{structure cell} \mbox{$\cell{\bar f}{c \of Tf}{f \of a}$}, which satisfies the following associativity and unit axioms.
  	\begin{displaymath}
  	  \begin{tikzpicture}[textbaseline]
  	    \matrix(m)[math, column sep=0.75em, row sep=2em]
  	    { \phantom{T^2C} & T^2 C &[-0.9em] \phantom{T^2C} & TC & \phantom{T^2C} \\
  	      T^2 A & & TA & \phantom{T^2C} & C \\
  	      & TA & & A & \\ };
  	    \path[map]  (m-1-2) edge node[above] {$Tc$} (m-1-4)
  	                (m-1-4) edge node[above right] {$c$} (m-2-5)
  	                (m-2-1) edge node[above left] {$T^2 f$} (m-1-2)
  	                        edge node[below] {$Ta$} (m-2-3)
  	                        edge node[below left] {$\mu_A$} (m-3-2)
  	                (m-2-3) edge node[below right, inner sep=1.5pt] {$Tf$} (m-1-4)
  	                        edge node[above right, inner sep=1.5pt] {$a$} (m-3-4)
  	                (m-3-2) edge node[below] {$a$} (m-3-4)
  	                (m-3-4) edge node[below right] {$f$} (m-2-5)
  	                (m-2-3) edge[cell, shorten >= 8pt, shorten <= 8pt] node[below right] {$\bar a$} (m-3-2)
  	                (m-1-2) edge[cell, shorten >= 8pt, shorten <= 8pt] node[above right] {$T\bar f$} (m-2-3)
  	                ($(m-2-3)!0.5!(m-2-5)+(0,0.8em)$) edge[cell] node[right] {$\bar f$} ($(m-2-3)!0.5!(m-2-5)-(0,0.8em)$);
  	  \end{tikzpicture} = \begin{tikzpicture}[textbaseline]
  	    \matrix(m)[math, column sep=0.75em, row sep=2em]
  	    { \phantom{T^2C} & T^2 C & \phantom{T^2C} &[-0.9em] TC & \phantom{T^2C}\\
  	      T^2 A & & TC & \phantom{T^2C} & C \\
  	      & TA & & A & \\ };
  	    \path[map]  (m-1-2) edge node[above] {$Tc$} (m-1-4)
  	                        edge node[below left, inner sep=1.5pt] {$\mu_C$} (m-2-3)
  	                (m-1-4) edge node[above right] {$c$} (m-2-5)
  	                (m-2-1) edge node[above left] {$T^2 f$} (m-1-2)
  	                        edge node[below left] {$\mu_A$} (m-3-2)
  	                (m-2-3) edge node[below] {$c$} (m-2-5)
  	                (m-3-2) edge node[above left, inner sep=1.5pt] {$Tf$} (m-2-3)
  	                        edge node[below] {$a$} (m-3-4)
  	                (m-3-4) edge node[below right] {$f$} (m-2-5)
  	                (m-1-4) edge[cell, shorten >= 8pt, shorten <= 8pt] node[below right] {$\bar c$} (m-2-3)
  	                (m-2-3) edge[cell, shorten >= 8pt, shorten <= 8pt] node[below left] {$\bar f$} (m-3-4);
  	  \end{tikzpicture}
  	\end{displaymath}
  	\begin{displaymath}
  	  \begin{tikzpicture}[textbaseline]
      	\matrix(m)[math, column sep=0.75em, row sep=2em]
	      { \phantom{T^2C} & C &[-0.9em] \phantom{T^2C} & \phantom{T^2C} & \phantom{T^2C} \\
  	      A & \phantom{T^2C} & & & C \\
  	      & TA & & A & \\ };
  	    \path[map]  (m-1-2)	edge[bend left=21] node[above right] {$\id$} (m-2-5)
										(m-2-1) edge node[above left] {$f$} (m-1-2)
														edge[bend left=21] node[above right] {$\id$} (m-3-4)
  	                        edge node[below left] {$\eta_A$} (m-3-2)
  	                (m-3-2) edge node[below] {$a$} (m-3-4)
  	                (m-3-4) edge node[below right] {$f$} (m-2-5)
  	                (m-2-3) edge[cell, shorten >= 6pt, shorten <= 14pt] node[below right] {$\hat a$} (m-3-2);
  	 	\end{tikzpicture} = \begin{tikzpicture}[textbaseline]
  	    \matrix(m)[math, column sep=0.75em, row sep=2em]
  	    { \phantom{T^2C} & C & \phantom{T^2C} &[-0.9em] & \phantom{T^2C} \\
  	      A & \phantom{T^2C} & TC & \phantom{T^2C} & C \\
  	      & TA & & A & \\ };
  	    \path[map]  (m-1-2)	edge[bend left=21] node[above right] {$\id$} (m-2-5)
														edge node[below left, inner sep=1.5pt] {$\eta_C$} (m-2-3)
										(m-2-1) edge node[above left] {$f$} (m-1-2)
  	                        edge node[below left] {$\eta_A$} (m-3-2)
  	                (m-2-3) edge node[below] {$c$} (m-2-5)
  	                (m-3-2) edge node[above left, inner sep=1.5pt] {$Tf$} (m-2-3)
  	                				edge node[below] {$a$} (m-3-4)
  	                (m-3-4) edge node[below right] {$f$} (m-2-5)
  	                (m-2-3) edge[cell, shorten >= 8pt, shorten <= 8pt] node[below left] {$\bar f$} (m-3-4)
  	                (m-1-4) edge[cell, shorten >= 6pt, shorten <= 14pt] node[below right] {$\hat c$} (m-2-3);
  	  \end{tikzpicture}
  	\end{displaymath}
  
  	Dually a \emph{colax $T$-morphism} $\map fAC$ is equipped with a structure cell $\cell{\bar f}{f \of a}{c \of Tf}$ such that
  	\begin{displaymath}
  		(\bar c . T^2 f) \of (c . T\bar f) \of (\bar f . Ta) = (\bar f . \mu_A) \of (f . \bar a) \qquad \text{and} \qquad (\bar f . \eta_A) \of (f . \hat a) = \hat c . f.	
  	\end{displaymath}
  	Lax or colax $T$-morphisms whose structure cell $\bar f$ is invertible are called \emph{pseudo $T$\ndash mor\-phisms}; notice that, for an invertible cell $\cell{\bar f}{c \of Tf}{f \of a}$, the pair $(f, \bar f)$ forms a lax $T$-morphism if and only if $(f, \inv{\bar f})$ forms a colax $T$-morphism. Pseudo $T$-morphisms are called \emph{strict} whenever $\bar f$ is the identity cell.
  \end{definition}
  
  This leaves the definition of cells between $T$-morphisms.
	\begin{definition} \label{definition: vertical T-cell}
  	Given lax $T$-morphisms $f$ and $\map gAC$, a \emph{$T$-cell} between $f$ and $g$ is a cell $\cell\phi fg$ satisfying
  	\begin{displaymath}
  	  \begin{tikzpicture}[textbaseline]
  	    \matrix(m)[math, column sep=3em, row sep=3em]{TA & TC \\ A & C \\};
  	    \path[map]  (m-1-1) edge[bend left=35] node[above] {$Tf$} (m-1-2)
  	                        edge[bend right=35] node[below] {$Tg$} (m-1-2)
  	                        edge node[left] {$a$} (m-2-1)
  	                (m-1-2) edge node[right] {$c$} (m-2-2)
  	                (m-2-1) edge[bend right=35] node[below] {$g$} (m-2-2)
  	                ($(m-1-1)!0.5!(m-1-2)+(-0.3em,0.8em)$) edge[cell] node[right] {$T\phi$} ($(m-1-1)!0.5!(m-1-2)-(0.3em,0.8em)$)
  	                (m-1-2) edge[cell, shorten >= 21.5pt, shorten <= 21.5pt, transform canvas = {yshift=-10pt}] node[below right] {$\bar g$} (m-2-1);
  	  \end{tikzpicture} = \begin{tikzpicture}[textbaseline]
				\matrix(m)[math, column sep=3em, row sep=3em]{TA & TC \\ A & C. \\};
    	  \path[map]  (m-1-1) edge[bend left=35] node[above] {$Tf$} (m-1-2)
    	                      edge node[left] {$a$} (m-2-1)
    	              (m-1-2) edge node[right] {$c$} (m-2-2)
    	              (m-2-1) edge[bend left=35] node[above] {$f$} (m-2-2)
    	              				edge[bend right=35] node[below] {$g$} (m-2-2)
    	              ($(m-2-1)!0.5!(m-2-2)+(-0.3em,0.8em)$) edge[cell] node[right] {$\phi$} ($(m-2-1)!0.5!(m-2-2)-(0.3em,0.8em)$)
    	              (m-1-2) edge[cell, shorten >= 21.5pt, shorten <= 21.5pt, transform canvas = {yshift=10pt}] node[above left] {$\bar f$} (m-2-1);
    	\end{tikzpicture}
  	\end{displaymath}
  	Likewise a $T$-cell between colax $T$-morphisms $f$ and $\map gAC$ is a cell $\cell\phi fg$ satisfying $(c.T\phi) \of \bar f = \bar g \of (\phi.a)$.
	\end{definition}
	For any strict $2$-monad $T$ lax $T$-algebras, lax $T$-morphisms and $T$-cells form a $2$\ndash cat\-e\-gory $\Alg lT$, and we denote by $\Alg{ps}T$ the sub-2-category of lax $T$-algebras, pseudo $T$\ndash morphisms and all $T$-cells between them. Likewise lax $T$-algebras, colax $T$\ndash mor\-phisms and $T$-cells form a $2$-category $\Alg cT$.
	
	\begin{example} \label{example: monoidal enriched categories}
		Let $T$ be the `free strict monoidal $\V$-category'-monad described in \exref{example: free strict monoidal enriched category monad}. A lax $V(T)$-algebra $A = (A, \tens, \mf a, \mf i)$ is a \emph{lax monoidal $\V$-category}: it consists of a $\V$-category $A$ equipped with a $\V$-functor $\map\tens{TA}A$ that defines the tensor product $(x_1 \tens \dotsb \tens x_n) = \tens(\ul x)$ of each (possibly empty) sequence $\ul x$ in $TA$, as well as $\V$-natural maps
		\begin{displaymath}
			\map{\mf a}{\bigpars{(x_{11} \tens \dotsb \tens x_{1m_1}) \tens \dotsb \tens (x_{n1} \tens \dotsb \tens x_{nm_n})}}{(x_{11} \tens \dotsb \tens x_{nm_n})},
		\end{displaymath}
		for each double sequence $\ull x \in T^2A$, and $\map{\mf i}x{(x)}$, where $x \in A$. These transformations satisfy certain coherence conditions, see Definition 3.1.1 of \cite{Leinster04} for example. A lax monoidal $\V$-category $A$ whose structure transformations $\mf a$ and $\mf i$ are invertible is simply called a \emph{monoidal $\V$-category}; if they are identities then $A$ is called a \emph{strict} monoidal $\V$-category. These notions of monoidal category are often called \emph{unbiased}, referring to the fact that tensor products of all arities are part of their structure, in contrast to the classical notion, which is biased towards the nullary (unit) and binary tensor products.
		
		Lax $V(T)$-morphisms are \emph{lax monoidal $\V$-functors}, that is $\V$-functors $\map fAC$ equipped with $\V$-natural maps
		\begin{displaymath}
			\map{f_\tens}{\pars{fx_1 \tens \dotsb \tens fx_n}}{f(x_1 \tens \dotsb \tens x_n)},
		\end{displaymath}
		its \emph{compositors}, that are compatible with the coherence transformations of $A$ and $C$; see Definition~3.1.3 of \cite{Leinster04}. Dually, in the notion of colax $V(T)$-morphism the direction of the compositors is reversed; such $\V$-functors are called \emph{colax monoidal $\V$-functors}. Lax or colax monoidal $\V$-functors whose compositors are invertible are simply called \emph{monoidal $\V$-functors}; those with identities as compositors are called \emph{strict} monoidal $\V$-functors.
		
		A $V(T)$-cell $f \Rar g$ between lax monoidal $\V$-functors $f$ and $g$, called a \emph{monoidal transformation}, is simply a $\V$-natural transformation $\nat\phi fg$ that is compatible with the compositors $f_\tens$ and $g_\tens$: the diagrams below commute for each sequence $\ul x \in TA$. Monoidal transformations between colax monoidal $\V$-functors are defined analogously.
		\begin{displaymath}
			\begin{tikzpicture}
				\matrix(m)[math2em]
				{ \pars{fx_1 \tens \dotsb \tens fx_n} & f(x_1 \tens \dotsb \tens x_n) \\
					\pars{gx_1 \tens \dotsb \tens gx_n} & g(x_1 \tens \dotsb \tens x_n) \\ };
				\path[map]	(m-1-1) edge node[above] {$f_\tens$} (m-1-2)
														edge node[left] {$(\phi_{x_1} \tens \dotsb \tens \phi_{x_n})$} (m-2-1)
										(m-1-2) edge node[right] {$\phi_{(x_1 \tens \dotsb \tens x_n)}$} (m-2-2)
										(m-2-1) edge node[below] {$g_\tens$} (m-2-2);
			\end{tikzpicture}
		\end{displaymath}
	\end{example}
	\begin{example} \label{example: symmetric monoidal enriched category}
		Consider the `free symmetric strict monoidal $\V$-category'-monad $T_\Sigma$, that was briefly considered at the end of \exref{example: free strict monoidal enriched category monad}. The notions of lax $V(T_\Sigma)$\ndash algebra and lax $V(T_\Sigma)$-morphism are that of symmetric lax monoidal category and symmetric lax monoidal functor, as follows. A \emph{symmetric lax monoidal $\V$-category} $A$ is a lax monoidal $\V$-category $A$ that is equipped with $\V$-natural symmetries
		\begin{displaymath}
			\map{\mf s_\sigma}{(x_1 \tens \dotsb \tens x_n)}{(x_{\sigma 1} \tens \dotsb \tens x_{\sigma n})},
		\end{displaymath}
		for all $\sigma \in \Sigma_n$ and $\ul x \in A^{\tens n}$, that are functorial in the sense that $\mf s_\tau \of \mf s_\sigma = \mf s_{\tau \of \sigma}$ and $\mf s_{\id} = \id$, and that make the following diagrams commute. Firstly, given permutations $\sigma_i \in \Sigma_{m_i}$, the diagrams
		\begin{displaymath}
			\begin{tikzpicture}[font=\scriptsize]
				\matrix(m)[math2em]{\bigpars{(x_{11} \tens \dotsb \tens x_{1m_1}) \tens \dotsb \tens (x_{n1} \tens \dots \tens x_{nm_n})} & (x_{11} \tens \dotsb \tens x_{nm_n}) \\
					\bigpars{(x_{1\sigma_11} \tens \dotsb \tens x_{1\sigma_1m_1}) \tens \dotsb \tens (x_{n\sigma_n1} \tens \dotsb \tens x_{n\sigma_nm_n})} & (x_{1\sigma_11} \tens \dotsb \tens x_{n\sigma_nm_n}) \\};
				\path[map]	(m-1-1) edge node[above] {$\mf a$} (m-1-2)
														edge node[left] {$(\mf s_{\sigma_1} \tens \dotsb \tens \mf s_{\sigma_n})$} (m-2-1)
										(m-1-2) edge node[right] {$\mf s_{(\sigma_1, \dotsc, \sigma_n)}$} (m-2-2)
										(m-2-1) edge node[below] {$\mf a$} (m-2-2);
			\end{tikzpicture}
		\end{displaymath}
		commute where, writing $N = m_1 + \dotsb + m_n$, the permutation $\pars{\sigma_1, \dotsc, \sigma_n} \in \sym_N$ denotes the disjoint union of the $\sigma_i$. Secondly, for each $\tau \in \Sigma_n$, the diagrams
		\begin{displaymath}
			\begin{tikzpicture}[font=\scriptsize]
				\matrix(m)[math2em]{\bigpars{(x_{11} \tens \dotsb \tens x_{1m_1}) \tens \dotsb \tens (x_{n1} \tens \dots \tens x_{nm_n})} & (x_{11} \tens \dotsb \tens x_{nm_n}) \\
					\bigpars{(x_{\tau 11} \tens \dotsb \tens x_{\tau 1m_{\tau 1}}) \tens \dotsb \tens (x_{\tau n1} \tens \dotsb \tens x_{\tau nm_{\tau n}})} & (x_{\tau 11} \tens \dotsb \tens x_{\tau nm_{\tau n}}) \\};
				\path[map]	(m-1-1) edge node[above] {$\mf a$} (m-1-2)
														edge node[left] {$\mf s_\tau$} (m-2-1)
										(m-1-2) edge node[right] {$\mf s_{\tau_{(m_1, \dotsc, m_n)}}$} (m-2-2)
										(m-2-1) edge node[below] {$\mf a$} (m-2-2);
			\end{tikzpicture}
		\end{displaymath}
		commute, where the `block permutation' $\tau_{\pars{m_1, \dotsc, m_n}} \in \sym_N$ permutes the $n$ blocks of the $\pars{m_{\tau 1}, \dotsc, m_{\tau n}}$\ndash partition of $\set{1, \dotsc, N}$ exactly like $\tau$ permutes the elements of $\set{1, \dotsc, n}$, while preserving the ordering of the elements in each block.
		
		A lax monoidal functor $\map fAC$ between symmetric lax monoidal categories is called \emph{symmetric} whenever the diagrams below commute. Transformations between symmetric lax monoidal functors are simply monoidal transformations.
		\begin{equation} \label{symmetry axiom}
			\begin{tikzpicture}[textbaseline]
				\matrix(m)[math2em]
				{	(fx_1 \tens \dotsb \tens fx_n) & f(x_1 \tens \dotsb \tens x_n) \\
					(fx_{\sigma 1} \tens \dotsb \tens fx_{\sigma n}) & f(x_{\sigma 1} \tens \dotsb \tens x_{\sigma n}) \\ };
				\path[map]	(m-1-1) edge node[above] {$f_\tens$} (m-1-2)
														edge node[left] {$\mathfrak s_\sigma$} (m-2-1)
										(m-1-2) edge node[right] {$f\mathfrak s_\sigma$} (m-2-2)
										(m-2-1) edge node[below] {$f_\tens$} (m-2-2);
			\end{tikzpicture}
		\end{equation}
	\end{example}	
	
	Having recalled, for a normal oplax double monad $T$ on $\K$, the notions of lax $V(T)$\ndash al\-ge\-bra, weak $V(T)$-morphism and $V(T)$-cell in the $2$-category $V(\K)$, we now introduce the notion of horizontal $T$-morphism between lax $V(T)$-algebras and that of a cell between such morphisms. Both are straightforward generalisations of notions that were introduced by Grandis and Par\'e in Section 7.1 of \cite{Grandis-Pare04}, where $V(T)$\ndash algebras were assumed to be strict. Afterwards, in \propref{double category of horizontal T-morphisms}, we will show that lax $V(T)$-algebras, weak $V(T)$\ndash morphisms, horizontal $T$-morphisms and their cells form a double category.
	\begin{definition} \label{definition: horizontal T-morphism}
		Let $T = (T, \mu, \eta)$ be a normal oplax double monad on a double category $\K$. Given lax $V(T)$\ndash algebras $A = (A, a, \bar a, \hat a)$ and $B = (B, b, \bar b, \hat b)$, a \emph{horizontal $T$\ndash morphism} $A \slashedrightarrow B$ is a horizontal morphism $\hmap JAB$ equipped with a \emph{structure cell}
		\begin{displaymath}
			\begin{tikzpicture}
				\matrix(m)[math175em]{TA & TB \\ A & B \\};
				\path[map]  (m-1-1) edge[barred] node[above] {$TJ$} (m-1-2)
														edge node[left] {$a$} (m-2-1)
										(m-1-2) edge node[right] {$b$} (m-2-2)
										(m-2-1) edge[barred] node[below] {$J$} (m-2-2);
				\path[transform canvas={shift={($(m-1-2)!(0,0)!(m-2-2)$)}}] (m-1-1) edge[cell] node[right] {$\bar J$} (m-2-1);			
			\end{tikzpicture}
		\end{displaymath}
		satisfying the following associativity and unit axioms.
		\begin{align*}
			\begin{tikzpicture}[textbaseline]
  	    \matrix(m)[math175em, ampersand replacement=\&]
  	    {	T^2 A \& T^2 B \& T^2 B \\
					TA \& TB \& TB \\
					A \& B \& B \\ };
  	    \path[map]	(m-1-1) edge[barred] node[above] {$T^2 J$} (m-1-2)
														edge node[left] {$Ta$} (m-2-1)
										(m-1-2) edge node[right] {$Tb$} (m-2-2)
										(m-1-3) edge node[right] {$\mu_B$} (m-2-3)
										(m-2-1) edge[barred] node[below] {$TJ$} (m-2-2)
														edge node[left] {$a$} (m-3-1)
										(m-2-2) edge node[right] {$b$} (m-3-2)
										(m-2-3) edge node[right] {$b$} (m-3-3)
										(m-3-1) edge[barred] node[below] {$J$} (m-3-2);
				\path				(m-1-2) edge[eq] (m-1-3)
										(m-3-2) edge[eq] (m-3-3);
				\path[transform canvas={shift={($(m-2-1)!0.5!(m-2-2)$)}}]
										(m-1-2) edge[cell] node[right] {$T\bar J$} (m-2-2);
				\path[transform canvas={shift={($(m-2-1)!0.5!(m-2-2)-(0,0.25em)$)}}]
										(m-2-2) edge[cell] node[right] {$\bar J$} (m-3-2);
				\path[transform canvas={shift={($(m-1-1)!0.5!(m-2-2)$)}}]
										(m-2-3) edge[cell] node[right] {$\bar b$} (m-3-3);
  	  \end{tikzpicture} \quad &= \quad \begin{tikzpicture}[textbaseline]
	      \matrix(m)[math175em, ampersand replacement=\&]
  	    {	T^2 A \& T^2 A \& T^2 B \\
					TA \& TA \& TB \\
					A \& A \& B \\ };
				\path[map]	(m-1-1) edge node[left] {$Ta$} (m-2-1)
										(m-1-2) edge[barred] node[above] {$T^2 J$} (m-1-3)
														edge node[left] {$\mu_A$} (m-2-2)
										(m-1-3) edge node[right] {$\mu_B$} (m-2-3)
										(m-2-1) edge node[left] {$a$} (m-3-1)
										(m-2-2) edge[barred] node[below] {$TJ$} (m-2-3)
														edge node[left] {$a$} (m-3-2)
										(m-2-3) edge node[right] {$b$} (m-3-3)
										(m-3-2) edge[barred] node[below] {$J$} (m-3-3);
				\path				(m-1-1) edge[eq] (m-1-2)
										(m-3-1) edge[eq] (m-3-2);
				\path[transform canvas={shift={($(m-1-1)!0.5!(m-2-2)$)}}]
										(m-2-2) edge[cell] node[right] {$\bar a$} (m-3-2);
				\path[transform canvas={shift={($(m-2-1)!0.5!(m-2-2)$)}}]
										(m-1-3) edge[cell] node[right] {$\mu_J$} (m-2-3);
				\path[transform canvas={shift={($(m-2-1)!0.5!(m-2-2)-(0,0.25em)$)}}]
										(m-2-3) edge[cell] node[right] {$\bar J$} (m-3-3);
  	  \end{tikzpicture} \\
  		\begin{tikzpicture}[textbaseline]
    	  \matrix(m)[math175em, ampersand replacement=\&]
    	  { A \& B \& B \\
    	  	\phantom{TA} \& \phantom{T^2B} \& TB \\
    	  	A \& B \& B \\};
				\path[map]	(m-1-1) edge[barred] node[above] {$J$} (m-1-2)
										(m-1-3) edge node[right] {$\eta_B$} (m-2-3)
										(m-2-3) edge node[right] {$b$} (m-3-3)
										(m-3-1) edge[barred] node[below] {$J$} (m-3-2);
				\path				(m-1-1) edge[eq] (m-3-1)
										(m-1-2) edge[eq] (m-1-3)
														edge[eq] (m-3-2)
										(m-3-2) edge[eq] (m-3-3);
				\path[transform canvas={shift={($(m-1-1)!0.5!(m-2-2)$)}}]
										(m-2-3) edge[cell] node[right] {$\hat b$} (m-3-3);
    	\end{tikzpicture} \quad &= \quad \begin{tikzpicture}[textbaseline]
  	    \matrix(m)[math175em, ampersand replacement=\&]
    	  { A \& A \& B \\
					\phantom{TA} \& TA \& TB \\
					A \& A \& B \\ };
				\path[map]	(m-1-2) edge[barred] node[above] {$J$} (m-1-3)
														edge node[left] {$\eta_A$} (m-2-2)
										(m-1-3) edge node[right] {$\eta_B$} (m-2-3)
										(m-2-2) edge[barred] node[below] {$TJ$} (m-2-3)
														edge node[left] {$a$} (m-3-2)
										(m-2-3) edge node[right] {$b$} (m-3-3)
										(m-3-2) edge[barred] node[below] {$J$} (m-3-3);
				\path				(m-1-1) edge[eq] (m-1-2)
														edge[eq] (m-3-1)
										(m-3-1)	edge[eq] (m-3-2);
				\path[transform canvas={shift={($(m-1-1)!0.5!(m-2-2)$)}}]
										(m-2-2) edge[cell] node[right] {$\hat a$} (m-3-2);
				\path[transform canvas={shift={($(m-2-1)!0.5!(m-2-2)$)}}]
										(m-1-3) edge[cell] node[right] {$\eta_J$} (m-2-3);
				\path[transform canvas={shift={($(m-2-1)!0.5!(m-2-2)-(0,0.25em)$)}}]
										(m-2-3) edge[cell] node[right] {$\bar J$} (m-3-3);
    	\end{tikzpicture}
		\end{align*} 
	\end{definition}
	\begin{example} \label{example: monoidal enriched profunctors}
		Consider the double monad $T$ of free strict monoidal $\V$-categories described in \exref{example: free strict monoidal enriched category monad}; we will call a horizontal $T$-morphism $A \brar B$ between lax monoidal $\V$-categories $A$ and $B$ a \emph{monoidal $\V$-profunctor}. It consists of a $\V$\ndash profunctor $\hmap JAB$ equipped with $\V$-maps
		\begin{displaymath}
			\map{J_\tens}{J(x_1, y_1) \tens \dotsb \tens J(x_n, y_n)}{J\bigpars{(x_1 \tens \dotsb \tens x_n), (y_1 \tens \dotsb \tens y_n)}}
		\end{displaymath}
		which are compatible with the actions of $TA$ and $TB$ and satisfy the following coherence conditions. The unit axiom states that all diagrams of the form
		\begin{displaymath}
			\begin{tikzpicture}
				\matrix(m)[math175em, column sep=0.5em]{J(x,y) & & J\bigpars{x, (y)} \\ & J\bigpars{(x), (y)} & \\};
				\path[map]	(m-1-1) edge[above] node {$\rho_{\mf i}$} (m-1-3)
														edge[below left] node{$J_\tens$} (m-2-2)
										(m-2-2) edge[below right] node{$\lambda_{\mf i}$} (m-1-3);
			\end{tikzpicture}
		\end{displaymath}
		commute, while the associativity axiom means that those of the form
		\begin{displaymath}
			\begin{tikzpicture}[font=\scriptsize]
				\matrix(l)[math, yshift=6em]{\bigpars{J(x_{11}, y_{11}) \tens \dotsb \tens J(x_{1m_1}, y_{1m_1})} \tens \dotsb \tens \bigpars{J(x_{n1}, y_{n1}) \tens \dotsb \tens J(x_{nm_n}, y_{nm_n})} \\};
				\matrix(m)[math175em, row sep=3em, column sep=0em, ampersand replacement=\&]{\begin{aligned} J\bigpars{(x_{11} &\tens \dotsb \tens x_{1m_1}), (y_{11} \tens \dotsb \tens y_{1m_1})} \tens \dotsb \\ &\tens J\bigpars{(x_{n1} \tens \dotsb \tens x_{nm_n}), (y_{n1} \tens \dotsb \tens y_{nm_n})}\end{aligned} \& J(x_{11}, y_{11}) \tens \dotsb \tens J(x_{nm_n}, y_{nm_n}) \\
				\begin{aligned} J\Bigl (\bigpars{(x_{11} \tens \dotsb \tens x_{1m_1}) &\tens \dotsb \tens (x_{n1} \tens \dotsb \tens x_{nm_n})}, \\ \bigpars{(y_{11} \tens \dotsb \tens y_{1m_1}) &\tens \dotsb \tens (y_{n1} \tens \dotsb \tens y_{nm_n})}\Bigr )\end{aligned} \& J\bigpars{(x_{11} \tens \dotsb \tens x_{nm_n}), (y_{11} \tens \dotsb \tens y_{nm_n})} \\};
				\matrix(n)[math175em, yshift=-6em]{J\Bigpars{\bigpars{(x_{11} \tens \dotsb \tens x_{1m_1}) \tens \dotsb \tens (x_{n1} \tens \dotsb \tens x_{nm_n})}, (y_{11} \tens \dotsb \tens y_{nm_n})} \\};
				\path[map]	(l-1-1) edge node[above left, yshift=-2pt] {$J_\tens \tens \dotsb \tens J_\tens$} ([yshift=12pt]m-1-1)
														edge[iso] node {$\iso$} (m-1-2)
										($(m-1-1)+(0,-13pt)$) edge[xshift=-17pt] node[left] {$J_\tens$} ([yshift=13pt]m-2-1)
										(m-1-2) edge node[right] {$J_\tens$} (m-2-2)
										($(m-2-1)+(0,-15pt)$) edge node[below left, yshift=2pt] {$\rho_{\mf a}$} (n-1-1)
										(m-2-2) edge node[below right, yshift=1pt] {$\lambda_{\mf a}$} (n-1-1);
			\end{tikzpicture}
		\end{displaymath}
		commute. If $A$ and $B$ are symmetric lax monoidal $\V$-categories then we call $J$ a \emph{symmetric} monoidal $\V$-profunctor whenever the diagrams below commute.
		\begin{displaymath}
			\begin{tikzpicture}
				\matrix(m)[math2em]{J(x_1, y_1) \tens \dotsb \tens J(x_n, y_n) & J(x_{\sigma 1}, y_{\sigma 1}) \tens \dotsb \tens J(x_{\sigma n}, y_{\sigma n}) \\
					J\bigpars{(x_1 \tens \dotsb \tens x_n), (y_1 \tens \dotsb \tens y_n)} & J\bigpars{(x_{\sigma 1} \tens \dotsb \tens x_{\sigma n}), (y_{\sigma 1} \tens \dotsb \tens y_{\sigma n})} \\};
				\matrix(n)[math, yshift=-5em]{J\bigpars{(x_1 \tens \dotsb \tens x_n), (y_{\sigma 1} \tens \dotsb \tens y_{\sigma n})} \\};
				\path[map]	(m-1-1) edge node[above] {$\sigma$} (m-1-2)
														edge node[left] {$J_\tens$} (m-2-1)
										(m-1-2) edge node[right] {$J_\tens$} (m-2-2)
										(m-2-1) edge node[below left, yshift=1pt] {$\rho_{\mf s_\sigma}$} (n-1-1)
										(m-2-2) edge node[below right, yshift=1pt] {$\lambda_{\mf s_\sigma}$} (n-1-1);
			\end{tikzpicture}
		\end{displaymath}
	\end{example}
	The following definition introduces general $T$-cells; we will see in \propref{double category of horizontal T-morphisms} that it generalises the vertical $V(T)$-cells of \defref{definition: vertical T-cell}.
	\begin{definition} \label{definition: T-cell}
		Let $T$ be a normal oplax double monad on a double category $\K$. Consider a cell $\phi$ as on the left below, where $f$ and $g$ are lax $V(T)$-morphisms while $J$ and $K$ are horizontal $T$-morphisms. We call $\phi$ a \emph{$T$-cell} whenever the identity on the right is satisfied.
		\begin{displaymath}
			\begin{tikzpicture}[textbaseline]
				\matrix(m)[math175em]{A & B \\ C & D \\};
				\path[map]  (m-1-1) edge[barred] node[above] {$J$} (m-1-2)
														edge node[left] {$f$} (m-2-1)
										(m-1-2) edge node[right] {$g$} (m-2-2)
										(m-2-1) edge[barred] node[below] {$K$} (m-2-2);
				\path[transform canvas={shift={($(m-1-2)!(0,0)!(m-2-2)$)}}] (m-1-1) edge[cell] node[right] {$\phi$} (m-2-1);			
			\end{tikzpicture} \qquad \qquad \begin{tikzpicture}[textbaseline]
				\matrix(m)[math175em]{TA & TB & TB \\ TC & TD & B \\ C & D & D \\};
				\path[map]	(m-1-1) edge[barred] node[above] {$TJ$} (m-1-2)
														edge node[left] {$Tf$} (m-2-1)
										(m-1-2) edge node[right] {$Tg$} (m-2-2)
										(m-1-3) edge node[right] {$b$} (m-2-3)
										(m-2-1) edge[barred] node[below] {$TK$} (m-2-2)
														edge node[left] {$c$} (m-3-1)
										(m-2-2) edge node[right] {$d$} (m-3-2)
										(m-2-3) edge node[right] {$g$} (m-3-3)
										(m-3-1) edge[barred] node[below] {$K$} (m-3-2);
				\path				(m-1-2) edge[eq] (m-1-3)
										(m-3-2) edge[eq] (m-3-3);
				\path[transform canvas={shift={($(m-2-2)!0.5!(m-2-3)$)}}] (m-1-1) edge[cell] node[right] {$T\phi$} (m-2-1)
										(m-2-1) edge[cell, transform canvas={yshift=-3pt}] node[right] {$\bar K$} (m-3-1);
				\path[transform canvas={shift={($(m-2-3)!0.5!(m-3-2)$)}}] (m-1-2) edge[cell] node[right] {$\bar g$} (m-2-2);
			\end{tikzpicture} = \begin{tikzpicture}[textbaseline]
				\matrix(m)[math175em]{TA & TA & TB \\ TC & A & B \\ C & C & D \\};
				\path[map]	(m-1-1) edge node[left] {$Tf$} (m-2-1)
										(m-1-2) edge[barred] node[above] {$TJ$} (m-1-3)
														edge node[left] {$a$} (m-2-2)
										(m-1-3) edge node[right] {$b$} (m-2-3)
										(m-2-1) edge node[left] {$c$} (m-3-1)
										(m-2-2) edge[barred] node[below] {$J$} (m-2-3)
														edge node[left] {$f$} (m-3-2)
										(m-2-3) edge node[right] {$g$} (m-3-3)
										(m-3-2) edge[barred] node[below] {$K$} (m-3-3);
				\path				(m-1-1) edge[eq] (m-1-2)
										(m-3-1) edge[eq] (m-3-2);
				\path[transform canvas={shift={($(m-2-3)!0.5!(m-3-2)$)}}]	(m-1-1) edge[cell] node[right] {$\bar f$} (m-2-1);
				\path[transform canvas={shift={($(m-2-2)!0.5!(m-2-3)$)}}] (m-1-2) edge[cell] node[right] {$\bar J$} (m-2-2)
										(m-2-2) edge[cell, transform canvas={yshift=-3pt}] node[right] {$\phi$} (m-3-2);
			\end{tikzpicture}
		\end{displaymath}
		Similarly, in the case that $f$ and $g$ are colax $V(T)$-morphisms, we call $\phi$ a $T$-cell whenever $\bar f \hc (\bar K \of T\phi) = (\phi \of \bar J) \hc \bar g$.
	\end{definition}
	
	\begin{example}
		Let $T$ be the double monad of `free strict monoidal categories' described in \exref{example: free strict monoidal enriched category monad}, and consider a transformation $\phi$ in $\enProf\V$, as on the left above, where $f$ and $g$ are lax monoidal $\V$-functors (\exref{example: monoidal enriched categories}) while $J$ and $K$ are monoidal $\V$-profunctors (\exref{example: monoidal enriched profunctors}). Remember that $\phi$ is given by a family of $\V$\ndash maps $\map\phi{J(x,y)}{K(fx, gy)}$ that is compatible with the actions of $A$ and $B$. The transformation $\phi$ is a $T$-cell if the following diagram commutes; in that case we call $\phi$ a \emph{monoidal transformation}.
		\begin{displaymath}
			\begin{tikzpicture}
				\matrix(m)[math2em, column sep=0.45em]{	J(x_1, y_1) \tens \dotsb \tens J(x_n, y_n) & K(fx_1, gy_1) \tens \dotsb \tens K(fx_n, gy_n) \\
					J\bigpars{(x_1 \tens \dotsb \tens x_n), (y_1 \tens \dotsb \tens y_n)} & K\bigpars{(fx_1 \tens \dotsb \tens fx_n), (gy_1 \tens \dotsb \tens gy_n)} \\
					K\bigpars{f(x_1 \tens \dotsb \tens x_n), g(y_1 \tens \dotsb \tens y_n)} & K\bigpars{(fx_1 \tens \dotsb \tens fx_n), g(y_1 \tens \dotsb \tens y_n)} \\ };
				\path[map]	(m-1-1) edge node[above] {$\phi \tens \dotsb \tens \phi$} (m-1-2)
														edge node[left] {$J_\tens$} (m-2-1)
										(m-1-2) edge node[right] {$K_\tens$} (m-2-2)
										(m-2-1) edge node[left] {$\phi$} (m-3-1)
										(m-2-2) edge node[right] {$\rho_{g_\tens}$} (m-3-2)
										(m-3-1) edge node[below] {$\lambda_{f_\tens}$} (m-3-2);
			\end{tikzpicture}
		\end{displaymath}
		A monoidal transformation $\nat\phi J{K(f, g)}$ where $f$ and $g$ are colax monoidal functors is defined similarly.
	\end{example}
	
	Given a normal oplax double monad $T$, we now show that lax $V(T)$-algebras, any choice of weak $V(T)$-morphisms, horizontal $T$-morphisms and $T$-cells form a double category $\Alg wT$, such that $V(\Alg wT) = \Alg w{V(T)}$. Consequently we will call lax $V(T)$\ndash algebras simply \emph{lax $T$-algebras}, while weak $V(T)$-morphisms will be called \emph{weak vertical $T$-morphisms}.
	\begin{proposition} \label{double category of horizontal T-morphisms}
		Let $T$ be a normal oplax double monad on a double category $\K$, and let `weak' mean either `colax' or `lax'. The double category structure of $\K$ lifts to make lax $V(T)$-algebras, weak $V(T)$-morphisms, horizontal $T$-morphisms and $T$-cells into a double category $\Alg wT$, such that $V(\Alg wT) = \Alg w{V(T)}$.
	\end{proposition}
	We denote by $\Alg{ps}T$ the sub-double category of $\Alg lT$ consisting of lax $T$\ndash algebras, pseudo $V(T)$-morphisms, horizontal $T$-morphisms and all $T$-cells whose vertical source and target are pseudo $V(T)$-morphisms. Clearly $V(\Alg lT) = \Alg l{V(T)}$ implies that $V(\Alg{ps}T) = \Alg{ps}{V(T)}$.
	\begin{proof}
		We shall show that the double category structure of $\K$ lifts to make lax $V(T)$\ndash al\-ge\-bras, lax $V(T)$-morphisms, horizontal $T$-morphisms and $T$-cells into a double category $\Alg lT$; the case of $\Alg cT$ is similar. First notice that the equality $V(\Alg wT) = \Alg w{V(T)}$ determines the structure on the category $(\Alg wT)_\textup v$ completely; in particular the algebra structure on the vertical composite $h \of f$ of lax $V(T)$-morphisms $\map fAC$ and $\map hCE$ must be given by the composite $\ol{h \of f} = (\bar h \of 1_{Tf}) \hc (1_h \of \bar f)$. 
		
		To also lift the horizontal structure of $\K$ to $\Alg lT$ it remains to define structure cells for each composite $J \hc H$, of horizontal $T$-morphisms $\hmap JAB$ and $\hmap HBC$, as well as for the horizontal units $\hmap{1_A}AA$, in a way such that any composite of $T$\ndash cells forms again a $T$-cell, the horizontal units $1_f$ form $T$-cells, and the associators and unitors of $\K$ form $T$-cells. We define the structure cell of $J \hc H$ by
		\begin{equation} \label{structure cell of horizontal composite}
			\ol{J \hc H} = \begin{tikzpicture}[textbaseline]
				\matrix(m)[math175em]{ TA & & TC \\ TA & TB & TC \\ A & B & C. \\};
				\path[map]	(m-1-1) edge[barred] node[above] {$T(J \hc H)$} (m-1-3)
										(m-2-1) edge[barred] node[above] {$TJ$} (m-2-2)
														edge node[left] {$a$} (m-3-1)
										(m-2-2) edge[barred] node[above] {$TH$} (m-2-3)
														edge node[right] {$b$} (m-3-2)
										(m-2-3) edge node[right] {$c$} (m-3-3)
										(m-3-1) edge[barred] node[below] {$J$} (m-3-2)
										(m-3-2) edge[barred] node[below] {$H$} (m-3-3)
										(m-1-2) edge[cell] node[right] {$T_\hc$} (m-2-2);
				\path				(m-1-1) edge[eq] (m-2-1)
										(m-1-3) edge[eq] (m-2-3);
				\path[transform canvas={shift={($(m-2-2)!0.5!(m-2-3)$)}}]
										(m-2-1) edge[cell] node[right] {$\bar J$} (m-3-1)
										(m-2-2) edge[cell] node[right] {$\bar H$} (m-3-2);
			\end{tikzpicture}
		\end{equation}
		That this composite satisfies the associativity axiom of \defref{definition: horizontal T-morphism} is shown by the following equation of composites where, to save space, only the non-identity cells are depicted while objects and morphisms are left out. The cells $\bar c$ and $\bar a$ here are the structure cells of $C$ and $A$ respectively (see \defref{definition: lax algebra}) and the identities follow from the naturality of $T_\hc$; the associativity axioms of $H$ and $J$ and the fact that \mbox{$(T^2)_\hc = T_\hc \of TT_\hc$}; the composition axiom of $\mu$ (see \defref{definition: transformation}). The unit axiom for $\ol{J \hc H}$ follows similarly from that for $\bar H$ and $\bar J$, as well as the composition axiom for $\eta$.
		\begin{displaymath}
			\begin{tikzpicture}[textbaseline, x=2em, y=1.75em, font=\scriptsize]
				\draw	(2,0) -- (2,4) -- (0,4) -- (0,0) -- (3,0) -- (3,3) -- (0,3)
							(0,2) -- (2,2)
							(0,1) -- (2,1)
							(1,0) -- (1,1);
				\draw[shift={(0.5,0.5)}]
							(0,0) node {$\bar J$}
							(1,0) node {$\bar H$}
							(2,1) node {$\bar c$};
				\draw[shift={(0,0.5)}]
							(1,1) node {$T_\hc$}
							(1,2) node {$T(\bar J \hc \bar H)$}
							(1,3) node {$TT_\hc$};
			\end{tikzpicture} \mspace{20mu}=\mspace{20mu} \begin{tikzpicture}[textbaseline, x=2em, y=1.75em, font=\scriptsize]
				\draw	(2,0) -- (2,4) -- (0,4) -- (0,0) -- (3,0) -- (3,2) -- (0,2)
							(0,3) -- (2,3)
							(0,1) -- (2,1)
							(1,0) -- (1,2);
				\draw[shift={(0.5,0.5)}]
							(0,0) node {$\bar J$}
							(1,0) node {$\bar H$}
							(0,1) node {$T\bar J$}
							(1,1) node {$T\bar H$};
				\draw[shift={(0.5,0)}]
							(2,1) node {$\bar c$};
				\draw[shift={(0,0.5)}]
							(1,2) node {$T_\hc$}
							(1,3) node {$TT_\hc$};
			\end{tikzpicture} \mspace{20mu}=\mspace{20mu} \begin{tikzpicture}[textbaseline, x=2em, y=1.75em, font=\scriptsize]
				\draw	(3,2) -- (0,2) -- (0,0) -- (3,0) -- (3,3) -- (1,3) -- (1,0)
							(1,1) -- (3,1)
							(2,0) -- (2,2);
				\draw[shift={(0.5,0.5)}]
							(1,0) node {$\bar J$}
							(2,0) node {$\bar H$}
							(1,1) node {$\mu_J$}
							(2,1) node {$\mu_H$};
				\draw	(0.5,1) node {$\bar a$}
							(2,2.5) node {$(T^2)_\hc$};
			\end{tikzpicture} \mspace{20mu}=\mspace{20mu} \begin{tikzpicture}[textbaseline, x=2em, y=1.75em, font=\scriptsize]
				\draw (1,3) -- (1,0) -- (3,0) -- (3,3) -- (0,3) -- (0,0) -- (1,0)
							(1,2) -- (3,2)
							(1,1) -- (3,1)
							(2,0) -- (2,1);
				\draw[shift={(0.5,0.5)}]
							(1,0) node {$\bar J$}
							(2,0) node {$\bar H$}
							(0,1) node {$\bar a$};
				\draw[shift={(0,0.5)}]
							(2,1) node {$T_\hc$}
							(2,2) node {$\mu_{J \hc H}$};
			\end{tikzpicture}
		\end{displaymath}
		The  structure cell of a horizontal unit $1_A$ we take to be $T1_A = 1_{TA} \xRar{1_a} 1_A$, where \mbox{$\map a{TA}A$} is the structure map of $A$. The associativity and unit axioms follow from the unit axioms for $\mu$ and $\eta$.
		
		It is now easily checked that vertical and horizontal composites of $T$-cells form again $T$-cells and that, for a lax $T$-morphism $\map fAC$, the unit cell $1_f$ is a $T$-cell. It is also readily seen that the components of the associator and unitors of $\K$ form $T$-cells, using the associativity and unit axioms for $T$; we conclude that, with the chosen algebra structures above, the structure of double category $\K$ lifts to form a double category $\Alg lT$.
		
		Finally notice that the $T$-cell axiom for a vertical cell $\cell\phi fg$ coincides with the $V(T)$-cell axiom for $\phi$ (see \defref{definition: vertical T-cell}). We conclude that $V(\Alg lT) = \Alg l{V(T)}$ as asserted.
	\end{proof}
	
	\section{Kan extensions in \texorpdfstring{$\Alg wT$}{Alg(T)}}
	In this section we study Kan extensions in the double category $\Alg wT$. Firstly, in light of \propref{right Kan extensions along conjoints as right Kan extensions in V(K)}, we wonder whether $\Alg wT$ has restrictions, so that Kan extensions in $\Alg w{V(T)} = V(\Alg wT)$ correspond to Kan extensions in $\Alg wT$. Secondly we wonder whether $\Alg wT$ has opcartesian tabulations, so that we can obtain analogues of \thmref{pointwise Kan extensions in terms of Kan extensions} and \propref{pointwise right Kan extensions along conjoints as pointwise right Kan extensions in V(K)} for $\Alg wT$. As we shall see, restrictions and tabulations can often be `lifted' along the forgetful double functors $\Alg wT \to \K$.
	
	We start with the lifting of restrictions. Given a double functor $\map F\K\L$ and morphisms $\map fAC$, $\hmap KCD$ and $\map gBD$ in $\K$, we say that $F$ \emph{lifts the restriction $K(f, g)$} if for any cartesian cell $\phi$ in $\L$, as on the left below, there exists a unique cell $\phi'$ in $\K$, as on the right, such that $F\phi' = \phi$ and, moreover, $\phi'$ is cartesian.
	\begin{displaymath}
		\begin{tikzpicture}
			\matrix(m)[math175em]{FA & FB \\ FC & FD \\};
			\path[map]  (m-1-1) edge[barred] node[above] {$J$} (m-1-2)
													edge node[left] {$Ff$} (m-2-1)
									(m-1-2) edge node[right] {$Fg$} (m-2-2)
									(m-2-1) edge[barred] node[below] {$FK$} (m-2-2);
			\path[transform canvas={shift={($(m-1-2)!(0,0)!(m-2-2)$)}}] (m-1-1) edge[cell] node[right] {$\phi$} (m-2-1);			
		\end{tikzpicture} \qquad\qquad\qquad\qquad \begin{tikzpicture}
			\matrix(m)[math175em]{A & B \\ C & D \\};
			\path[map]  (m-1-1) edge[barred] node[above] {$J'$} (m-1-2)
													edge node[left] {$f$} (m-2-1)
									(m-1-2) edge node[right] {$g$} (m-2-2)
									(m-2-1) edge[barred] node[below] {$K$} (m-2-2);
			\path[transform canvas={shift={($(m-1-2)!(0,0)!(m-2-2)$)}}] (m-1-1) edge[cell] node[right] {$\phi'$} (m-2-1);
		\end{tikzpicture}
	\end{displaymath}
	\begin{proposition} \label{lifting restrictions}
		For a normal oplax double monad $T$ on a double category $\K$ the following hold for the forgetful double functors $\Alg wT \to \K$, where $\textup w \in \set{\textup c, \textup l, \textup{ps}}$:
		\begin{enumerate}[label=\textup{(\alph*)}]
			\item $\Alg cT \to \K$ lifts restrictions $K(f, g)$ where $g$ is a pseudo $T$-morphism;
			\item $\Alg lT \to \K$ lifts restrictions $K(f, g)$ where $f$ is a pseudo $T$-morphism;
			\item $\Alg{ps}T \to \K$ lifts all restrictions.
		\end{enumerate}
	\end{proposition}
	We conclude that if $\K$ is an equipment then so is $\Alg{ps}T$. Moreover, in that case $\Alg cT$ has all companions while, in general, it has only conjoints of pseudo $T$-morphisms; dually $\Alg lT$ has all conjoints but only companions of pseudo $T$\ndash morphisms.
	\begin{example}
		If $\map fAC$ is a monoidal $\V$-functor (\exref{example: monoidal enriched categories}), $\hmap KCD$ a monoidal $\V$-profunctor (\exref{example: monoidal enriched profunctors}) and $\map gBD$ a lax monoidal $\V$-functor then the restriction $K(f,g)$ admits a canonical monoidal structure that is given by the composites
		\begin{multline*}
			K(fx_1, gy_1) \tens \dotsb \tens K(fx_n, gy_n) \xrar{K_\tens} K\bigpars{(fx_1 \tens \dotsb \tens fx_n), (gy_1 \tens \dotsb \tens gy_n)} \\
			\xrar{\rho_{g_\tens} \of \lambda_{\inv f_\tens}} K\bigpars{f(x_1 \tens \dotsb \tens x_n), g(y_1 \tens \dotsb \tens y_n)}.
		\end{multline*}
	\end{example}
	\begin{proof}[of \propref{lifting restrictions}.]
		We will prove (b). The inclusion $\Alg{ps}T \to \Alg lT$ lifts all restrictions so that (c) follows, while proving (a) is similar to proving (b).
		
		Consider a cartesian cell $\phi$ in $\K$ as on the left below, where $f$ is a pseudo $T$\ndash morphism, $g$ is a lax $T$-morphism and $K$ is a horizontal $T$-morphism. We claim that the unique factorisation $\bar J$ of the composite on the right below, through the cartesian cell $\phi$, as shown, forms a structure cell for $J$.
		\begin{equation} \label{restriction structure cell}
			\begin{tikzpicture}[textbaseline]
				\matrix(m)[math175em]{A & B \\ C & D \\};
				\path[map]  (m-1-1) edge[barred] node[above] {$J$} (m-1-2)
														edge node[left] {$f$} (m-2-1)
										(m-1-2) edge node[right] {$g$} (m-2-2)
										(m-2-1) edge[barred] node[below] {$K$} (m-2-2);
				\path[transform canvas={shift={($(m-1-2)!(0,0)!(m-2-2)$)}}] (m-1-1) edge[cell] node[right] {$\phi$} (m-2-1);
			\end{tikzpicture} \qquad\quad \begin{tikzpicture}[textbaseline]
				\matrix(m)[math175em]{TA & TA & TB & TB \\ A & TC & TD & B \\ C & C & D & D \\};
				\path[map]	(m-1-1) edge node[left] {$a$} (m-2-1)
										(m-1-2) edge[barred] node[above] {$TJ$} (m-1-3)
														edge node[left] {$Tf$} (m-2-2)
										(m-1-3) edge node[right] {$Tg$} (m-2-3)
										(m-1-4) edge node[right] {$b$} (m-2-4)
										(m-2-1) edge node[left] {$f$} (m-3-1)
										(m-2-2) edge[barred] node[below] {$TK$} (m-2-3)
														edge node[left] {$c$} (m-3-2)
										(m-2-3) edge node[right] {$d$} (m-3-3)
										(m-2-4) edge node[right] {$g$} (m-3-4)
										(m-3-2) edge[barred] node[below] {$K$} (m-3-3);
				\path				(m-1-1) edge[eq] (m-1-2)
										(m-1-3) edge[eq] (m-1-4)
										(m-3-1) edge[eq] (m-3-2)
										(m-3-3) edge[eq] (m-3-4);
				\path[transform canvas={shift={($(m-2-3)!0.5!(m-3-3)$)}}]	(m-1-1) edge[cell, transform canvas={xshift=-8pt}] node[right] {$\inv{\bar f}$} (m-2-1)
										(m-1-3) edge[cell] node[right] {$\bar g$} (m-2-3);
				\path[transform canvas={shift=(m-2-3)}]	(m-1-2) edge[cell] node[right] {$T\phi$} (m-2-2)
										(m-2-2) edge[cell, transform canvas={yshift=-3pt}] node[right] {$\bar K$} (m-3-2);
			\end{tikzpicture} = \begin{tikzpicture}[textbaseline]
				\matrix(m)[math175em]{TA & TB \\ A & B \\ C & D \\};
				\path[map]	(m-1-1) edge[barred] node[above] {$TJ$} (m-1-2)
														edge node[left] {$a$} (m-2-1)
										(m-1-2) edge node[right] {$b$} (m-2-2)
										(m-2-1) edge[barred] node[below] {$J$} (m-2-2)
														edge node[left] {$f$} (m-3-1)
										(m-2-2) edge node[right] {$g$} (m-3-2)
										(m-3-1) edge[barred] node[below] {$K$} (m-3-2);
				\path[transform canvas={shift=(m-2-2)}]	(m-1-1) edge[cell] node[right] {$\bar J$} (m-2-1)
										(m-2-1) edge[cell, transform canvas={yshift=-3pt}] node[right] {$\phi$} (m-3-1);
			\end{tikzpicture}
		\end{equation}
		We have to show that $\bar J$ satisfies the associativity and unit axiom of \defref{definition: horizontal T-morphism}. To show that the first holds consider the following equation of composites in $\K$, whose identities are described below. Since its left-hand and right-hand side are the left-hand and right-hand side of the associativity axiom for $\bar J$ postcomposed with $\phi$, it implies the associativity axiom for $\bar J$, because factorisations through cartesian cells are unique.
		\begin{align*}
			\begin{tikzpicture}[textbaseline, x=2em, y=1.75em, font=\scriptsize]
				\draw	(1,3) -- (1,0) -- (0,0) -- (0,3) -- (2,3) -- (2,1) -- (0,1)
							(0,2) -- (1,2);
				\draw[shift={(0.5,0.5)}]
							(0,0) node {$\phi$}
							(0,1) node {$\bar J$}
							(0,2) node {$T\bar J$};
				\draw (1.5,2) node {$\bar b$};
			\end{tikzpicture} \mspace{15mu} \overset{(\textup i)} = \mspace{15mu} &\begin{tikzpicture}[textbaseline, x=2em, y=1.75em, font=\scriptsize]
				\draw (3,1) -- (4,1) -- (4,3) -- (3,3) -- (3,0) -- (0,0) -- (0,2) -- (3,2)
							(1,0) -- (1,3) -- (2,3) -- (2,0)
							(1,1) -- (2,1);
				\draw[shift={(0.5,0.5)}]
							(1,0) node {$\bar K$}
							(1,1) node {$T\phi$}
							(1,2) node {$T\bar J$};
				\draw[shift={(0.5,0)}]
							(0,1) node {$\inv{\bar f}$}
							(2,1) node {$\bar g$}
							(3,2) node {$\bar b$};
			\end{tikzpicture} \mspace{15mu} \overset{(\textup{ii})} = \mspace{15mu} \begin{tikzpicture}[textbaseline, x=2em, y=1.75em, font=\scriptsize]
				\draw (1,1) -- (4,1) -- (4,3) -- (1,3) -- (1,0) -- (0,0) -- (0,2) -- (1,2)
							(4,1) -- (4,0) -- (5,0) -- (5,3) -- (6,3) -- (6,1) -- (5,1)
							(2,3) -- (2,0) -- (3,0) -- (3,3)
							(2,2) -- (3,2)
							(4,2) -- (5,2);
				\draw[shift={(0.5,0)}]
							(0,1) node {$\inv{\bar f}$}
							(1,2) node {$T\inv{\bar f}$}
							(3,2) node {$T\bar g$}
							(4,1) node {$\bar g$}
							(5,2) node {$\bar b$};
				\draw[shift={(0.5,0.5)}]
							(2,0) node {$\bar K$}
							(2,1) node {$T\bar K$}
							(2,2) node {$T^2\phi$};
			\end{tikzpicture} \\[1em]
			\overset{(\textup{iii})} = \mspace{11.5mu} &\begin{tikzpicture}[textbaseline, x=2em, y=1.75em, font=\scriptsize]
				\draw (2,2) -- (4,2) -- (4,0) -- (2,0) -- (2,3) -- (3,3) -- (3,0)
							(1,2) -- (0,2) -- (0,0) -- (1,0) -- (1,3) -- (2,3)
							(1,1) -- (3,1)
							(4,1) -- (5,1)
							(4,2) -- (4,3) -- (5,3) -- (5,0) -- (6,0) -- (6,2) -- (5,2);
				\draw[shift={(0.5,0)}]
							(0,1) node {$\inv{\bar f}$}
							(1,2) node {$T\inv{\bar f}$}
							(3,1) node {$\bar d$}
							(5,1) node {$\bar g$};
				\draw[shift={(0.5,0.5)}]
							(2,0) node {$\bar K$}
							(2,1) node {$T\bar K$}
							(2,2) node {$T^2\phi$};
			\end{tikzpicture} \mspace{15mu} \overset{(\textup{iv})} = \mspace{15mu} \begin{tikzpicture}[textbaseline, x=2em, y=1.75em, font=\scriptsize]
				\draw	(4,3) -- (4,0) -- (3,0) -- (3,3) -- (5,3) -- (5,1) -- (3,1)
							(1,2) -- (0,2) -- (0,0) -- (1,0) -- (1,3) -- (2,3) -- (2,1) -- (1,1)
							(2,2) -- (4,2)
							(2,1) -- (2,0) -- (3,0)
							(5,2) -- (6,2) -- (6,0) -- (5,0) -- (5,1);
				\draw[shift={(0.5,0)}]
							(0,1) node {$\inv{\bar f}$}
							(1,2) node {$T\inv{\bar f}$}
							(2,1) node {$\bar c$}
							(5,1) node {$\bar g$};
				\draw[shift={(0.5,0.5)}]
							(3,0) node {$\bar K$}
							(3,1) node {$\mu_K$}
							(3,2) node {$T^2\phi$};
			\end{tikzpicture} \\[1em]
			\overset{(\textup v)} = \mspace{13.5mu} &\begin{tikzpicture}[textbaseline, x=2em, y=1.75em, font=\scriptsize]
				\draw (1,2) -- (4,2) -- (4,0) -- (1,0) -- (1,3) -- (0,3) -- (0,1) -- (1,1)
							(2,0) -- (2,3) -- (3,3) -- (3,0)
							(2,1) -- (3,1);
				\draw[shift={(0.5,0)}]
							(0,2) node {$\bar a$}
							(1,1) node {$\inv{\bar f}$}
							(3,1) node {$\bar g$};
				\draw[shift={(0.5,0.5)}]
							(2,0) node {$\bar K$}
							(2,1) node {$T\phi$}
							(2,2) node {$\mu_J$};
			\end{tikzpicture} \mspace{15mu} \overset{(\textup{vi})} = \mspace{15mu} \begin{tikzpicture}[textbaseline, x=2em, y=1.75em, font=\scriptsize]
				\draw	(1,3) -- (1,0) -- (2,0) -- (2,3) -- (0,3) -- (0,1) -- (2,1)
							(1,2) -- (2,2)
							(0.5, 2) node {$\bar a$};
				\draw[shift={(0.5,0.5)}]
							(1,0) node {$\phi$}
							(1,1) node {$\bar J$}
							(1,2) node {$\mu_J$};
			\end{tikzpicture}
		\end{align*}
		The identity (i) above follows from the factorisation \eqref{restriction structure cell}; (ii) follows from the $T$\ndash image of \eqref{restriction structure cell}, where the $T$-image of the left-hand side is rewritten as $T\inv{\bar f} \hc T(\bar K \of T\phi) \hc T\bar g$ by using the unit axioms for $T$ (see \eqref{unit axiom for normal oplax double functors}); (iii) follows from the associativity axiom for $\bar g$; (iv) from the associativity axiom for $\bar K$; (v) from the naturality of $\mu$ and the associativity axiom for $\bar f$; (vi) from the factorisation \eqref{restriction structure cell}. Showing that $\bar J$ satisfies the unit axiom is done analogously: it is equivalent to the identity obtained by composing it with $\phi$, which follows from the factorisation \eqref{restriction structure cell} and the unit axioms for $\bar g$, $\bar K$ and $\bar f$. This completes the proof that $\bar J$ forms a structure cell making $J$ into a horizontal $T$-morphism. 
		
		Composing \eqref{restriction structure cell} on the left with $\bar f$ we see that $\phi$ forms a $T$-cell \mbox{$(J, \bar J) \Rar (K, \bar K)$}, showing that $\phi$ lifts along $\Alg lT \to \K$. In fact, because the factorisation \eqref{restriction structure cell} is unique, this lift of $\phi$ is unique as well. To complete the proof we have to show that $\phi$ forms a cartesian $T$-cell. Hence consider a $T$-cell $\psi$ as on the left-hand side below, where $h$ and $k$ are lax $T$-morphisms while $H$ is a horizontal $T$-morphism.
		\begin{displaymath}
			\begin{tikzpicture}[textbaseline]
    		\matrix(m)[math175em]{X & Y \\ A & B \\ C & D \\};
    		\path[map]  (m-1-1) edge[barred] node[above] {$H$} (m-1-2)
        		                edge node[left] {$h$} (m-2-1)
        		        (m-1-2) edge node[right] {$k$} (m-2-2)
        		        (m-2-1) edge node[left] {$f$} (m-3-1)
        		        (m-2-2) edge node[right] {$g$} (m-3-2)
        		        (m-3-1) edge[barred] node[below] {$K$} (m-3-2);
    		\path[transform canvas={shift={($(m-2-1)!0.5!(m-1-1)$)}}] (m-2-2) edge[cell] node[right] {$\psi$} (m-3-2);
  		\end{tikzpicture} = \begin{tikzpicture}[textbaseline]
    		\matrix(m)[math175em]{X & Y \\ A & B \\ C & D \\};
    		\path[map]  (m-1-1) edge[barred] node[above] {$H$} (m-1-2)
        		                edge node[left] {$h$} (m-2-1)
            		    (m-1-2) edge node[right] {$k$} (m-2-2)
            		    (m-2-1) edge[barred] node[below] {$J$} (m-2-2)
            		            edge node[left] {$f$} (m-3-1)
            		    (m-2-2) edge node[right] {$g$} (m-3-2)
            		    (m-3-1) edge[barred] node[below] {$K$} (m-3-2);
    		\path[transform canvas={shift=(m-2-1))}]
        		        (m-1-2) edge[cell] node[right] {$\psi'$} (m-2-2)
        		        (m-2-2) edge[transform canvas={yshift=-0.3em}, cell] node[right] {$\phi$} (m-3-2);
  		\end{tikzpicture}
		\end{displaymath}
		Since $\phi$ is cartesian in $\K$ we obtain a unique factorisation $\psi'$ of $\psi$ as shown; we have to prove that $\psi'$ is a $T$-cell. The following equation shows that the $T$-cell axiom for $\psi'$ holds after composing it with $\phi$, where the identities follow from \eqref{restriction structure cell}, the factorisation above, the $T$-cell axiom for $\psi$ and again the factorisation above.
		\begin{displaymath}
			\begin{tikzpicture}[textbaseline, x=2em, y=1.75em, font=\scriptsize]
				\draw	(0,1) -- (2,1) -- (2,3) -- (0,3) -- (0,0) -- (1,0) -- (1,3)
							(0,2) -- (1,2)
							(1.5,2) node {$\bar k$};
				\draw[shift={(0.5,0.5)}]
							(0,0) node {$\phi$}
							(0,1) node {$\bar J$}
							(0,2) node {$T\psi'$};
			\end{tikzpicture} \mspace{8mu} = \mspace{8mu} \begin{tikzpicture}[textbaseline, x=2em, y=1.75em, font=\scriptsize]
				\draw (3,2) -- (0,2) -- (0,0) -- (3,0) -- (3,3) -- (4,3) -- (4,1) -- (3,1)
							(1,0) -- (1,3) -- (2,3) -- (2,0)
							(1,1) -- (2,1);
				\draw[shift={(0.5,0)}]
							(0,1) node {$\inv{\bar f}$}
							(2,1) node {$\bar g$}
							(3,2) node {$\bar k$};
				\draw[shift={(0.5,0.5)}]
							(1,0) node {$\bar K$}
							(1,1) node {$T\phi$}
							(1,2) node {$T\psi'$};
			\end{tikzpicture} \mspace{8mu} = \mspace{8mu} \begin{tikzpicture}[textbaseline, x=2em, y=1.75em, font=\scriptsize]
				\draw (1,2) -- (0,2) -- (0,0) -- (3,0) -- (3,2) -- (2,2)
							(1,0) -- (1,3) -- (2,3) -- (2,0)
							(1,1) -- (2,1)
							(3,2) -- (3,3) -- (4,3) -- (4,1) -- (3,1)
							(1.5,0.5) node {$\bar K$};
				\draw[shift={(0.5,0)}]
							(0,1) node {$\inv{\bar f}$}
							(1,2) node {$T\psi$}
							(2,1) node {$\bar g$}
							(3,2) node {$\bar k$};
			\end{tikzpicture} \mspace{8mu} = \mspace{8mu} \begin{tikzpicture}[textbaseline, x=2em, y=1.75em, font=\scriptsize]
				\draw (1,1) -- (0,1) -- (0,3) -- (2,3) -- (2,0) -- (1,0) -- (1,3)
							(1,2) -- (2,2)
							(1.5,2.5) node {$\bar J$};
				\draw[shift={(0.5,0)}]
							(0,2) node {$\bar h$}
							(1,1) node {$\psi$};
			\end{tikzpicture} \mspace{8mu} = \mspace{8mu} \begin{tikzpicture}[textbaseline, x=2em, y=1.75em, font=\scriptsize]
				\draw (1,3) -- (1,0) -- (2,0) -- (2,3) -- (0,3) -- (0,1) -- (2,1)
							(1,2) -- (2,2)
							(0.5,2) node {$\bar h$};
				\draw[shift={(0.5,0.5)}]
							(1,0) node {$\phi$}
							(1,1) node {$\psi'$}
							(1,2) node {$\bar J$};
			\end{tikzpicture}
		\end{displaymath}
		Since factorisations through $\phi$ are unique the $T$-cell axiom for $\psi'$ follows, completing the proof.
	\end{proof}
	
	We now turn to tabulations, see \defref{definition: tabulation}.
	\begin{proposition} \label{lifting tabulations}
		Let $T$ be a normal oplax double monad on an equipment $\K$. The forgetful functors $\Alg wT \to \K$ lift tabulations in the following sense. If the cell $\pi$ below forms the tabulation of $J$ in $\K$, where $J$ is a horizontal $T$-morphism, then there exists a unique lax $T$-algebra structure on $\gen J$ with respect to which $\pi_A$ and $\pi_B$ are strict $T$\ndash morphisms and $\pi$ is a $T$-cell. Thus a $T$-cell, $\pi$ forms the tabulation of $J$ in $\Alg wT$, which is opcartesian as soon as $\pi$ is opcartesian in $\K$.
		\begin{displaymath}
			\begin{tikzpicture}
  			\matrix(m)[tab]{\tab J & \tab J \\ A & B \\};
  			\path[map]  (m-1-1) edge node[left] {$\pi_A$} (m-2-1)
										(m-1-2) edge node[right] {$\pi_B$} (m-2-2)
										(m-2-1) edge[barred] node[below] {$J$} (m-2-2);
				\path				(m-1-1) edge[eq] (m-1-2);
				\path[transform canvas={shift={($(m-1-2)!(0,0)!(m-2-2)$)}}] (m-1-1) edge[cell] node[right] {$\pi$} (m-2-1);
			\end{tikzpicture}
		\end{displaymath}
	\end{proposition}
	\begin{example}
		The tabulation $\tab J$ (see \exref{example: tabulations of profunctors}) of a monoidal (unenriched) profunctor $\hmap JAB$ (see \exref{example: monoidal enriched profunctors}) admits a canonical lax monoidal structure whose tensor product on objects and morphisms is given by applying the tensor products of $A$, $J$ and $B$ coordinatewise:
		\begin{align*}
			\bigpars{(x_1, u_1, y_1) \tens \dotsb \tens (x_n, u_n, y_n)} &= \bigpars{(x_1 \tens \dotsb \tens x_n), (u_1 \tens \dotsb \tens u_n), (y_1 \tens \dotsb \tens y_n)}; \\
			\bigpars{(p_1, q_1) \tens \dotsb \tens (p_n, q_n)} &= \bigpars{(p_1 \tens \dotsb \tens p_n), (q_1 \tens \dotsb \tens q_n)}.
		\end{align*}
		Similarly its coherence maps are given by those of $A$ and $B$: $\mf a_{(\ull x, \ull u, \ull y)} = (\mf a_{\ull x}, \mf a_{\ull y})$ for any double sequence $(\ull x, \ull u, \ull y)$ of objects in $\tab J$, and $\mf i_{(x, u, y)} = (\mf i_x, \mf i_y)$ for any object $(x,u,y)$ in $\tab J$.
	\end{example}
	For the proof of \propref{lifting tabulations} it is useful to record the following simple property of tabulations.
	\begin{lemma}
		The projections $\pi_A$ and $\pi_B$ of a tabulation $\pi$, as on the left below, are jointly monic in the following sense. Any two cells $\phi$ and $\psi$ as on the right below coincide as soon as $1_{\pi_A} \of \phi = 1_{\pi_A} \of \psi$ and $1_{\pi_B} \of \phi = 1_{\pi_B} \of \psi$.
		\begin{displaymath}
			\begin{tikzpicture}[textbaseline]
  			\matrix(m)[tab]{\tab J & \tab J \\ A & B \\};
  			\path[map]  (m-1-1) edge node[left] {$\pi_A$} (m-2-1)
										(m-1-2) edge node[right] {$\pi_B$} (m-2-2)
										(m-2-1) edge[barred] node[below] {$J$} (m-2-2);
				\path				(m-1-1) edge[eq] (m-1-2);
				\path[transform canvas={shift={($(m-1-2)!(0,0)!(m-2-2)$)}}] (m-1-1) edge[cell] node[right] {$\pi$} (m-2-1);
			\end{tikzpicture} \qquad\qquad \begin{tikzpicture}[textbaseline]
				\matrix(m)[tab]{X & Y \\ \gen J & \gen J \\};
				\path[map]	(m-1-1) edge[barred] node[above] {$H$} (m-1-2)
														edge node[left] {$f$} (m-2-1)
										(m-1-2) edge node[right] {$g$} (m-2-2);
				\path				(m-2-1) edge[eq] (m-2-2);
				\path[transform canvas={shift={($(m-1-2)!0.5!(m-2-2)$)}}]	(m-1-1) edge[cell] node[right] {$\phi$} (m-2-1);
			\end{tikzpicture} \qquad\qquad \begin{tikzpicture}[textbaseline]
				\matrix(m)[tab]{X & Y \\ \gen J & \gen J \\};
				\path[map]	(m-1-1) edge[barred] node[above] {$H$} (m-1-2)
														edge node[left] {$f$} (m-2-1)
										(m-1-2) edge node[right] {$g$} (m-2-2);
				\path				(m-2-1) edge[eq] (m-2-2);
				\path[transform canvas={shift={($(m-1-2)!0.5!(m-2-2)$)}}]	(m-1-1) edge[cell] node[right] {$\psi$} (m-2-1);
			\end{tikzpicture}
		\end{displaymath}
	\end{lemma}
	\begin{proof}
		Applying the $2$-dimensional universal property of $\pi$ to the identity
		\begin{displaymath}
			\begin{tikzpicture}[textbaseline, x=2em, y=1.75em, font=\scriptsize]
				\draw	(0,1) -- (2,1) -- (2,0) -- (1,0) -- (1,2) -- (0,2) -- (0,1);
				\draw[shift={(0.5,0.5)}]
							(0,1) node {$\phi$}
							(1,0) node {$\pi$};
			\end{tikzpicture} \mspace{20mu} = \mspace{20mu} \begin{tikzpicture}[textbaseline, x=2em, y=1.75em, font=\scriptsize]
				\draw (0,0) -- (1,0) -- (1,2) -- (2,2) -- (2,1) -- (0,1) -- (0,0);
				\draw[shift={(0.5,0.5)}]
							(0,0) node {$\pi$}
							(1,1) node {$\phi$};
			\end{tikzpicture}
		\end{displaymath}
		we find that there is exactly one cell $\chi$ such that $1_{\pi_A} \of \chi = 1_{\pi_A} \of \phi$ and $1_{\pi_B} \of \chi = 1_{\pi_B} \of \phi$. Since both $\chi = \phi$ and $\chi = \psi$ satisfy these identities we must have $\phi = \psi$.
	\end{proof}
	\begin{proof}[of \propref{lifting tabulations}.]
		Given a horizontal $T$-morphism $\hmap{(J, \bar J)}AB$, we assume that there exists a cell $\pi$ in $\K$, as on the left above, that defines the tabulation of its underlying horizontal morphism $\hmap JAB$. We will construct a lax $T$-algebra structure on $\gen J$ and check that it is unique such that $\pi_A$ and $\pi_B$ become strict $T$\ndash morphisms, and that $\pi$ becomes a $T$-cell. Following this we prove that as a $T$\ndash cell $\pi$ satisfies the universal properties of a tabulation, and that it is opcartesian whenever $\pi$ is opcartesian in $\K$.
		
		\emph{Lax $T$-algebra structure on $\gen J$.} The universal properties of the tabulation $\pi$ induce a lax $T$-algebra structure on $\gen J$ as follows. For the structure map \mbox{$\map w{T\gen J}{\gen J}$} consider the composite on the left-hand side below: by the $1$\ndash di\-men\-sional universal property of tabulations it factors uniquely through $\pi$ as a vertical morphism $\map w{T\gen J}{\gen J}$ as shown.
		\begin{equation} \label{structure map for tabulation}
			\begin{tikzpicture}[textbaseline]
				\matrix(m)[tab]{T\gen J & T\gen J \\ TA & TB \\ A & B \\};
				\path[map]	(m-1-1) edge node[left] {$T\pi_A$} (m-2-1)
										(m-1-2) edge node[right] {$T\pi_B$} (m-2-2)
										(m-2-1) edge[barred] node[below] {$TJ$} (m-2-2)
														edge node[left] {$a$} (m-3-1)
										(m-2-2) edge node[right] {$b$} (m-3-2)
										(m-3-1) edge[barred] node[below] {$J$} (m-3-2);
				\path				(m-1-1) edge[eq] (m-1-2);
				\path[transform canvas={shift=(m-2-2)}]	(m-1-1) edge[cell] node[right] {$T\pi$} (m-2-1)
										(m-2-1) edge[cell, transform canvas={yshift=-3pt}] node[right] {$\bar J$} (m-3-1);
			\end{tikzpicture} = \begin{tikzpicture}[textbaseline]
				\matrix(m)[tab]{T\gen J & T\gen J \\ \gen J & \gen J \\ A & B \\};
				\path[map]	(m-1-1) edge node[left] {$w$} (m-2-1)
										(m-1-2) edge node[right] {$w$} (m-2-2)
										(m-2-1) edge node[left] {$\pi_A$} (m-3-1)
										(m-2-2) edge node[right] {$\pi_B$} (m-3-2)
										(m-3-1) edge[barred] node[below] {$J$} (m-3-2);
				\path				(m-1-1) edge[eq] (m-1-2)
										(m-2-1) edge[eq] (m-2-2);
				\path[transform canvas={shift=(m-2-2)}] (m-2-1) edge[cell] node[right] {$\pi$} (m-3-1);
			\end{tikzpicture}
		\end{equation}
		
		To obtain the vertical associator cell $\cell{\bar w}{w \of Tw}{w \of \mu_{\gen J}}$ we consider the identity below, which is the associativity axiom for $J$ precomposed with $T^2\pi$. Notice that the second column of the left-hand side coincides with $\pi \of 1_w \of 1_{\mu_{\gen J}}$, and the first of the right-hand side with $\pi \of 1_w \of 1_{Tw}$. Invoking the 2-dimensional universal property of $\pi$ we obtain a unique cell $\cell{\bar w}{w \of Tw}{w \of \mu_{\gen J}}$ such that $1_{\pi_A} \of \bar w = \bar a \of 1_{T^2\pi_A}$ and $1_{\pi_B} \of \bar w = \bar b \of 1_{T^2\pi_B}$; we take $\bar w$ as the associator of $\gen J$.
		\begin{displaymath}
			\begin{tikzpicture}[textbaseline]
				\matrix(m)[tab]
				{ \phantom{T^2\gen J} & T^2\gen J & T^2\gen J \\
					T^2A & T^2A & T^2B \\
					TA & TA & TB \\
					A & A & B \\};
				\path[map]	(m-1-2) edge node[left] {$T^2\pi_A$} (m-2-2)
										(m-1-3) edge node[right] {$T^2\pi_B$} (m-2-3)
										(m-2-1) edge node[left] {$Ta$} (m-3-1)
										(m-2-2) edge[barred] node[below] {$T^2J$} (m-2-3)
														edge node[left] {$\mu_A$} (m-3-2)
										(m-2-3) edge node[right] {$\mu_B$} (m-3-3)
										(m-3-1) edge node[left] {$a$} (m-4-1)
										(m-3-2) edge[barred] node[below] {$TJ$} (m-3-3)
														edge node[left] {$a$} (m-4-2)
										(m-3-3) edge node[right] {$b$} (m-4-3)
										(m-4-2) edge[barred] node[below] {$J$} (m-4-3);
				\path				(m-1-2) edge[eq] (m-1-3)
										(m-2-1) edge[eq] (m-2-2)
										(m-4-1) edge[eq] (m-4-2);
				\path[transform canvas={shift={($(m-3-2)!0.5!(m-3-3)$)}}]	(m-2-1) edge[cell] node[right] {$\bar a$} (m-3-1);
				\path[transform canvas={shift={($(m-2-3)!0.5!(m-3-2)$)}}]	(m-1-2) edge[cell] node[right] {$T^2\pi$} (m-2-2)
										(m-2-2) edge[cell, transform canvas={yshift=-3pt}] node[right] {$\mu_J$} (m-3-2)
										(m-3-2) edge[cell, transform canvas={yshift=-2pt}] node[right] {$\bar J$} (m-4-2);
			\end{tikzpicture} = \begin{tikzpicture}[textbaseline]
				\matrix(m)[tab]
				{ T^2\gen J & T^2\gen J & \phantom{T^2\gen J} \\
					T^2A & T^2B & T^2B \\
					TA & TB & TB \\
					A & B & B \\ };
				\path[map]	(m-1-1) edge node[left] {$T^2\pi_A$} (m-2-1)
										(m-1-2) edge node[right] {$T^2\pi_B$} (m-2-2)
										(m-2-1) edge[barred] node[below] {$T^2J$} (m-2-2)
														edge node[left] {$Ta$} (m-3-1)
										(m-2-2) edge node[right] {$Tb$} (m-3-2)
										(m-2-3) edge node[right] {$\mu_B$} (m-3-3)
										(m-3-1) edge[barred] node[below] {$TJ$} (m-3-2)
														edge node[left] {$a$} (m-4-1)
										(m-3-2) edge node[right] {$b$} (m-4-2)
										(m-3-3) edge node[right] {$b$} (m-4-3)
										(m-4-1) edge[barred] node[below] {$J$} (m-4-2);
				\path				(m-1-1) edge[eq] (m-1-2)
										(m-2-2) edge[eq] (m-2-3)
										(m-4-2) edge[eq] (m-4-3);
				\path[transform canvas={shift={($(m-2-3)!0.5!(m-3-2)$)}}]	(m-1-1) edge[cell] node[right] {$T^2\pi$} (m-2-1)
										(m-2-1) edge[cell, transform canvas={yshift=-3pt}] node[right] {$T\bar J$} (m-3-1)
										(m-3-1) edge[cell, transform canvas={yshift=-2pt}] node[right] {$\bar J$} (m-4-1);
				\path[transform canvas={shift={($(m-3-2)!0.5!(m-3-3)$)}}]	(m-2-2) edge[cell] node[right] {$\bar b$} (m-3-2);
			\end{tikzpicture}
		\end{displaymath}
		To prove the coherence axiom for $\bar w$ (see \defref{definition: lax algebra}) we use the lemma above as follows. It is clear from the definitions of $w$ and $\bar w$ that, by postcomposing either side of the coherence axiom for $\bar w$ with $1_{\pi_A}$, we obtain the corresponding side of the coherence axiom for $\bar a$ precomposed with $1_{T^3\pi_A}$. Likewise, postcomposing them with $1_{\pi_B}$ we obtain the sides of the coherence axiom for $\bar b$, precomposed with $1_{T^3\pi_B}$. Since the coherence axioms for $\bar a$ and $\bar b$ hold, and because $1_{\pi_A}$ and $1_{\pi_B}$ are jointly monic, the coherence axiom for $\bar w$ follows.
		
		The vertical unitor cell $\cell{\hat w}{\id_{\gen J}}{w \of \eta_{\gen J}}$ is obtained in a similar way, by applying the $2$-dimension universal propery of $\pi$ to the unit axiom for $\bar J$ precomposed with $\pi$; it is unique such that $1_{\pi_A} \of \hat w = \hat a \of 1_{\pi_A}$ and $1_{\pi_B} \of \hat w = \hat b \of 1_{\pi_B}$. The coherence axioms for $\hat w$ follow from those for $\hat a$ and $\hat b$, as well as the joint monicity of $1_{\pi_A}$ and $1_{\pi_B}$, analogous to how the coherence axiom for $\bar w$ followed from that for $\bar a$ and $\bar b$ above. This completes the definition of the lax $T$-algebra $\gen J = (\gen J, w, \bar w, \hat w)$.
		
		\emph{Uniqueness of the algebra structure on $\gen J$.} Summarising the above, the algebra structure on $\gen J$ is determined as follows: the structure map $w$ is uniquely determined by the identity \eqref{structure map for tabulation}, while the associator $\bar w$ is uniquely determined by the identities $1_{\pi_A} \of \bar w = \bar a \of 1_{T^2\pi_A}$ and $1_{\pi_B} \of \bar w = \bar b \of 1_{T^2\pi_B}$, and the unitor $\hat w$ by the identities $1_{\pi_A} \of \hat w = \hat a \of 1_{\pi_A}$ and $1_{\pi_B} \of \hat w = \hat b \of 1_{\pi_B}$. It follows from \eqref{structure map for tabulation} that $\pi$ becomes a $T$-cell if we take the structure cells of the projections $\pi_A$ and $\pi_B$ to be trivial, while the identities determining $\bar w$ and $\hat w$ are precisely the coherence axioms for these trivial structure cells. Conversely, by the uniqueness of $w$, $\bar w$ and $\hat w$, any lift of $\pi$ to a $T$-cell with strict vertical $T$-morphisms must have $(w, \bar w, \hat w)$ as the algebraic structure on $\gen J$.
		
		\emph{Universal properties of $\pi$ as a $T$-cell.} We will show that $\pi$ forms a tabulation of $J$ in $\Alg lT$ and in $\Alg{ps}T$; proving that it forms one in $\Alg cT$ is similar to the proof for $\Alg lT$. For the $1$-dimensional universal property consider a $T$-cell $\phi$ as on the left below, where $X = (X, x, \bar x, \hat x)$ is a lax $T$-algebra and where $\phi_A$ and $\phi_B$ are lax $T$-morphisms. Since $\pi$ is a tabulation in $\K$ the cell $\phi$ factors uniquely through $\pi$ as a vertical morphism $\phi'$ as shown; we have to show that $\phi'$ can be lifted to a lax $T$-morphism.
		\begin{displaymath}
			\begin{tikzpicture}[textbaseline]
  			\matrix(m)[math175em]{X & X \\ A & B \\};
  			\path[map]  (m-1-1) edge node[left] {$\phi_A$} (m-2-1)
										(m-1-2) edge node[right] {$\phi_B$} (m-2-2)
										(m-2-1) edge[barred] node[below] {$J$} (m-2-2);
				\path				(m-1-1) edge[eq] (m-1-2);
				\path[transform canvas={shift={($(m-1-2)!(0,0)!(m-2-2)$)}}] (m-1-1) edge[cell] node[right] {$\phi$} (m-2-1);
			\end{tikzpicture} = \begin{tikzpicture}[textbaseline]
  			\matrix(m)[tab]{X & X \\ \tab J & \tab J \\ A & B \\};
  			\path[map]  (m-1-1) edge node[left] {$\phi'$} (m-2-1)
  									(m-1-2) edge node[right] {$\phi'$} (m-2-2)
  									(m-2-1) edge node[left] {$\pi_A$} (m-3-1)
										(m-2-2) edge node[right] {$\pi_B$} (m-3-2)
										(m-3-1) edge[barred] node[below] {$J$} (m-3-2);
				\path				(m-1-1) edge[eq] (m-1-2)
										(m-2-1) edge[eq] (m-2-2);
				\path[transform canvas={shift=(m-2-2)}] (m-2-1) edge[cell] node[right] {$\pi$} (m-3-1);
			\end{tikzpicture} \qquad\qquad \begin{tikzpicture}[textbaseline, x=2em, y=1.75em, font=\scriptsize]
				\draw (1,1) -- (2,1) -- (2,0) -- (0,0) -- (0,2) -- (1,2) -- (1,0)
							(0.5,1) node {$\bar{\phi_A}$};
				\draw[shift={(0.5,0.5)}]
							(1,0) node {$\phi$};
			\end{tikzpicture} \mspace{20mu} = \mspace{20mu} \begin{tikzpicture}[textbaseline, x=2em, y=1.75em, font=\scriptsize]
				\draw	(1,0) -- (1,2) -- (2,2) -- (2,0) -- (0,0) -- (0,2) -- (1,2)
							(0,1) -- (1,1)
							(1.5,1) node {$\bar{\phi_B}$};
				\draw[shift={(0.5,0.5)}]
							(0,0) node {$\bar J$}
							(0,1) node {$T\phi$};
			\end{tikzpicture}
		\end{displaymath}
		To obtain a structure cell for $\phi'$ we consider the $T$-cell axiom for $\phi$, as on the right above. Notice that $\bar J \of T\phi = \pi \of 1_w \of 1_{T\phi'}$ and $\phi \of 1_x = \pi \of 1_{\phi'} \of 1_x$ here, by the identity on the left above and \eqref{structure map for tabulation}, so that by invoking the $2$-dimensional universal property of $\pi$ we obtain a unique vertical cell $\cell{\bar{\phi'}}{w \of T\phi'}{\phi' \of x}$ such that $1_{\pi_A} \of \bar{\phi'} = \bar{\phi_A}$ and $1_{\pi_B} \of \bar{\phi'} = \bar{\phi_B}$. From these identities it follows that the coherence axioms for $\bar{\phi_A}$ and $\bar{\phi_B}$ coincide with those of $\bar{\phi'}$ postcomposed with $1_{\pi_A}$ and $1_{\pi_B}$ respectively. Since the latter are jointly monic the coherence axioms for $\bar{\phi'}$ follow, and we conclude that $(\phi', \bar{\phi'})$ forms a unique lift of $\phi'$; this completes the proof of $\pi$ satisfying the $1$-dimensional universal property in $\Alg lT$.
		
		To see that the same property is satisfied by $\pi$ in $\Alg{ps}T$ assume that, in the above, $\phi_A$ and $\phi_B$ are pseudo $T$-morphisms. Then, using the $2$-dimensional universal property of $\pi$ again, we can obtain a unique vertical cell $\cell\chi{\phi' \of x}{w \of T\phi'}$ such that $1_{\pi_A} \of \chi = \inv{\bar{\phi_A}}$ and $1_{\pi_B} \of \chi = \inv{\bar{\phi_B}}$. Using the joint monicity of $\pi_A$ and $\pi_B$ again, it follows that $\chi$ forms the inverse of $\bar{\phi'}$ and we conclude that the lift $(\phi', \bar{\phi'})$ of $\phi'$ is a pseudo $T$-morphism.
		 
		The $2$-dimensional universal property for $\pi$, as a $T$-cell, follows easily from the corresponding property for $\pi$ as a cell in $\K$. Indeed, consider an identity $\xi_A \hc \psi = \phi \hc \xi_B$ of $T$-cells, as in \defref{definition: tabulation}. By the $2$-dimensional universal property of $\pi$ there exists a unique cell $\xi'$ in $\K$ such that $1_{\pi_A} \of \xi' = \xi_A$ and $1_{\pi_B} \of \xi' = \xi_B$. It follows that the $T$-cell axiom for $\xi'$, postcomposed with $1_{\pi_A}$ and $1_{\pi_B}$, coincides with the $T$-cell axioms for $\xi_A$ and $\xi_B$ so that, by the joint monicity of $1_{\pi_A}$ and $1_{\pi_B}$, the $T$-cell axiom for $\xi'$ follows. We conclude that $\pi$ as a $T$-cell forms the tabulation of the horizontal $T$-morphism $J$, both in $\Alg lT$ and $\Alg{ps}T$.
		 
		\emph{The $T$-cell $\pi$ is opcartesian as soon as $\pi$ is opcartesian in $\K$.} If $\pi$ is an opcartesian cell in $\K$ then so is $T\pi$, by \propref{preserving of cartesian and opcartesian cells}. To prove that this ensures that $\pi$ is opcartesian as a $T$-cell we consider a $T$-cell $\phi$ as on the left below; in $\K$ it factors through the opcartesian cell $\pi$ as a cell $\phi'$, as shown.
		\begin{displaymath}
			\begin{tikzpicture}[textbaseline]
    		\matrix(m)[tab]{\gen J & \gen J \\ A & B \\ C & D \\};
    		\path[map]  (m-1-1) edge node[left] {$\pi_A$} (m-2-1)
      	 		        (m-1-2) edge node[right] {$\pi_B$} (m-2-2)
      	 		        (m-2-1) edge node[left] {$f$} (m-3-1)
      	 		        (m-2-2) edge node[right] {$g$} (m-3-2)
      	 		        (m-3-1) edge[barred] node[below] {$K$} (m-3-2);
      	\path				(m-1-1) edge[eq] (m-1-2);
    		\path[transform canvas={shift={($(m-2-1)!0.5!(m-1-1)$)}}] (m-2-2) edge[cell] node[right] {$\phi$} (m-3-2);
  		\end{tikzpicture} = \begin{tikzpicture}[textbaseline]
    		\matrix(m)[tab]{\gen J & \gen J \\ A & B \\ C & D \\};
    		\path[map]  (m-1-1) edge node[left] {$\pi_A$} (m-2-1)
      	     		    (m-1-2) edge node[right] {$\pi_B$} (m-2-2)
      	     		    (m-2-1) edge[barred] node[below] {$J$} (m-2-2)
      	     		            edge node[left] {$f$} (m-3-1)
      	     		    (m-2-2) edge node[right] {$g$} (m-3-2)
      	     		    (m-3-1) edge[barred] node[below] {$K$} (m-3-2);
      	\path				(m-1-1) edge[eq] (m-1-2);
    		\path[transform canvas={shift=(m-2-1))}]
      	 		        (m-1-2) edge[cell] node[right] {$\pi$} (m-2-2)
      	 		        (m-2-2) edge[transform canvas={yshift=-0.3em}, cell] node[right] {$\phi'$} (m-3-2);
  		\end{tikzpicture}
  	\end{displaymath}
  	To show that $\phi'$ satisfies the $T$-cell axiom consider the following equation of composites in $\K$, where the identities are given by the factorisation above, the $T$-cell axiom for $\phi$, again the factorisation above, and the $T$-cell axiom for $\pi$.
  	\begin{displaymath}
  		\begin{tikzpicture}[textbaseline, x=2em, y=1.75em, font=\scriptsize]
  			\draw	(1,0) -- (1,3) -- (0,3) -- (0,0) -- (2,0) -- (2,2) -- (0,2)
  						(0,1) -- (1,1)
  						(1.5,1) node {$\bar g$};
  			\draw[shift={(0.5,0.5)}]
  						(0,0) node {$\bar K$}
  						(0,1) node {$T\phi'$}
  						(0,2) node {$T\pi$};
  		\end{tikzpicture} \mspace{10mu} = \mspace{10mu} \begin{tikzpicture}[textbaseline, x=2em, y=1.75em, font=\scriptsize]
  			\draw (1,0) -- (1,3) -- (0,3) -- (0,0) -- (2,0) -- (2,2) -- (1,2)
  						(0,1) -- (1,1)
  						(0.5,0.5) node {$\bar K$}
  						(0.5,2) node {$T\phi$}
  						(1.5,1) node {$\bar g$};
  		\end{tikzpicture} \mspace{10mu} = \mspace{10mu} \begin{tikzpicture}[textbaseline, x=2em, y=1.75em, font=\scriptsize]
  			\draw (1,2) -- (0,2) -- (0,0) -- (1,0) -- (1,3) -- (2,3) -- (2,0) -- (3,0) -- (3,3) -- (4,3) -- (4,1) -- (3,1)
  						(1,1) -- (2,1)
  						(2,2) -- (3,2)
  						(0.5,1) node {$\bar f$}
  						(2.5,1) node {$\phi$};
  		\end{tikzpicture} \mspace{10mu} = \mspace{10mu} \begin{tikzpicture}[textbaseline, x=2em, y=1.75em, font=\scriptsize]
  			\draw	(1,2) -- (0,2) -- (0,0) -- (1,0) -- (1,3) -- (2,3) -- (2,0) -- (3,0) -- (3,3) -- (4,3) -- (4,1) -- (1,1)
  						(2,2) -- (3,2)
  						(0.5,1) node {$\bar f$}
  						(2.5,0.5) node {$\phi'$}
  						(2.5,1.5) node {$\pi$};
  		\end{tikzpicture} \mspace{10mu} = \mspace{10mu} \begin{tikzpicture}[textbaseline, x=2em, y=1.75em, font=\scriptsize]
  			\draw (1,0) -- (1,3) -- (2,3) -- (2,0) -- (0,0) -- (0,2) -- (2,2)
  						(1,1) -- (2,1)
  						(0.5,1) node {$\bar f$}
  						(1.5,0.5) node {$\phi'$}
  						(1.5,1.5) node {$\bar J$}
  						(1.5,2.5) node {$T\pi$};
  		\end{tikzpicture}
  	\end{displaymath}
  	Since the two sides of the equation above are the sides of the $T$-cell axiom for $\phi'$ precomposed with $T\pi$, the axiom itself follows because $T\pi$ is opcartesian, and factorisations through opcartesian cells are unique. This completes the proof.
  \end{proof}
	
	Combining \propref{lifting restrictions} and \propref{lifting tabulations} we conclude that if $T$ is a normal oplax double monad on an equipment $\K$ that has opcartesian tabulations, then $\Alg{ps}T$ is again an equipment that has opcartesian tabulations. It follows that pointwise Kan extensions in $\Alg{ps}T$ can be defined in terms of Kan extensions (by \thmref{pointwise Kan extensions in terms of Kan extensions}) so that pointwise right Kan extensions in the $2$-category $\Alg{ps}{V(T)}$ correspond to pointwise right Kan extensions along conjoints in $\Alg{ps}T$ (by \propref{pointwise right Kan extensions along conjoints as pointwise right Kan extensions in V(K)}).
	
	Even though in the cases $\textup w = \textup c$ and $\textup w = \textup l$ the double categories $\Alg wT$ fail to be equipments in general, pointwise Kan extensions in $\Alg wT$ can still be defined in terms of Kan extensions by the following variant of \thmref{pointwise Kan extensions in terms of Kan extensions}. Remember that, for a pseudo $T$-morphism $\map fCA$ and a horizontal $T$-morphism $\hmap JAB$, the restriction $J(f, \id)$ exists in each of the double categories $\Alg wT$, by \propref{lifting restrictions}.
	
	\begin{theorem} \label{algebraic pointwise Kan extensions in terms of Kan extensions}
		Let $T$ be a normal oplax double monad on an equipment $\K$ and let \mbox{$\textup w \in \brcs{\textup c, \textup l, \textup{ps}}$}. Consider a cell $\eps$ in $\Alg wT$ as on the left below.
		\begin{displaymath}
			\begin{tikzpicture}[textbaseline]
  			\matrix(m)[math175em]{A & B \\ M & M \\};
  			\path[map]  (m-1-1) edge[barred] node[above] {$J$} (m-1-2)
														edge node[left] {$r$} (m-2-1)
										(m-1-2) edge node[right] {$d$} (m-2-2);
				\path				(m-2-1) edge[eq] (m-2-2);
				\path[transform canvas={shift={($(m-1-2)!(0,0)!(m-2-2)$)}}] (m-1-1) edge[cell] node[right] {$\eps$} (m-2-1);
			\end{tikzpicture} \qquad\qquad\qquad\qquad\qquad \begin{tikzpicture}[textbaseline]
				\matrix(m)[math175em]{C & B \\ A & B \\ M & M \\};
				\path[map]	(m-1-1) edge[barred] node[above] {$J(f, \id)$} (m-1-2)
														edge node[left] {$f$} (m-2-1)
										(m-2-1) edge[barred] node[below] {$J$} (m-2-2)
														edge node[left] {$r$} (m-3-1)
										(m-2-2) edge node[right] {$d$} (m-3-2);
				\path				(m-1-2) edge[eq] (m-2-2)
										(m-3-1) edge[eq] (m-3-2);
				\path[transform canvas={shift=(m-2-1), yshift=-2pt}]	(m-2-2) edge[cell] node[right] {$\eps$} (m-3-2);
				\draw				($(m-1-1)!0.5!(m-2-2)$)	node {\textup{cart}};
			\end{tikzpicture} 
		\end{displaymath}
		For the following conditions the implications (a) $\Leftrightarrow$ (b) $\Rightarrow$ (c) hold, while (c) $\Rightarrow$ (a) holds as soon as $\K$ has opcartesian tabulations.
		\begin{enumerate}[label=(\alph*)]
			\item The cell $\eps$ defines $r$ as the pointwise right Kan extension of $d$ along $J$;
			\item for all pseudo $T$-morphisms $\map fCA$ the composite on the right above defines $r \of f$ as the pointwise right Kan extension of $d$ along $J(f, \id)$;
			\item for all strict $T$-morphisms $\map fCA$ the composite on the right above defines $r \of f$ as the right Kan extension of $d$ along $J(f, \id)$.
		\end{enumerate}
	\end{theorem}
	\begin{proof}
		The proof given in \cite{Koudenburg14a} for its Theorem 5.11, which has been restated in the present paper as \thmref{pointwise Kan extensions in terms of Kan extensions}, applies to the double categories $\Alg wT$ without needing any adjustment. This is because all conjoints and companions used in that proof exist in $\Alg wT$, as all of them are conjoints or companions of pseudo $T$\ndash morphisms, and because, in the proof of the implication (c) $\Rightarrow$ (a), the projections of the lifted tabulations in $\Alg wT$ are strict $T$-morphisms; see \propref{lifting tabulations}.
	\end{proof}
	
	Just like the proof of \thmref{pointwise Kan extensions in terms of Kan extensions} applies to the double category $\Alg lT$, so does the proof of \propref{pointwise right Kan extensions along conjoints as pointwise right Kan extensions in V(K)} (which is Proposition 5.12 of \cite{Koudenburg14a}). We thus obtain the following result, which describes pointwise right Kan extensions along conjoints in $\Alg lT$ in terms of right Kan extensions in the $2$-category $\Alg l{V(T)}$, in a sense  that is analogous to Street's definition of pointwise Kan extension given in \cite{Street74}.
	
	In the diagram on the right below $f \slash j$ denotes the `comma object' of \mbox{$\map fCA$} and $\map jBA$ in the $2$-category $\Alg l{V(T)}$, see e.g.\ Definition 5.2 of \cite{Koudenburg14a}. In terms of the double category $\Alg lT$, the comma object $f \slash j$ can be defined as the tabulation of the restriction $A(f, j)$; in particular $f \slash j$ exists whenever $f$ is a pseudo $T$-morphism and $\K$ is an equipment that has tabulations, by \propref{lifting restrictions} and \propref{lifting tabulations}.
	\begin{proposition}
		Let $T$ be a normal oplax double monad on an equipment $\K$ that has opcartesian tabulations. Consider a vertical $T$-cell $\eps$ between lax $T$-morphisms, as on the left below, as well as its factorisation through the opcartesian cell defining $j^*$, as shown.
		\begin{displaymath}
			\begin{tikzpicture}[textbaseline]
    		\matrix(m)[math175em]{B & B \\ A & \phantom A \\ M & M \\};
    		\path[map]  (m-1-1) edge node[left] {$j$} (m-2-1)
       			        (m-1-2) edge node[right] {$d$} (m-3-2)
       			        (m-2-1) edge node[left] {$r$} (m-3-1);
      	\path				(m-1-1) edge[eq] (m-1-2)
      							(m-3-1) edge[eq] (m-3-2);
    		\path[transform canvas={shift={($(m-2-1)!0.5!(m-1-1)$)}}] (m-2-2) edge[cell] node[right] {$\eps$} (m-3-2);
  		\end{tikzpicture} = \begin{tikzpicture}[textbaseline]
				\matrix(m)[math175em]{B & B \\ A & B \\ M & M \\};
				\path[map]	(m-1-1) edge node[left] {$j$} (m-2-1)
										(m-2-1) edge[barred] node[below] {$j^*$} (m-2-2)
														edge node[left] {$r$} (m-3-1)
										(m-2-2) edge node[right] {$d$} (m-3-2);
				\path				(m-1-1) edge[eq] (m-1-2)
										(m-1-2) edge[eq] (m-2-2)
										(m-3-1) edge[eq] (m-3-2);
				\path[transform canvas={shift=(m-2-1), yshift=-2pt}]	(m-2-2) edge[cell] node[right] {$\eps'$} (m-3-2);
				\draw				($(m-1-1)!0.5!(m-2-2)$)	node {\textup{opcart}};
			\end{tikzpicture} \qquad\qquad\qquad \begin{tikzpicture}[textbaseline]
				\matrix(m)[math2em, column sep=0.25em, row sep=2.5em]{f \slash j & & C \\ B & & A \\[-0.6em] \phantom M & M & \phantom M \\};
				\path[map]	(m-1-1) edge node[above] {$\pi_C$} (m-1-3)
														edge node[left] {$\pi_B$} (m-2-1)
										(m-1-3) edge node[right] {$f$} (m-2-3)
										(m-2-1) edge node[above] {$j$} (m-2-3)
														edge node[below left] {$d$} (m-3-2)
										(m-2-3) edge node[below right] {$r$} (m-3-2);
				\path				(m-1-3) edge[cell, shorten >=13pt, shorten <=13pt] node[above left] {$\pi$} (m-2-1)
										(m-2-3) edge[cell, shorten >=9pt, shorten <=9pt] node[below right, xshift=-2pt] {$\eps$} ($(m-2-1)!0.5!(m-3-2)$);
			\end{tikzpicture}
		\end{displaymath}
		The factorisation $\eps'$ defines $r$ as the pointwise right Kan extension of $d$ along $j^*$ in $\Alg lT$ precisely if, for each strict $T$-morphism $\map fCA$, the composite on the right above defines $r \of f$ as the right Kan extension of $d \of \pi_B$ along $\pi_C$ in $\Alg l{V(T)}$.
	\end{proposition}
	\begin{example}
		The `free strict monoidal category'-monad (\exref{example: free strict monoidal enriched category monad}) on the equipment $\Prof$ of (unenriched) profunctors satisfies the hypotheses of the proposition above. Thus a monoidal transformation $\cell\eps{r \of j}d$ of lax monoidal functors, as on the left above, defines $r$ as the pointwise Kan extension of $d$ along $j$ if, for every strict monoidal functor $\map fCA$, the composite on the right above defines $r \of f$ as the right Kan extension of $d \of \pi_B$ along $\pi_C$. In this case $f \slash j$ is the usual `comma category' that has triples $(z, u, y)$ as objects, where $z \in C$, $y \in B$ and $\map u{fz}{jy}$ in $A$. Its monoidal structure, that is induced by that of $C$ and $B$, maps a sequence $(\ul z, \ul u, \ul y)$ to the triple consisting of the tensor products $(z_1 \tens \dotsb \tens z_n)$ and $(y_1 \tens \dotsb \tens y_n)$ as well as the map
		\begin{displaymath}
			f(z_1 \tens \dotsb \tens z_n) = (fz_1 \tens \dotsb \tens fz_n) \xrar{(u_1 \tens \dotsb \tens u_n)} (jy_1 \tens \dotsb \tens jy_n) \xrar{j_\tens} j(y_1 \tens \dotsb \tens y_n).
		\end{displaymath} 
	\end{example}
	
	\section{Lifting algebraic Kan extensions}
  In this section we prove our main result, which gives conditions allowing Kan extensions to be lifted along the forgetful double functors $\Alg wT \to \K$. As an application we recover and generalise a result of Getzler, given in \cite{Getzler09}, on the lifting of symmetric monoidal pointwise left Kan extensions.
  
  We start by making the idea of `lifting Kan extensions' precise.
  \begin{definition}
  	Consider a normal oplax double functor $\map F\K\L$ as well as morphisms $\hmap JAB$ and $\map dBM$ in $\K$. We say that $F$ \emph{lifts} the (pointwise) right Kan extension of $d$ along $J$ if, given any cell $\eps$ in $\L$, as on the left below, that defines $r$ as the (pointwise) right Kan extension of $Fd$ along $FJ$, there exists a unique cell $\eps'$ in $\K$ as on the right such that $F\eps' = \eps$ and, moreover, $\eps'$ defines $r'$ as the (pointwise) right Kan extension of $d$ along $J$.
  	\begin{displaymath}
  		\begin{tikzpicture}[textbaseline]
  			\matrix(m)[math175em]{FA & FB \\ FM & FM \\};
  			\path[map]  (m-1-1) edge[barred] node[above] {$FJ$} (m-1-2)
														edge node[left] {$r$} (m-2-1)
										(m-1-2) edge node[right] {$Fd$} (m-2-2);
				\path				(m-2-1) edge[eq] (m-2-2);
				\path[transform canvas={shift={($(m-1-2)!(0,0)!(m-2-2)$)}}] (m-1-1) edge[cell] node[right] {$\eps$} (m-2-1);
			\end{tikzpicture} \qquad\qquad\qquad\qquad \begin{tikzpicture}[textbaseline]
  			\matrix(m)[math175em]{A & B \\ M & M \\};
  			\path[map]  (m-1-1) edge[barred] node[above] {$J$} (m-1-2)
														edge node[left] {$r'$} (m-2-1)
										(m-1-2) edge node[right] {$d$} (m-2-2);
				\path				(m-2-1) edge[eq] (m-2-2);
				\path[transform canvas={shift={($(m-1-2)!(0,0)!(m-2-2)$)}}] (m-1-1) edge[cell] node[right] {$\eps'$} (m-2-1);
			\end{tikzpicture}
  	\end{displaymath}
  \end{definition}
  
  To state the main theorem we need Definitions \ref{definition: preserving Kan extensions} and \ref{definition: right exact monad}.
  \begin{definition} \label{definition: preserving Kan extensions}
  	Let $\map dBM$, $\map fMN$ and $\hmap JAB$ be morphisms in a double category $\K$. We say that $f$ \emph{preserves} the (pointwise) right Kan extension of $d$ along $J$ if, for any cell $\eps$ on the left below that defines $r$ as the (pointwise) right Kan extension of $d$ along $J$, the composite $1_f \of \eps$ in the middle defines $f \of r$ as the (pointwise) right Kan extension of $f \of d$ along $J$.
  	\begin{displaymath}
  		\begin{tikzpicture}[textbaseline]
  			\matrix(m)[math175em]{A & B \\ M & M \\};
  			\path[map]  (m-1-1) edge[barred] node[above] {$J$} (m-1-2)
														edge node[left] {$r$} (m-2-1)
										(m-1-2) edge node[right] {$d$} (m-2-2);
				\path				(m-2-1) edge[eq] (m-2-2);
				\path[transform canvas={shift={($(m-1-2)!(0,0)!(m-2-2)$)}}] (m-1-1) edge[cell] node[right] {$\eps$} (m-2-1);
			\end{tikzpicture} \qquad\qquad\qquad \begin{tikzpicture}[textbaseline]
				\matrix(m)[math175em, column sep=1.5em]{A & B \\ M & M \\ N & N \\};
				\path[map]	(m-1-1) edge[barred] node[above] {$J$} (m-1-2)
														edge node[left] {$r$} (m-2-1)
										(m-1-2) edge node[right] {$d$} (m-2-2)
										(m-2-1) edge node[left] {$f$} (m-3-1)
										(m-2-2) edge node[right] {$f$} (m-3-2);
				\path				(m-2-1) edge[eq] (m-2-2)
										(m-3-1) edge[eq] (m-3-2);
				\path[transform canvas={shift=(m-2-2)}]	(m-1-1) edge[cell] node[right] {$\eps$} (m-2-1);
			\end{tikzpicture} \qquad\qquad\qquad \begin{tikzpicture}[textbaseline]
				\matrix(m)[math175em, column sep=1.5em]{TA & TB \\ TM & TM \\ M & M \\};
				\path[map]	(m-1-1) edge[barred] node[above] {$TJ$} (m-1-2)
														edge node[left] {$Tr$} (m-2-1)
										(m-1-2) edge node[right] {$Td$} (m-2-2)
										(m-2-1) edge node[left] {$m$} (m-3-1)
										(m-2-2) edge node[right] {$m$} (m-3-2);
				\path				(m-2-1) edge[eq] (m-2-2)
										(m-3-1) edge[eq] (m-3-2);
				\path[transform canvas={shift=(m-2-2)}]	(m-1-1) edge[cell] node[right] {$T\eps$} (m-2-1);
			\end{tikzpicture}
  	\end{displaymath}
  	
  	Next let $T$ be a normal oplax double monad on $\K$ and $M = (M, m, \bar m, \hat m)$ a lax $T$\ndash algebra. As a variation on the above, we say that \emph{the algebraic structure of $M$ preserves} the (pointwise) right Kan extension of $d$ along $J$ if for any cell $\eps$ as on the left above, that defines $r$ as the (pointwise) right Kan extension of $d$ along $J$, the composite $1_m \of T\eps$ on the right defines $m \of Tr$ as the (pointwise) right Kan extension of $m \of Td$ along $TJ$.
  \end{definition}
  
  Recall from \propref{pointwise Kan extensions in terms of weighted limits} that, in the double category $\enProf\V$, pointwise right Kan extensions along $\V$-weights $\hmap J1B$ coincide with $J$-weighted limits. Consequently we say that a $\V$-functor $\map fMN$ \emph{preserves} $J$-weighted limits if it preserves all pointwise right Kan extensions along $J$. That this coincides with the classical notion of `preserving weighted limits', see e.g.\ Section 3.2 in \cite{Kelly82}, in the case that $\V$ is complete and closed symmetric monoidal, follows from \propref{classical definition of weighted limit}. Notice that a $\V$-functor preserves all pointwise right Kan extensions whenever it preserves weighted limits, again by \propref{pointwise Kan extensions in terms of weighted limits}.
  
  \begin{example} \label{example: pointwise left Kan extensions preserved by colax monoidal structure}
  	In anticipation of \propref{lifting pointwise monoidal left Kan extensions} we describe monoidal $\V$-cat\-e\-gories whose monoidal structure preserves pointwise \emph{left} Kan extensions. Consider a monoidal $\V$-category $M = (M, \tens, \mf a, \mf i)$ (\exref{example: monoidal enriched categories}), a $\V$-functor \mbox{$\map dAM$} and a $\V$-profunctor $\hmap JAB$; we claim that the algebraic structure of $M$ preserves the pointwise left Kan extension of $d$ along $J$ whenever the $n$-ary tensor products $\tens_n = \brks{M^{\tens n} \xrar{j_n} TM \xrar\tens M}$ preserve it in every variable, for each $n \geq 1$, where $\nat{j_n}{(\dash)^{\tens n}}T$ denotes the double transformation that includes the $n$th tensor powers into $T$. Of course, because the components of $\mf a$ and $\mf i$ are isomorphisms, it suffices that the tensor product $\tens_2$ does so. Moreover, by the discussion above, it also suffices that $\tens_2$ preserves weighted colimits.
  	
  	To prove the claim we consider a transformation $\eta$, as on the left below, that defines $l$ as the pointwise left Kan extension of $d$ along $J$ (see \defref{definition: pointwise left Kan extension}). We have to show that the composite $1_\tens \of T\eta$ defines $\tens \of Tl$ as a pointwise left Kan extension. By \lemref{pointwise Kan extensions along TJ} below we may equivalently prove that the composites $1_{\tens_n} \tens \eta^{\tens n}$ in the middle define $\tens_n \of l^{\tens n}$ as pointwise left Kan extensions, for each $n \geq 1$.
  	\begin{displaymath}
			\begin{tikzpicture}[textbaseline]
  			\matrix(m)[math175em]{A & B \\ M & M \\};
  			\path[map]  (m-1-1) edge[barred] node[above] {$J$} (m-1-2)
														edge node[left] {$d$} (m-2-1)
										(m-1-2) edge node[right] {$l$} (m-2-2);
				\path				(m-2-1) edge[eq] (m-2-2);
				\path[transform canvas={shift={($(m-1-2)!(0,0)!(m-2-2)$)}}] (m-1-1) edge[cell] node[right] {$\eta$} (m-2-1);
			\end{tikzpicture} \qquad \begin{tikzpicture}[textbaseline]
				\matrix(m)[math175em]{A^{\tens n} & B^{\tens n} \\ M^{\tens n} & M^{\tens n} \\ M & M \\};
				\path[map]	(m-1-1) edge[barred] node[above] {$J^{\tens n}$} (m-1-2)
														edge node[left] {$d^{\tens n}$} (m-2-1)
										(m-1-2) edge node[right] {$l^{\tens n}$} (m-2-2)
										(m-2-1) edge node[left] {$\tens_n$} (m-3-1)
										(m-2-2) edge node[right] {$\tens_n$} (m-3-2);
				\path				(m-2-1) edge[eq] (m-2-2)
										(m-3-1) edge[eq] (m-3-2);
				\path[transform canvas={shift=(m-2-2)}]	(m-1-1) edge[cell] node[right] {$\eta^{\tens n}$} (m-2-1);
			\end{tikzpicture} \qquad \begin{tikzpicture}[textbaseline]
				\matrix(m)[math175em, column sep=1.25em]{B^{\tens i} \tens A^{\tens(n-i)} & B^{\tens i'} \tens A^{\tens (n - i')} \\ M^{\tens n} & M^{\tens n} \\ M & M \\};
				\path[map]	(m-1-1) edge[barred] node[above, inner sep=8pt] {$1_B^{\tens i} \tens J \tens 1_A^{\tens (n - i')}$} (m-1-2)
														edge node[left] {$l^{\tens i} \tens d^{\tens (n-i)}$} (m-2-1)
										(m-1-2) edge node[right] {$l^{\tens i'} \tens d^{\tens(n-i')}$} (m-2-2)
										(m-2-1) edge node[left] {$\tens_n$} (m-3-1)
										(m-2-2) edge node[right] {$\tens_n$} (m-3-2);
				\path				(m-2-1) edge[eq] (m-2-2)
										(m-3-1) edge[eq] (m-3-2);
				\path[transform canvas={shift=(m-2-2), xshift=-3.3em}]	(m-1-1) edge[cell] node[right] {$1^{\tens i}_l \tens \eta \tens 1^{\tens(n-i')}_d$} (m-2-1);
			\end{tikzpicture}
		\end{displaymath}
		
		To see that they do consider, for each $0 \leq i < n$, the transformation on the right above, where $i' = i + 1$. That it defines $\tens_n \of (l^{\tens i'} \tens d^{\tens(n-i')})$ as a pointwise left Kan extension follows, by using \propref{pointwise Kan extensions along unit tensored V-profunctors} twice, from the fact that for all objects $y_1, \dotsc, y_i \in B$ and $x_{i+2}, \dotsc, x_n \in A$ the transformations
		\begin{displaymath}
			(ly_1 \tens \dotsb \tens ly_i \tens \dash \tens dx_{i+2} \tens \dotsb \tens dx_n) \of \eta
		\end{displaymath}
		define pointwise Kan extensions, by the assumption on $\tens_n$. Therefore, by repeatedly applying \propref{iterated pointwise Kan extensions in V-Prof} to the horizontal composite of the transformations on the right above, we can conclude that the composite $1_{\tens_n} \of \eta^{\tens n}$ defines $\tens_n \of l^{\tens n}$ as a pointwise left Kan extension, as needed.
  \end{example}
  
  Let $\V$ be a cocomplete symmetric monoidal category whose tensor product preserves colimits in both variables, so that $\V$-profunctors form a monoidal double category $\enProf\V$ (\exref{example: monoidal double category of enriched profunctors}), and let $T$ denote the `free strict monoidal $\V$\ndash cat\-e\-gory'\ndash monad on $\enProf\V$ (\exref{example: free strict monoidal enriched category monad}). For each $n \geq 0$ we denote by $\nat{j_n}{(\dash)^{\tens n}}T$ the double transformation that includes the $n$th tensor powers into $T$.
  \begin{lemma} \label{pointwise Kan extensions along TJ}
  	Let $\V$ and $\nat{j_n}{(\dash)^{\tens n}}T$ be as above. A cell
  	\begin{displaymath}
  		\begin{tikzpicture}[textbaseline]
  			\matrix(m)[math175em]{TA & TB \\ M & M \\};
  			\path[map]  (m-1-1) edge[barred] node[above] {$TJ$} (m-1-2)
														edge node[left] {$d$} (m-2-1)
										(m-1-2) edge node[right] {$l$} (m-2-2);
				\path				(m-2-1) edge[eq] (m-2-2);
				\path[transform canvas={shift={($(m-1-2)!(0,0)!(m-2-2)$)}}] (m-1-1) edge[cell] node[right] {$\eta$} (m-2-1);
			\end{tikzpicture}
		\end{displaymath}
		defines $l$ as the pointwise left Kan extension of $d$ along $TJ$ if and only if, for each $n \geq 0$, the composite $J^{\tens n} \xRar{(j_n)_J} TJ \xRar\eta 1_M$ defines $l \of j_n$ as the pointwise left Kan extension of $d \of j_n$ along $J^{\tens n}$.
  \end{lemma}
  \begin{proof}
  	For each $\ul y \in TB$ of length $n$ consider the identity below.
  	\begin{displaymath}
  		\begin{tikzpicture}[textbaseline]
  			\matrix(m)[math175em, column sep=2em]{\phantom{TB} & 1 \\ TA & 1 \\ TA & TB \\ M & M \\};
  			\path[map]	(m-1-1) edge[barred] node[above] {$TJ(j_n, \ul y)$} (m-1-2)
  													edge node[left] {$j_n$} (m-2-1)
  									(m-2-1) edge[barred] node[below] {$TJ(\id, \ul y)$} (m-2-2)
  									(m-2-2) edge node[right] {$\ul y$} (m-3-2)
  									(m-3-1) edge[barred] node[below] {$TJ$} (m-3-2)
  													edge node[left] {$d$} (m-4-1)
  									(m-3-2) edge node[right] {$l$} (m-4-2);
  			\path				(m-1-2) edge[eq] (m-2-2)
  									(m-2-1) edge[eq] (m-3-1)
  									(m-4-1) edge[eq] (m-4-2);
  			\path[transform canvas={shift=($(m-2-2)!0.5!(m-3-2)$), yshift=-3pt}]	(m-3-1) edge[cell] node[right] {$\eta$} (m-4-1);
  			\draw				(m-1-1) node {$A^{\tens n}$}
  									($(m-1-1)!0.5!(m-2-2)$) node {cart}
  									($(m-2-1)!0.5!(m-3-2)$) node[yshift=-3pt] {cart};
  		\end{tikzpicture} = \begin{tikzpicture}[textbaseline]
  			\matrix(m)[math175em, column sep=2em]{A^{\tens n} & 1 \\ A^{\tens n} & B^{\tens n} \\ TA & TB \\ M & M \\};
  			\path[map]	(m-1-1) edge[barred] node[above] {$J^{\tens n}(\id, \ul y)$} (m-1-2)
  									(m-1-2) edge node[right] {$\ul y$} (m-2-2)
  									(m-2-1) edge[barred] node[below] {$J^{\tens n}$} (m-2-2)
  													edge node[left] {$j_n$} (m-3-1)
  									(m-2-2) edge node[right] {$j_n$} (m-3-2)
  									(m-3-1) edge[barred] node[below] {$TJ$} (m-3-2)
  													edge node[left] {$d$} (m-4-1)
  									(m-3-2) edge node[right] {$l$} (m-4-2);
  			\path				(m-1-1) edge[eq] (m-2-1)
  									(m-4-1) edge[eq] (m-4-2);
  			\path[transform canvas={shift=($(m-2-2)!0.5!(m-3-2)$), yshift=-3pt}]	(m-3-1) edge[cell] node[right] {$\eta$} (m-4-1)
  									(m-2-1) edge[cell, transform canvas={xshift=-0.8em}] node[right] {$(j_n)_J$} (m-3-1);
  			\draw				($(m-1-1)!0.5!(m-2-2)$) node {cart};
  		\end{tikzpicture}
  	\end{displaymath}
  	In light of \propref{pointwise Kan extensions in terms of weighted limits} it suffices to prove that the composite of the bottom two cells on the left-hand side defines $l(\ul y)$ as a pointwise left Kan extension precisely if the full composite on the right-hand side defines $(l \of j_n)(\ul y)$ as one. It is readily checked that the top cartesian cell on the left-hand side is opcartesian, so that this equivalence follows from the horizontal dual of \propref{pointwise exact cells}(a).
  \end{proof}
  
  \begin{definition} \label{definition: right exact monad}
  	A double monad $T = (T, \mu, \eta)$ is called \emph{(pointwise) right exact} if the cells $\mu_J$ and $\eta_J$ are (pointwise) right exact for each horizontal morphism $\hmap JAB$ (see \defref{definition: exact cell}). (Pointwise) left exact double monads are defined analogously.
  \end{definition}
  
  \begin{example} \label{free strict monoidal category monad is pointwise exact}
  	The `free strict monoidal $\V$-category'-monad $T$ of \exref{example: free strict monoidal enriched category monad} is both pointwise left and right exact. For example, to show that the cell $\mu_J$, where $\hmap JAB$ is a $\V$-profunctor, is pointwise left exact it suffices, by the horizontal dual of \propref{pointwise exact cells}(b), to show that the factorisation $\mu_J'$ of $\mu_J$ through the cartesian cell defining $TJ(\id, \mu_B)$, as in the right-hand side below, is opcartesian.
  	\begin{displaymath}
			\begin{tikzpicture}[textbaseline]
    		\matrix(m)[math175em, column sep=1.5em]{T^2A & T^2B \\ TA & \\ C & D \\};
    		\path[map]  (m-1-1) edge[barred] node[above] {$T^2J$} (m-1-2)
        		                edge node[left] {$\mu_A$} (m-2-1)
        		        (m-1-2) edge node[right] {$g$} (m-3-2)
        		        (m-2-1) edge node[left] {$f$} (m-3-1)
        		        (m-3-1) edge[barred] node[below] {$K$} (m-3-2);
    		\path[transform canvas={shift={($(m-2-1)!0.5!(m-1-1)$)}}] (m-2-2) edge[cell] node[right] {$\chi$} (m-3-2);
  		\end{tikzpicture} = \begin{tikzpicture}[textbaseline]
    		\matrix(m)[math175em, column sep=2em]{T^2A & T^2B \\ TA & T^2B \\ C & D \\};
    		\path[map]  (m-1-1) edge[barred] node[above] {$T^2J$} (m-1-2)
        		                edge node[left] {$\mu_A$} (m-2-1)
            		    (m-2-1) edge[barred] node[below, inner sep=5pt] {$TJ(\id, \mu_B)$} (m-2-2)
            		            edge node[left] {$f$} (m-3-1)
            		    (m-2-2) edge node[right] {$g$} (m-3-2)
            		    (m-3-1) edge[barred] node[below] {$K$} (m-3-2);
        \path				(m-1-2) edge[eq] (m-2-2);
    		\path[transform canvas={shift=(m-2-1))}]
        		        (m-1-2) edge[cell] node[right] {$\mu_J'$} (m-2-2)
        		        (m-2-2) edge[transform canvas={yshift=-0.4em}, cell] node[right] {$\chi'$} (m-3-2);
  		\end{tikzpicture}
		\end{displaymath}
		To see this consider a cell $\chi$ as above, we have to show that it factors uniquely through $\mu_J'$ as shown. Given $\ul x \in TA$ and $\ull y \in T^2B$ notice that, if $\card{\ul x} = \card{\mu\ull y}$, then there exists a unique $\ull x \in T^2A$ with $\mu\ull x = \ul x$, $\card{\ull x} = \card{\ull y}$ and $\card{\ull x_i} = \card{\ull y_i}$ for all $1 \leq i \leq \card{\ull x}$; in this case we take as the $\V$-map $\map{\chi'_{(\ul x, \ull y)}}{TJ(\ul x, \mu\ull y)}{K(f\ul x, g\ull y)}$ the composite
		\begin{displaymath}
			TJ(\ul x, \mu\ull y) \iso T^2J(\ull x, \ull y) \xrar\chi K(f\ul x, g\ull y),
		\end{displaymath}
		where the isomorphism is the inverse of $(\mu_J)_{(\ull x, \ull y)}$; see \eqref{concatenation}. In the case that $\card{\ul x} \neq \card{\mu\ull y}$ we must take $\chi'_{(\ul x, \ull y)}$ to be the unique map $\emptyset \to K(f\ul x, g\ull y)$. It is clear that, defined like this, $\chi'$ forms the unique factorisation of $\chi$ through $\mu'_J$.
  \end{example}
  
  We are now ready to state and prove our main result. Its horizontal dual, \thmref{horizontal dual main theorem}, which concerns (pointwise) left Kan extensions between colax $T$\ndash algebras, has already been stated in the introduction where, to simplify exposition, it has been restricted to pseudo $T$-algebras and pointwise left Kan extensions.
  \begin{theorem}\label{main result}
  	Let $T$ be a right exact, normal oplax double monad on a double category $\K$ and let `weak' mean either `colax', `lax' or `pseudo'. Given lax $T$\ndash algebras $A$, $B$ and $M$, consider the following conditions on a horizontal $T$\ndash morphism \mbox{$\hmap JAB$} and a weak vertical $T$-morphism $\map dBM$:
  	\begin{enumerate}
  		\item[(p)]	the algebraic structure of $M$ preserves the right Kan extension of $d$ along $J$;
  		\item[(e)]	the structure cell of $J$ is right $d$-exact;
  		\item[(l)]	the forgetful double functor $\Alg wT \to \K$ lifts the right Kan extension of $d$ along $J$.
  	\end{enumerate}
  	The following hold:
  	\begin{enumerate}[label=\textup{(\alph*)}]
  		\item if `weak' means `colax' then (p) implies (l);
  		\item if `weak' means `lax' then (e) implies (l);
  		\item if `weak' means `pseudo' then any two of (p), (e) and (l) imply the third.
  	\end{enumerate}
  	
  	If $T$ is a pointwise right exact, normal pseudo double monad then analogous results hold for pointwise Kan extensions, obtained by replacing `right Kan extension' in (p) and (l) by `pointwise right Kan extension', and `right $d$-exact' in (e) by `pointwise right $d$-exact'.
  \end{theorem}
  
  \begin{remark}
  	We remark that, in the ``pointwise part'' of the above theorem, the condition that $T$ be a pseudo double monad can be dispensed with whenever $\K$ is an equipment that has opcartesian tabulations.
  	
  	Indeed, consider a cell $\eps$ as on the left below and assume that it defines $r$ as a pointwise right Kan extension in $\K$. Using the arguments given in the proof below it can be lifted to form a $T$-cell in $\Alg wT$. Now, instead of showing that this lift defines a pointwise right Kan extension in $\Alg wT$ in the usual sense (\defref{definition: right Kan extension}), we may equivalently show that, for every strict $T$-morphism $\map fCA$, the composite on the right below defines $r \of f$ as an ordinary right Kan extension in $\Alg wT$, by \thmref{algebraic pointwise Kan extensions in terms of Kan extensions}. Proving the latter amounts to showing that certain factorisations in $\K$, through the composites on the right below, lift to $\Alg wT$; compare the lifting of the factorisation \eqref{T-cell factorisation} in the proof below. However, in contrast to \eqref{T-cell factorisation}, these factorisations will be vertical cells, so that requiring that the compositor $T_\hc$ be invertible is not necessary.
  	\begin{displaymath}
			\begin{tikzpicture}[textbaseline]
  			\matrix(m)[math175em]{A & B \\ M & M \\};
  			\path[map]  (m-1-1) edge[barred] node[above] {$J$} (m-1-2)
														edge node[left] {$r$} (m-2-1)
										(m-1-2) edge node[right] {$d$} (m-2-2);
				\path				(m-2-1) edge[eq] (m-2-2);
				\path[transform canvas={shift={($(m-1-2)!(0,0)!(m-2-2)$)}}] (m-1-1) edge[cell] node[right] {$\eps$} (m-2-1);
			\end{tikzpicture} \qquad\qquad\qquad\qquad\qquad \begin{tikzpicture}[textbaseline]
				\matrix(m)[math175em]{C & B \\ A & B \\ M & M \\};
				\path[map]	(m-1-1) edge[barred] node[above] {$J(f, \id)$} (m-1-2)
														edge node[left] {$f$} (m-2-1)
										(m-2-1) edge[barred] node[below] {$J$} (m-2-2)
														edge node[left] {$r$} (m-3-1)
										(m-2-2) edge node[right] {$d$} (m-3-2);
				\path				(m-1-2) edge[eq] (m-2-2)
										(m-3-1) edge[eq] (m-3-2);
				\path[transform canvas={shift=(m-2-1), yshift=-2pt}]	(m-2-2) edge[cell] node[right] {$\eps$} (m-3-2);
				\draw				($(m-1-1)!0.5!(m-2-2)$)	node {\textup{cart}};
			\end{tikzpicture} 
		\end{displaymath}
  \end{remark}
  
  \begin{proof}[of \thmref{main result}.]
  	We treat the ordinary and pointwise case simultaneously. As usual we denote the algebra structure on $A$ by $A = (A, a, \bar a, \hat a)$, where $\map a{TA}A$ is the structure map, $\cell{\bar a}{a \of Ta}{a \of \mu_A}$ is the associator and $\cell{\hat a}{\id_A}{a \of \eta_A}$ is the unitor; likewise $B = (B, b, \bar b, \hat b)$ and $M = (M, m, \bar m, \hat m)$, while the structure cell $\bar J$ of $J$ is of the form as on the left below. To start we consider a cell $\eps$ in $\K$, as on the right below, and assume that it defines $r$ either as the ordinary right Kan extension of $d$ along $J$ or as the pointwise right Kan extension of $d$ along $J$.
  	\begin{displaymath}
  		\begin{tikzpicture}[textbaseline]
				\matrix(m)[math175em]{TA & TB \\ A & B \\};
				\path[map]  (m-1-1) edge[barred] node[above] {$TJ$} (m-1-2)
														edge node[left] {$a$} (m-2-1)
										(m-1-2) edge node[right] {$b$} (m-2-2)
										(m-2-1) edge[barred] node[below] {$J$} (m-2-2);
				\path[transform canvas={shift={($(m-1-2)!(0,0)!(m-2-2)$)}}] (m-1-1) edge[cell] node[right] {$\bar J$} (m-2-1);			
			\end{tikzpicture} \qquad\qquad\qquad\qquad \begin{tikzpicture}[textbaseline]
  			\matrix(m)[math175em]{A & B \\ M & M \\};
  			\path[map]  (m-1-1) edge[barred] node[above] {$J$} (m-1-2)
														edge node[left] {$r$} (m-2-1)
										(m-1-2) edge node[right] {$d$} (m-2-2);
				\path				(m-2-1) edge[eq] (m-2-2);
				\path[transform canvas={shift={($(m-1-2)!(0,0)!(m-2-2)$)}}] (m-1-1) edge[cell] node[right] {$\eps$} (m-2-1);
			\end{tikzpicture} 
  	\end{displaymath}
  	
  	\emph{Part (a).} If `weak' means `colax' then the algebraic structure on $d$ is given by a cell $\cell{\bar d}{d \of b}{m \of Td}$ as in the composite on the left-hand side below.
  	\begin{equation} \label{colax algebra structure}
  		\begin{tikzpicture}[textbaseline]
  			\matrix(m)[math175em]{TA & TB & TB \\ A & B & TM \\ M & M & M \\};
  			\path[map]	(m-1-1) edge[barred] node[above] {$TJ$} (m-1-2)
  													edge node[left] {$a$} (m-2-1)
  									(m-1-2) edge node[right] {$b$} (m-2-2)
  									(m-1-3) edge node[right] {$Td$} (m-2-3)
  									(m-2-1) edge[barred] node[below] {$J$} (m-2-2)
  													edge node[left] {$r$} (m-3-1)
  									(m-2-2) edge node[right] {$d$} (m-3-2)
  									(m-2-3) edge node[right] {$m$} (m-3-3);
  			\path				(m-1-2) edge[eq] (m-1-3)
  									(m-3-1) edge[eq] (m-3-2)
  									(m-3-2) edge[eq] (m-3-3);
  			\path[transform canvas={shift={($(m-2-2)!0.5!(m-2-3)$)}}]	(m-1-1) edge[cell] node[right] {$\bar J$} (m-2-1)
  									(m-2-1) edge[cell, transform canvas={yshift=-3pt}] node[right] {$\eps$} (m-3-1);
  			\path[transform canvas={shift={($(m-2-3)!0.5!(m-3-2)$)}}]	(m-1-2) edge[cell] node[right] {$\bar d$} (m-2-2);
  		\end{tikzpicture} = \begin{tikzpicture}[textbaseline]
  			\matrix(m)[math175em]{TA & TA & TB \\ A & TM & TM \\ M & M & M \\};
  			\path[map]	(m-1-1) edge node[left] {$a$} (m-2-1)
  									(m-1-2) edge[barred] node[above] {$TJ$} (m-1-3)
  													edge node[left] {$Tr$} (m-2-2)
  									(m-1-3) edge node[right] {$Td$} (m-2-3)
  									(m-2-1) edge node[left] {$r$} (m-3-1)
  									(m-2-2) edge node[left] {$m$} (m-3-2)
  									(m-2-3) edge node[right] {$m$} (m-3-3);
  			\path				(m-1-1) edge[eq] (m-1-2)
  									(m-2-2) edge[eq] (m-2-3)
  									(m-3-1) edge[eq] (m-3-2)
  									(m-3-2) edge[eq] (m-3-3);
  			\path[transform canvas={shift={($(m-2-2)!0.5!(m-2-3)$)}}]	(m-1-2) edge[cell] node[right] {$T\eps$} (m-2-2);
  			\path[transform canvas={shift={($(m-2-3)!0.5!(m-3-2)$)}}]	(m-1-1) edge[cell] node[right] {$\bar r$} (m-2-1);
  		\end{tikzpicture}
  	\end{equation}
  	 Assuming that condition (p) holds, the second column in the right-hand side above defines $m \of Tr$ as the (pointwise) right Kan extension of $m \of Td$ along $TJ$, see \defref{definition: preserving Kan extensions}, so that in both the ordinary and the pointwise case the composite on the left factors uniquely as a vertical cell $\cell{\bar r}{r \of a}{m \of Tr}$, as shown.
  	 
  	 We claim that $\bar r$ makes $r$ into a colax $T$-morphism, that is it satisfies the associativity and unit axioms of \defref{definition: weak algebra morphisms}. To see that it satisfies the former consider the equation of composites below, whose left-hand and right-hand sides are given by the corresponding sides of the associativity axiom for $\bar r$, composed with $T^2 \eps$ on the right.
  	\begin{align*}
  		\begin{tikzpicture}[textbaseline, x=2em, y=1.75em, font=\scriptsize]
  			\draw	(1,2) -- (0,2) -- (0,0) -- (1,0) -- (1,3) -- (2,3) -- (2,1) -- (1,1)
  						(2,2) -- (4,2) -- (4,3) -- (3,3) -- (3,0) -- (2,0) -- (2,1)
  						(3.5,2.5) node {$T^2\eps$};
  			\draw[shift={(0.5,0)}]
  						(0,1) node {$\bar r$}
  						(1,2) node {$T\bar r$}
  						(2,1) node {$\bar m$};
  		\end{tikzpicture} &\mspace{7mu} = \mspace{7mu} \begin{tikzpicture}[textbaseline, x=2em, y=1.75em, font=\scriptsize]
  			\draw (2,2) -- (0,2) -- (0,0) -- (1,0) -- (1,3) -- (2,3) -- (2,1) -- (1,1)
  						(3,1) -- (2,1) -- (2,3) -- (3,3) -- (3,0) -- (4,0) -- (4,2) -- (3,2)
  						(1,1) -- (2,1);
  			\draw[shift={(0.5,0)}]
  						(0,1) node {$\bar r$}
  						(2,2) node {$T\bar d$}
  						(3,1) node {$\bar m$};
  			\draw[shift={(0.5,0.5)}]
  						(1,1) node {$T\eps$}
  						(1,2) node {$T\bar J$};
  		\end{tikzpicture} \mspace{7mu} = \mspace{7mu} \begin{tikzpicture}[textbaseline, x=2em, y=1.75em, font=\scriptsize]
  			\draw (1,0) -- (1,3) -- (0,3) -- (0,0) -- (2,0) -- (2,3) -- (3,3) -- (3,0) -- (4,0) -- (4,2) -- (3,2)
  						(0,1) -- (1,1)
  						(0,2) -- (2,2)
  						(2,1) -- (3,1);
  			\draw[shift={(0.5,0)}]
  						(1,1) node {$\bar d$}
  						(2,2) node {$T\bar d$}
  						(3,1) node {$\bar m$};
  			\draw[shift={(0.5,0.5)}]
  						(0,2) node {$T\bar J$}
  						(0,1) node {$\bar J$}
  						(0,0) node {$\eps$};
  		\end{tikzpicture} \mspace{7mu} = \mspace{7mu} \begin{tikzpicture}[textbaseline, x=2em, y=1.75em, font=\scriptsize]
  			\draw (1,3) -- (1,0) -- (0,0) -- (0,3) -- (2,3) -- (2,0) -- (3,0) -- (3,3) -- (4,3) -- (4,1) -- (3,1)
  						(0,1) -- (2,1)
  						(0,2) -- (1,2)
  						(2,2) -- (3,2);
  			\draw[shift={(0.5,0)}]
  						(1,2) node {$\bar b$}
  						(2,1) node {$\bar d$};
  			\draw[shift={(0.5,0.5)}]
  						(0,2) node {$T\bar J$}
  						(0,1) node {$\bar J$}
  						(0,0) node {$\eps$};
  		\end{tikzpicture} \\[1em]
  		&\mspace{7mu} = \mspace{7mu} \begin{tikzpicture}[textbaseline, x=2em, y=1.75em, font=\scriptsize]
  			\draw	(2,1) -- (0,1) -- (0,3) -- (2,3) -- (2,0) -- (1,0) -- (1,3)
  						(1,2) -- (3,2) -- (3,0) -- (2,0)
  						(3,2) -- (3,3) -- (4,3) -- (4,1) -- (3,1);
  			\draw[shift={(0.5,0)}]
  						(0,2) node {$\bar a$}
  						(2,1) node {$\bar d$};
  			\draw[shift={(0.5,0.5)}]
  						(1,2) node {$\mu_J$}
  						(1,1) node {$\bar J$}
  						(1,0) node {$\eps$};
  		\end{tikzpicture} \mspace{7mu} = \mspace{7mu} \begin{tikzpicture}[textbaseline, x=2em, y=1.75em, font=\scriptsize]
  			\draw (1,1) -- (0,1) -- (0,3) -- (1,3) -- (1,0) -- (2,0) -- (2,3) -- (4,3) -- (4,1) -- (2,1)
  						(1,2) -- (3,2)
  						(3,1) -- (3,3);
  			\draw[shift={(0.5,0)}]
  						(0,2) node {$\bar a$}
  						(1,1) node {$\bar r$};
  			\draw[shift={(0.5,0.5)}]
  						(2,1) node {$T\eps$}
  						(2,2) node {$\mu_J$};
  		\end{tikzpicture} \mspace{7mu} = \mspace{7mu} \begin{tikzpicture}[textbaseline, x=2em, y=1.75em, font=\scriptsize]
  			\draw	(1,1) -- (0,1) -- (0,3) -- (1,3) -- (1,0) -- (2,0) -- (2,3) -- (3,3) -- (3,1) -- (2,1)
  						(3,3) -- (4,3) -- (4,2) -- (3,2)
  						(1,2) -- (2,2);
  			\draw[shift={(0.5,0)}]
  						(0,2) node {$\bar a$}
  						(1,1) node {$\bar r$};
  			\draw[shift={(0.5,0.5)}]
  						(3,2) node {$T^2\eps$};
  		\end{tikzpicture}
  	\end{align*}
  	The identities here follow from the $T$-image the of factorisation \eqref{colax algebra structure}; the factorisation \eqref{colax algebra structure} itself; the associativity axiom for $\bar d$; the associativity axiom for $\bar J$; again \eqref{colax algebra structure}; the naturality of $\mu$. Now notice that the fourth column of the left and right-hand sides equals $1_m \of 1_{\mu_M} \of T^2 \eps = 1_m \of T\eps \of \mu_J$ which, because $\mu_J$ is assumed to be (pointwise) right exact, defines $m \of Tr \of \mu_A$ as the (pointwise) right Kan extension of $m \of Td \of \mu_B$ along $T^2J$. We conclude that the associativity axiom for $\bar r$ holds after it is composed on the right with a cell defining a right Kan extension; since factorisations through such cells are unique the axiom itself follows.
  	
  	\begin{align*}
  		\begin{tikzpicture}[textbaseline, x=2em, y=1.75em, font=\scriptsize]
  			\draw (1,0) -- (1,3) -- (0,3) -- (0,0) -- (2,0) -- (2,3) -- (3,3) -- (3,2) -- (1,2)
  						(1.5,1) node {$\hat m$}
  						(2.5,2.5) node {$\eps$};
  		\end{tikzpicture} &\mspace{13mu} = \mspace{13mu} \begin{tikzpicture}[textbaseline, x=2em, y=1.75em, font=\scriptsize]
				\draw	(1,0) -- (1,3) -- (2,3) -- (2,0)
							(2,2) -- (3,2) -- (3,0) -- (0,0) -- (0,1) -- (1,1)
							(0.5,0.5) node {$\eps$}
							(2.5,1) node {$\hat m$};
  		\end{tikzpicture} \mspace{13mu} = \mspace{13mu} \begin{tikzpicture}[textbaseline, x=2em, y=1.75em, font=\scriptsize]
  			\draw (2,1) -- (0,1) -- (0,0) -- (1,0) -- (1,3) -- (2,3) -- (2,0) -- (3,0) -- (3,3) -- (4,3) -- (4,1) -- (3,1)
  						(2,2) -- (3,2)
  						(0.5,0.5) node {$\eps$};
  			\draw[shift={(0.5,0)}]
  						(1,2) node {$\hat b$}
  						(2,1) node {$\bar d$};
  		\end{tikzpicture} \\[1em]
  		&\mspace{13mu} = \mspace{13mu} \begin{tikzpicture}[textbaseline, x=2em, y=1.75em, font=\scriptsize]
  			\draw	(2,1) -- (0,1) -- (0,3) -- (2,3) -- (2,0) -- (1,0) -- (1,3)
  						(1,2) -- (3,2) -- (3,0) -- (2,0)
  						(3,2) -- (3,3) -- (4,3) -- (4,1) -- (3,1);
  			\draw[shift={(0.5,0)}]
  						(0,2) node {$\hat a$}
  						(2,1) node {$\bar d$};
  			\draw[shift={(0.5,0.5)}]
  						(1,2) node {$\eta_J$}
  						(1,1) node {$\bar J$}
  						(1,0) node {$\eps$};
  		\end{tikzpicture}	\mspace{13mu} = \mspace{13mu} \begin{tikzpicture}[textbaseline, x=2em, y=1.75em, font=\scriptsize]
  			\draw (1,1) -- (0,1) -- (0,3) -- (1,3) -- (1,0) -- (2,0) -- (2,3) -- (4,3) -- (4,1) -- (2,1)
  						(1,2) -- (3,2)
  						(3,1) -- (3,3);
  			\draw[shift={(0.5,0)}]
  						(0,2) node {$\hat a$}
  						(1,1) node {$\bar r$};
  			\draw[shift={(0.5,0.5)}]
  						(2,1) node {$T\eps$}
  						(2,2) node {$\eta_J$};
  		\end{tikzpicture} \mspace{13mu} = \mspace{13mu} \begin{tikzpicture}[textbaseline, x=2em, y=1.75em, font=\scriptsize]
  			\draw	(1,1) -- (0,1) -- (0,3) -- (1,3) -- (1,0) -- (2,0) -- (2,3) -- (3,3) -- (3,1) -- (2,1)
  						(3,3) -- (4,3) -- (4,2) -- (3,2)
  						(1,2) -- (2,2);
  			\draw[shift={(0.5,0)}]
  						(0,2) node {$\hat a$}
  						(1,1) node {$\bar r$};
  			\draw[shift={(0.5,0.5)}]
  						(3,2) node {$\eps$};
  		\end{tikzpicture}
  	\end{align*}
  	
  	The unit axiom for $\bar r$ is proved similarly: we consider the equation above, whose left and right-hand sides coincide with those of the unit axiom for $\bar r$ composed with $\eps$ on the right. The identities follow from the unit axiom for $\bar d$, the unit axiom for $\bar J$, the factorisation \eqref{colax algebra structure} and the naturality of $\eta$. As before, the unit axiom for $\bar r$ follows from the fact that the last columns in the left and right-hand sides below are given by $1_m \of 1_{\eta_M} \of \eps = 1_m \of T\eps \of \eta_J$, which defines $m \of Tr \of \eta_A$ as the (pointwise) right Kan extension of $m \of Td \of \eta_B$ along $J$, because $\eta_J$ is (pointwise) right exact. We conclude that the cell $\bar r$ forms a colax structure cell for $\map rBM$. Moreover, notice that with this algebra structure on $r$ the factorisation \eqref{colax algebra structure} forms the $T$-cell axiom for $\eps$. In fact, since this factorisation is unique, this is the only way to lift the $\K$-cell $\eps$ to $\Alg cT$.
  	
  	Having obtained an algebra structure on $r$ it remains to show that $\eps$, as a $T$-cell, defines $r$ as the (pointwise) right Kan extension of $d$ along $J$. To see this consider a $T$-cell $\phi$ as on the left below where, in the case of an ordinary Kan extension, we assume $\hmap HCA$ to be the horizontal unit $1_A$. Since $\eps$ defines $r$ as the (pointwise) right Kan extension in $\K$ we obtain a unique factorisation $\phi'$ as shown below; we have to prove that $\phi'$ is a $T$-cell.
  	\begin{equation} \label{T-cell factorisation}
  		\begin{tikzpicture}[textbaseline]
				\matrix(m)[math175em]{C & A & B \\ M & & M \\};
				\path[map]	(m-1-1) edge[barred] node[above] {$H$} (m-1-2)
														edge node[left] {$s$} (m-2-1)
										(m-1-2) edge[barred] node[above] {$J$} (m-1-3)
										(m-1-3) edge node[right] {$d$} (m-2-3)
										(m-1-2) edge[cell] node[right] {$\phi$} (m-2-2);
				\path				(m-2-1) edge[eq] (m-2-3);
			\end{tikzpicture} = \begin{tikzpicture}[textbaseline]
				\matrix(m)[math175em]{C & A & B \\ M & M & M \\};
				\path[map]	(m-1-1) edge[barred] node[above] {$H$} (m-1-2)
														edge node[left] {$s$} (m-2-1)
										(m-1-2) edge[barred] node[above] {$J$} (m-1-3)
														edge node[right] {$r$} (m-2-2)
										(m-1-3) edge node[right] {$d$} (m-2-3);
				\path				(m-2-1) edge[eq] (m-2-2)
										(m-2-2) edge[eq] (m-2-3);
				\path[transform canvas={shift={($(m-1-2)!0.5!(m-2-3)$)}}] (m-1-1) edge[cell] node[right] {$\phi'$} (m-2-1)
										(m-1-2) edge[cell] node[right] {$\eps$} (m-2-2);
			\end{tikzpicture}
  	\end{equation}
  	To see this first notice that the compositor $\cell{T_\hc}{T(H \hc J)}{TH \hc TJ}$ is invertible in either case. Indeed if $r$ is an ordinary Kan extension then $H = 1_A$, so that this follows from the unit axiom for $T$ (see \eqref{unit axiom for normal oplax double functors}). On the other hand, if $r$ is a pointwise Kan extension then $T$ is assumed to be a pseudo double monad.
  	
  	Next consider the equation below, whose identities follow from the naturality of $T_\hc$ and the factorisation above; the $T$-cell axiom for $\phi$ (remember that $\ol{H \hc J} = (\bar H \hc \bar J) \of T_\hc$, see \eqref{structure cell of horizontal composite}); the factorisation above; the factorisation \eqref{colax algebra structure}.
  	\begin{align*}
  		\begin{tikzpicture}[textbaseline, x=2em, y=1.75em, font=\scriptsize]
  			\draw (3,2) -- (3,3) -- (1,3) -- (1,0) -- (0,0) -- (0,2) -- (3,2) -- (3,1) -- (1,1)
  						(2,1) -- (2,2)
  						(0.5,1) node {$\bar s$}
  						(1.5,1.5) node {$T\phi'$}
  						(2.5,1.5) node {$T\eps$}
  						(2,2.5) node {$T_\hc$};
  		\end{tikzpicture} \mspace{13mu} &= \mspace{13mu} \begin{tikzpicture}[textbaseline, x=2em, y=1.75em, font=\scriptsize]
  			\draw (1,1) -- (3,1) -- (3,2) -- (0,2) -- (0,0) -- (1,0) -- (1,2)
  						(0.5,1) node {$\bar s$}
  						(2,1.5) node {$T\phi$};
  		\end{tikzpicture} \mspace{13mu} = \mspace{13mu} \begin{tikzpicture}[textbaseline, x=2em, y=1.75em, font=\scriptsize]
  			\draw (2,1) -- (0,1) -- (0,0) -- (3,0) -- (3,2) -- (0,2) -- (0,1)
  						(1,1) -- (1,2)
  						(2,0) -- (2,2)
  						(0,2) -- (0,3) -- (2,3) -- (2,2)
							(2.5,1) node {$\bar d$}
							(1,0.5) node {$\phi$}
							(1,2.5) node {$T_\hc$};
				\draw[shift={(0.5,0.5)}]
							(0,1) node {$\bar H$}
							(1,1) node {$\bar J$};
  		\end{tikzpicture} \\[1em]
  		\mspace{13mu} &= \mspace{13mu} \begin{tikzpicture}[textbaseline, x=2em, y=1.75em, font=\scriptsize]
  			\draw (2,1) -- (0,1) -- (0,0) -- (3,0) -- (3,2) -- (0,2) -- (0,1)
  						(1,0) -- (1,2)
  						(2,0) -- (2,2)
  						(0,2) -- (0,3) -- (2,3) -- (2,2)
							(2.5,1) node {$\bar d$}
							(1,2.5) node {$T_\hc$};
				\draw[shift={(0.5,0.5)}]
							(0,0) node {$\phi'$}
							(0,1) node {$\bar H$}
							(1,0) node {$\eps$}
							(1,1) node {$\bar J$};
  		\end{tikzpicture} \mspace{13mu} = \mspace{13mu} \begin{tikzpicture}[textbaseline, x=2em, y=1.75em, font=\scriptsize]
  			\draw (2,2) -- (2,0) -- (0,0) -- (0,2) -- (3,2) -- (3,1) -- (2,1)
  						(1,0) -- (1,2)
  						(0,1) -- (1,1)
  						(0,2) -- (0,3) -- (3,3) -- (3,2)
  						(1.5,2.5) node {$T_\hc$}
  						(1.5,1) node {$\bar r$};
  			\draw[shift={(0.5,0.5)}]
  						(0,0) node {$\phi'$}
  						(0,1) node {$\bar H$}
  						(2,1) node {$T\eps$};
  		\end{tikzpicture}
  	\end{align*}
  	Precomposing the equation above with $\inv{T_\hc}$ we find that both sides of the $T$-cell axiom for $\phi'$ coincide after composition with $1_m \of T\eps$ on the right. Since the latter defines a (pointwise) right Kan extension the $T$-cell axiom itself follows. This finishes the proof of part (a).
  	
  	\emph{Part (b).} In the case that $d$ is a lax $T$-morphism, with structure cell $\cell{\bar d}{m \of Td}{d \of b}$, and the structure cell $\bar J$ of $J$ is (pointwise) right exact, so that the second column in the right-hand side below defines $r \of a$ as the (pointwise) right Kan extension of $d \of b$ along $TJ$, then the composite on the left below factors uniquely as a cell $\cell{\bar r}{m \of Tr}{r \of a}$ as shown.
  	\begin{equation} \label{lax algebra structure}
  		\begin{tikzpicture}[textbaseline]
  			\matrix(m)[math175em]{TA & TB & TB \\ TM & TM & B \\ M & M & M \\};
  			\path[map]	(m-1-1) edge[barred] node[above] {$TJ$} (m-1-2)
  													edge node[left] {$Tr$} (m-2-1)
  									(m-1-2) edge node[right] {$Td$} (m-2-2)
  									(m-1-3) edge node[right] {$b$} (m-2-3)
  									(m-2-1) edge node[left] {$m$} (m-3-1)
  									(m-2-2) edge node[right] {$m$} (m-3-2)
  									(m-2-3) edge node[right] {$d$} (m-3-3);
  			\path				(m-1-2) edge[eq] (m-1-3)
  									(m-2-1) edge[eq] (m-2-2)
  									(m-3-1) edge[eq] (m-3-2)
  									(m-3-2) edge[eq] (m-3-3);
  			\path[transform canvas={shift={($(m-2-2)!0.5!(m-2-3)$)}}]	(m-1-1) edge[cell] node[right] {$T\eps$} (m-2-1);
  			\path[transform canvas={shift={($(m-2-3)!0.5!(m-3-2)$)}}]	(m-1-2) edge[cell] node[right] {$\bar d$} (m-2-2);
  		\end{tikzpicture} = \begin{tikzpicture}[textbaseline]
  			\matrix(m)[math175em]{TA & TA & TB \\ TM & A & B \\ M & M & M \\};
  			\path[map]	(m-1-1) edge node[left] {$Tr$} (m-2-1)
  									(m-1-2) edge[barred] node[above] {$TJ$} (m-1-3)
  													edge node[left] {$a$} (m-2-2)
  									(m-1-3) edge node[right] {$b$} (m-2-3)
  									(m-2-1) edge node[left] {$m$} (m-3-1)
  									(m-2-2) edge[barred] node[below] {$J$} (m-2-3)
  													edge node[left] {$r$} (m-3-2)
  									(m-2-3) edge node[right] {$d$} (m-3-3);
  			\path				(m-1-1) edge[eq] (m-1-2)
  									(m-3-1) edge[eq] (m-3-2)
  									(m-3-2) edge[eq] (m-3-3);
  			\path[transform canvas={shift={($(m-2-2)!0.5!(m-2-3)$)}}]	(m-1-2) edge[cell] node[right] {$\bar J$} (m-2-2)
  									(m-2-2) edge[cell, transform canvas={yshift=-3pt}] node[right] {$\eps$} (m-3-2);
  			\path[transform canvas={shift={($(m-2-3)!0.5!(m-3-2)$)}}]	(m-1-1) edge[cell] node[right] {$\bar r$} (m-2-1);
  		\end{tikzpicture}
  	\end{equation}
  	Analogous to the proof of part (a) above, we can show that the cell $\bar r$ satisfies the associativity and unit axioms, so that it makes $\map rAM$ into a lax $T$-morphism. 
  	
  	Indeed, the associativity axiom follows from the identity below, which itself follows from the associativity axiom for $\bar J$, the factorisation above, the associativity axiom for $\bar d$ and the naturality of $\mu$. As before notice that the identity below is the associativity axiom for $\bar r$ composed on the right with the composite $\eps \of \bar J \of \mu_J$. This composite defines a (pointwise) right Kan extension, since both $\mu_J$ and $\bar J$ are assumed to be (pointwise) right exact, so that the associativity axiom itself follows.
  	\begin{displaymath}
  		\begin{tikzpicture}[textbaseline, x=2em, y=1.75em, font=\scriptsize]
  			\draw (3,3) -- (3,0) -- (4,0) -- (4,3) -- (2,3) -- (2,0) -- (1,0) -- (1,3) -- (0,3) -- (0,1) -- (1,1)
  						(4,1) -- (2,1)
  						(4,2) -- (3,2)
  						(2,2) -- (1,2);
  			\draw[shift={(0.5,0)}]
  						(0,2) node {$T\bar r$}
  						(1,1) node {$\bar r$}
  						(2,2) node {$\bar a$};
  			\draw[shift={(0.5,0.5)}]
  						(3,2) node {$\mu_J$}
  						(3,1) node {$\bar J$}
  						(3,0) node {$\eps$};
  		\end{tikzpicture} \mspace{13mu} = \mspace{13mu} \begin{tikzpicture}[textbaseline, x=2em, y=1.75em, font=\scriptsize]
  			\draw (3,0) -- (3,3) -- (4,3) -- (4,0) -- (2,0) -- (2,3) -- (1,3) -- (1,0) -- (0,0) -- (0,2) -- (1,2)
  						(4,1) -- (3,1)
  						(4,2) -- (2,2)
  						(1,1) -- (2,1);
  			\draw[shift={(0.5,0)}]
  						(2,1) node {$\bar r$}
  						(0,1) node {$\bar m$};
  			\draw[shift={(0.5,0.5)}]
  						(3,2) node {$\mu_J$}
  						(3,1) node {$\bar J$}
  						(3,0) node {$\eps$};
  		\end{tikzpicture}
  	\end{displaymath}
  	The unit axiom for $\bar r$ follows in the same way from the identity on the left below, which itself follows from the unit axioms for $\bar J$ and $\bar d$, the factorisation above and the naturality of $\eta$.
  	\begin{displaymath}
  		\begin{tikzpicture}[textbaseline, x=2em, y=1.75em, font=\scriptsize]
  			\draw (1,3) -- (1,0) -- (0,0) -- (0,3) -- (3,3) -- (3,0) -- (2,0) -- (2,3)
  						(1,1) -- (3,1)
  						(2,2) -- (3,2)
  						(1.5,2) node {$\hat a$};
  			\draw[shift={(0.5,0.5)}]
  						(2,0) node {$\eps$}
  						(2,1) node {$\bar J$}
  						(2,2) node {$\eta_J$};
  		\end{tikzpicture} \mspace{13mu} = \mspace{13mu} \begin{tikzpicture}[textbaseline, x=2em, y=1.75em, font=\scriptsize]
  			\draw (3,0) -- (3,3) -- (4,3) -- (4,0) -- (2,0) -- (2,3) -- (1,3) -- (1,0) -- (0,0) -- (0,2) -- (1,2)
  						(4,1) -- (3,1)
  						(4,2) -- (2,2)
  						(1,1) -- (2,1);
  			\draw[shift={(0.5,0)}]
  						(2,1) node {$\bar r$}
  						(0,1) node {$\hat m$};
  			\draw[shift={(0.5,0.5)}]
  						(3,2) node {$\eta_J$}
  						(3,1) node {$\bar J$}
  						(3,0) node {$\eps$};
  		\end{tikzpicture} \qquad\qquad \begin{tikzpicture}[textbaseline, x=2em, y=1.75em, font=\scriptsize]
  			\draw (1,2) -- (1,0) -- (3,0) -- (3,2) -- (0,2) -- (0,1) -- (1,1)
  						(2,0) -- (2,2)
  						(3,1) -- (2,1)
  						(0,2) -- (0,3) -- (3,3) -- (3,2)
  						(1.5,2.5) node {$T_\hc$}
  						(1.5,1) node {$\bar r$};
  			\draw[shift={(0.5,0.5)}]
  						(2,0) node {$\eps$}
  						(2,1) node {$\bar J$}
  						(0,1) node {$T\phi'$};
  		\end{tikzpicture} \mspace{13mu} = \mspace{13mu} \begin{tikzpicture}[textbaseline, x=2em, y=1.75em, font=\scriptsize]
  			\draw (1,1) -- (3,1) -- (3,0) -- (0,0) -- (0,2) -- (3,2) -- (3,1)
  						(2,0) -- (2,2)
  						(1,0) -- (1,2)
  						(1,2) -- (1,3) -- (3,3) -- (3,2)
  						(2,2.5) node {$T_\hc$}
							(0.5,1) node {$\bar s$};
				\draw[shift={(0.5,0.5)}]
							(2,0) node {$\eps$}
							(2,1) node {$\bar J$}
							(1,0) node {$\phi'$}
							(1,1) node {$\bar H$};
  		\end{tikzpicture}
  	\end{displaymath}
  	
  	Having given $\map rAM$ a lax $T$-morphism structure, we notice that the factorisation \eqref{lax algebra structure} forms the $T$-cell axiom for $\eps$. It remains to show that, as a $T$-cell, $\eps$ defines $r$ as the (pointwise) right Kan extension of $d$ along $J$. To do so we consider a $T$-cell $\phi$ as in \eqref{T-cell factorisation}, where now $s$ and $d$ are lax $T$-morphisms. Because $\eps$ defines $r$ as a (pointwise) right Kan extension, $\phi$ factors uniquely through $\eps$ as $\phi'$, as in \eqref{T-cell factorisation}, and we have to show that $\phi'$ is a $T$-cell. That the $T$-cell axiom for $\phi'$ holds when composed with $\eps \of \bar J$ on the right, and then with $\cell{T_\hc}{T(H \hc J)}{TH \hc TJ}$ on top, as shown on the right above, follows from the $T$-cell axiom for $\eps$ \eqref{lax algebra structure}, the factorisation $\phi = \phi' \hc \eps$, the naturality of $T_\hc$ and the $T$-cell axiom for $\phi$. Since $T_\hc$ is invertible and $\eps \of \bar J$ defines a (pointwise) right Kan extension the $T$-cell axiom for $\phi'$ follows. This concludes the proof of part (b).
  	
  	\emph{Part (c).} In this case $d$ is a pseudo $T$-morphism, with invertible structure cell \mbox{$\cell{\bar d}{m \of Td}{d \of b}$}. If two of (p), (e) and (l) hold then one of (p) and (e) must hold. Suppose that (e) holds; we will show that (p) and (l) are equivalent. By part (b) the cell $\eps$, that defines $r$ as a right Kan extension, lifts along $\Alg lT \to \K$ to a $T$-cell defining $(r, \bar r)$ as a right Kan extension in $\Alg lT$, where $\cell{\bar r}{m \of Tr}{r \of a}$ is uniquely determined by the identity \eqref{lax algebra structure}. If (p) holds then, using the fact that $\bar d$ is invertible, the left-hand side of \eqref{lax algebra structure} defines the right Kan extension of $d \of b$ along $TJ$. Because the second column in the right hand side of \eqref{lax algebra structure} does too, it follows that the factorisation $\bar r$ is invertible, so that the lift $(r, \bar r)$ of $r$ is a pseudomorphism; we conclude that condition (l) holds.
  	
  	On the other hand, if (l) holds then $r$ can be given the structure of a pseudo $T$\ndash mor\-phism such that $\eps$ forms a $T$-cell that defines $r$ as a right Kan extension in $\Alg{ps}T$. Since the lift of $\eps$ that we obtained by applying part (b) is unique, the structure cell $m \of Tr \Rar r \of a$ of the lift of $r$ to $\Alg{ps}T$ must coincide with the structure cell $\bar r$ that we obtained in \eqref{lax algebra structure}. Thus $\bar r$ is invertible and hence, by composing \eqref{lax algebra structure} on the right with $\inv{\bar d}$, we see that $1_m \of T\eps$ defines $m \of Tr$ as a right Kan extension. This means that the algebraic structure of $M$ preserves the right Kan extension of $d$ along $J$, that is (p) holds. We conclude that if (e) holds then (p) and (l) are equivalent. By applying analogous arguments to the factorisation \eqref{colax algebra structure} one shows that the equivalence of (e) and (l) follows from (p). This completes the proof of part (c).
  \end{proof}

 In the proposition below we apply \thmref{horizontal dual main theorem}, which is the horizontal dual of the theorem above, to the double monad of free strict monoidal $\V$-categories. Consider a (symmetric) monoidal $\V$-profunctor \mbox{$\hmap JAB$} (\exref{example: monoidal enriched profunctors}). Horizontally dual to the factorisation considered in \propref{pointwise exact cells}, its structure transformation $J_\tens$ factors uniquely as a transformation \mbox{$\cell{(J_\tens)^*}{\tens^*_A \hc TJ}{J(\id, \tens_B)}$}, of $\V$-profunctors $A \brar TB$, that is induced by the $\V$-maps
  \begin{multline*}
  	A\bigpars{x, (y_1 \tens \dotsb \tens y_n)} \tens J(y_1, z_1) \tens \dotsb \tens J(y_n, z_n) \xrar{\id \tens J_\tens} \\
  	A\bigpars{x, (y_1 \tens \dotsb \tens y_n)} \tens J\bigpars{(y_1 \tens \dotsb \tens y_n), (z_1 \tens \dotsb \tens z_n)} \xrar\lambda J\bigpars{x, (z_1 \tens \dotsb \tens z_n)}
  \end{multline*}
  where $x \in A$, $\ul y \in TA$ and $\ul z \in TB$. We say that $J$ satisfies the \emph{left Beck-Chevalley condition} if the transformation $(J_\tens)^*$ is an isomorphism.
  
  In the proposition below $\MonProf\V w$ (respectively $\sMonProf\V w$) denotes the double category of (symmetric) monoidal $\V$-categories, (symmetric) weak monoidal $\V$-functors, (symmetric) monoidal $\V$-profunctors and their transformations. For a double functor $\map F\K\L$, a horizontal morphism $\hmap JAB$ and an object $M$ in $\L$, we say that $F$ \emph{lifts all pointwise left Kan extensions along $J$ into $M$} if, for every vertical morphism $\map dAM$ in $\L$, $F$ lifts the pointwise left Kan extension of $d$ along $J$.
  
  \begin{proposition}\label{lifting pointwise monoidal left Kan extensions}
  	Let $\V$ be a cocomplete symmetric monoidal category whose tensor product preserves colimits on both sides, let $A$, $B$ and $M$ be monoidal $\V$-categories and let $\hmap JAB$ be a monoidal $\V$\ndash profunctor. Consider the following conditions:
  	\begin{enumerate}
  		\item[(p)]	the tensor product $\map{\tens_2}{M \tens M}M$ preserves pointwise left Kan extensions along $J$ in each variable;
  		\item[(e)]	$J$ satisfies the left Beck-Chevalley condition.
  	\end{enumerate}
  	If condition (e) holds then the forgetful double functor $\MonProf\V c \to \enProf\V$ lifts all pointwise left Kan extensions along $J$ into $M$. If (p) holds then so does \mbox{$\MonProf\V l \to \enProf\V$}; in the case that both conditions hold the same follows for $\MonProf\V{ps} \to \enProf\V$.
  	
  	Analogous assertions hold for the double functors $\sMonProf\V w \to \enProf\V$, for each $\textup w \in \set{\textup c, \textup l, \textup{ps}}$, in the case that $A$, $B$ and $M$ are symmetric monoidal $\V$\ndash cat\-e\-gories and $\hmap JAB$ is a symmetric monoidal $\V$-profunctor.
  \end{proposition}
  \begin{proof}
  	The assertions for the forgetful double functors $\MonProf\V w \to \enProf\V$ follow immediately from applying \thmref{horizontal dual main theorem} to the `free strict monoidal $\V$\ndash cat\-e\-gory'\ndash monad $T$ of \exref{example: free strict monoidal enriched category monad}. Indeed, $T$ is pointwise left exact (\exref{free strict monoidal category monad is pointwise exact}) and the conditions (p) and (e) above imply the corresponding conditions in \thmref{horizontal dual main theorem}, by \exref{example: pointwise left Kan extensions preserved by colax monoidal structure} and the horizontal dual of \propref{pointwise exact cells}(b).
  	
  	To see that the assertions remain true in the symmetric case we claim that, for each choice of $\textup w \in \set{\textup c, \textup l, \textup{ps}}$, any lift of a pointwise left Kan extension along the forgetful double functor $\MonProf\V w \to \enProf\V$ can be further lifted along the forgetful double functor \mbox{$\sMonProf\V w \to \MonProf\V w$}. We will prove this claim in the case that $\textup w = \textup l$, in the other cases it can be proved analogously. Hence consider symmetric monoidal $\V$\ndash categories $A$, $B$ and $M$, a symmetric lax monoidal $\V$-functor $\map dAM$ and a symmetric monoidal $\V$-profunctor $\hmap JAB$, and assume that in $\enProf\V$ the transformation $\eta$ on the left below defines a $\V$-functor $\map lBM$ as the pointwise left Kan extension of $d$ along $J$.
  	\begin{equation} \label{lax monoidal structure}
			\begin{tikzpicture}[textbaseline]
  			\matrix(m)[math175em]{A & B \\ M & M \\};
  			\path[map]  (m-1-1) edge[barred] node[above] {$J$} (m-1-2)
														edge node[left] {$d$} (m-2-1)
										(m-1-2) edge node[right] {$l$} (m-2-2);
				\path				(m-2-1) edge[eq] (m-2-2);
				\path[transform canvas={shift={($(m-1-2)!(0,0)!(m-2-2)$)}}] (m-1-1) edge[cell] node[right] {$\eta$} (m-2-1);
			\end{tikzpicture} \qquad\quad \begin{tikzpicture}[textbaseline]
  			\matrix(m)[math175em]{TA & TA & TB \\ TM & A & B \\ M & M & M \\};
  			\path[map]	(m-1-1) edge node[left] {$Td$} (m-2-1)
  									(m-1-2) edge[barred] node[above] {$TJ$} (m-1-3)
  													edge node[left] {$\tens$} (m-2-2)
  									(m-1-3) edge node[right] {$\tens$} (m-2-3)
  									(m-2-1) edge node[left] {$\tens$} (m-3-1)
  									(m-2-2) edge[barred] node[below] {$J$} (m-2-3)
  													edge node[left] {$d$} (m-3-2)
  									(m-2-3) edge node[right] {$l$} (m-3-3);
  			\path				(m-1-1) edge[eq] (m-1-2)
  									(m-3-1) edge[eq] (m-3-2)
  									(m-3-2) edge[eq] (m-3-3);
  			\path[transform canvas={shift={($(m-2-2)!0.5!(m-2-3)$)}}]	(m-1-2) edge[cell] node[right] {$J_\tens$} (m-2-2)
  									(m-2-2) edge[cell, transform canvas={yshift=-3pt}] node[right] {$\eta$} (m-3-2);
  			\path[transform canvas={shift={($(m-2-3)!0.5!(m-3-2)$)}}]	(m-1-1) edge[cell] node[right] {$d_\tens$} (m-2-1);
  		\end{tikzpicture} = \begin{tikzpicture}[textbaseline]
  			\matrix(m)[math175em]{TA & TB & TB \\ TM & TM & B \\ M & M & M \\};
  			\path[map]	(m-1-1) edge[barred] node[above] {$TJ$} (m-1-2)
  													edge node[left] {$Td$} (m-2-1)
  									(m-1-2) edge node[right] {$Tl$} (m-2-2)
  									(m-1-3) edge node[right] {$\tens$} (m-2-3)
  									(m-2-1) edge node[left] {$\tens$} (m-3-1)
  									(m-2-2) edge node[right] {$\tens$} (m-3-2)
  									(m-2-3) edge node[right] {$l$} (m-3-3);
  			\path				(m-1-2) edge[eq] (m-1-3)
  									(m-2-1) edge[eq] (m-2-2)
  									(m-3-1) edge[eq] (m-3-2)
  									(m-3-2) edge[eq] (m-3-3);
  			\path[transform canvas={shift={($(m-2-2)!0.5!(m-2-3)$)}}]	(m-1-1) edge[cell] node[right] {$T\eta$} (m-2-1);
  			\path[transform canvas={shift={($(m-2-3)!0.5!(m-3-2)$)}}]	(m-1-2) edge[cell] node[right] {$l_\tens$} (m-2-2);
  		\end{tikzpicture}
		\end{equation}
		If the condition (p) holds then $1_\tens \of T\eta$ defines $\tens \of Tl$ as a pointwise left Kan extension and, horizontally dual to \eqref{colax algebra structure}, the lax monoidal structure \mbox{$\nat{l_\tens}{\tens \of Tl}{l \of \tens}$} that makes $l$ into a lax monoidal $\V$-functor is obtained as the unique factorisation on the right above.
		\begin{displaymath}
			\begin{tikzpicture}[textbaseline]
				\matrix(m)[math175em]{B^{\tens n} & B^{\tens n} & B^{\tens n} \\
					M^{\tens n} & & B^{\tens n} \\
					& B & B \\
					M & M & M \\};
				\path[map]	(m-1-1) edge node[left] {$l^{\tens n}$} (m-2-1)
										(m-1-2) edge node[left, inner sep=1pt] {$\tens_n$} (m-3-2)
										(m-1-3) edge node[right] {$\sigma$} (m-2-3)
										(m-2-1) edge node[left] {$\tens_n$} (m-4-1)
										(m-2-3) edge node[right] {$\tens_n$} (m-3-3)
										(m-3-2) edge node[right] {$l$} (m-4-2)
										(m-3-3) edge node[right] {$l$} (m-4-3);
				\path				(m-1-1) edge[eq] (m-1-2)
										(m-1-2) edge[eq] (m-1-3)
										(m-3-2) edge[eq] (m-3-3)
										(m-4-1) edge[eq] (m-4-2)
										(m-4-2) edge[eq] (m-4-3);
				\path[transform canvas={shift=($(m-2-3)!0.5!(m-3-2)$)}]	(m-2-1) edge[cell] node[right] {$l_{\tens_n}$} (m-3-1);
				\path[transform canvas={shift=($(m-3-2)!0.5!(m-3-3)$), xshift=-0.6em}]	(m-1-2) edge[cell] node[right] {$\mf s_\sigma$} (m-2-2);
			\end{tikzpicture} = \begin{tikzpicture}[textbaseline]
				\matrix(m)[math175em]{B^{\tens n} & B^{\tens n} & B^{\tens n} & B^{\tens n} \\
					M^{\tens n} & M^{\tens n} & B^{\tens n} & B^{\tens n} \\
					& M^{\tens n} & M^{\tens n} & B \\
					M & M & M & M \\};
				\path[map]	(m-1-1) edge node[left] {$l^{\tens n}$} (m-2-1)
										(m-1-2) edge node[right] {$l^{\tens n}$} (m-2-2)
										(m-1-3) edge node[right] {$\sigma$} (m-2-3)
										(m-1-4) edge node[right] {$\sigma$} (m-2-4)
										(m-2-1) edge node[left] {$\tens_n$} (m-4-1)
										(m-2-2) edge node[right] {$\sigma$} (m-3-2)
										(m-2-3) edge node[right] {$l^{\tens n}$} (m-3-3)
										(m-2-4) edge node[right] {$\tens_n$} (m-3-4)
										(m-3-2) edge node[right] {$\tens_n$} (m-4-2)
										(m-3-3) edge node[right] {$\tens_n$} (m-4-3)
										(m-3-4) edge node[right] {$l$} (m-4-4);
				\path				(m-1-1) edge[eq] (m-1-2)
										(m-1-2) edge[eq] (m-1-3)
										(m-1-3) edge[eq] (m-1-4)
										(m-2-1) edge[eq] (m-2-2)
										(m-2-3) edge[eq] (m-2-4)
										(m-3-2) edge[eq] (m-3-3)
										(m-4-1) edge[eq] (m-4-2)
										(m-4-2) edge[eq] (m-4-3)
										(m-4-3) edge[eq] (m-4-4);
				\path[transform canvas={shift=(m-3-3), xshift=-0.5em}]	(m-2-1) edge[cell] node[right] {$\mf s_\sigma$} (m-3-1)
										(m-2-3) edge[cell] node[right] {$l_{\tens_n}$} (m-3-3);
			\end{tikzpicture}
		\end{displaymath}
		
		We claim that $l_\tens$ is compatible with the symmetries on $B$ and $M$, in the sense that it makes a diagram like \eqref{symmetry axiom} commute. Rewritten in terms of vertical cells in $\enProf\V$, this means that the identity above holds; here $l_{\tens_n} = l_\tens \of 1_{j_n}$, where $\nat{j_n}{(\dash)^{\tens n}}T$ denotes the double transformation that includes the $n$th tensor powers into $T$. The following equation, where $d_{\tens_n} = d_\tens \of 1_{j_n}$, $J_{\tens_n} = J_\tens \of (j_n)_J$ and whose identities are described below, shows that the two sides above coincide after composition on the left with $1_{\tens_n} \of \eta^{\tens n}$. The latter defines a pointwise left Kan extension because $1_\tens \of T\eta$ does, by \lemref{pointwise Kan extensions along TJ}, so that the identity above follows.
		\begin{align*}
			\begin{tikzpicture}[textbaseline, x=2em, y=1.75em, font=\scriptsize]
				\draw	(1,2) -- (0,2) -- (0,3) -- (2,3) -- (2,0) -- (1,0) -- (1,3)
							(2,3) -- (3,3) -- (3,1) -- (2,1);
				\draw[shift={(0.5,0.5)}]	(0,2) node {$\eta^{\tens n}$}
							(1,1) node {$l_{\tens_n}$}
							(2,1.5) node {$\mf s_\sigma$};
			\end{tikzpicture} &\mspace{13mu} = \mspace{13mu} \begin{tikzpicture}[textbaseline, x=2em, y=1.75em, font=\scriptsize]
				\draw (1,3) -- (1,0) -- (0,0) -- (0,3) -- (2,3) -- (2,0) -- (1,0)
							(1,1) -- (3,1) -- (3,3) -- (2,3);
				\draw[shift={(0.5,0.5)}]	(0,1) node {$d_{\tens_n}$}
							(1,0) node {$\eta$};
				\draw[shift={(0.5,0)}]	(1,2) node {$J_{\tens_n}$}
							(2,2) node {$\mf s_\sigma$};
			\end{tikzpicture} \mspace{13mu} = \mspace{13mu} \begin{tikzpicture}[textbaseline, x=2em, y=1.75em, font=\scriptsize]
				\draw (1,3) -- (1,0) -- (0,0) -- (0,3) -- (3,3) -- (3,1) -- (1,1)
							(2,3) -- (2,0) -- (3,0) -- (3,1)
							(2,2) -- (3,2);
				\draw[shift={(0.5,0.5)}]	(0,1) node {$d_{\tens_n}$}
							(1,1.5) node {$\mf s_\sigma$}
							(2,0) node {$\eta$}
							(2,1) node {$J_{\tens_n}$}
							(2,2) node {$\sigma_J$};
			\end{tikzpicture} \\[1em]
			&\mspace{13mu} = \mspace{13mu} \begin{tikzpicture}[textbaseline, x=2em, y=1.75em, font=\scriptsize]
				\draw (1,2) -- (0,2) -- (0,0) -- (1,0) -- (1,3) -- (2,3) -- (2,0) -- (4,0) -- (4,3) -- (3,3) -- (3,0)
							(1,1) -- (2,1)
							(3,1) -- (4,1)
							(2,2) -- (4,2);
				\draw[shift={(0.5,0)}]	(0,1) node {$\mf s_\sigma$}
							(2,1) node {$d_{\tens_n}$};
				\draw[shift={(0.5,0.5)}]	(3,0) node {$\eta$}
							(3,1) node {$J_{\tens_n}$}
							(3,2) node {$\sigma_J$};
			\end{tikzpicture} \mspace{13mu} = \mspace{13mu} \begin{tikzpicture}[textbaseline, x=2em, y=1.75em, font=\scriptsize]
				\draw (1,2) -- (0,2) -- (0,0) -- (1,0) -- (1,3) -- (2,3) -- (2,1) -- (3,1) -- (3,3) -- (2,3)
							(1,1) -- (2,1)
							(2,2) -- (4,2) -- (4,0) -- (3,0) -- (3,1);
				\draw[shift={(0.5,0)}]	(0,1) node {$\mf s_\sigma$}
							(3,1) node {$l_{\tens_n}$};
				\draw[shift={(0.5,0.5)}]	(2,1) node {$\eta^{\tens n}$}
							(2,2) node {$\sigma_J$};
			\end{tikzpicture} \mspace{13mu} = \mspace{13mu} \begin{tikzpicture}[textbaseline, x=2em, y=1.75em, font=\scriptsize]
				\draw	(2,2) -- (0,2) -- (0,3) -- (1,3) -- (1,0) -- (2,0) -- (2,3) -- (3,3) -- (3,0) -- (4,0) -- (4,2) -- (3,2)
							(2,1) -- (3,1);
				\draw[shift={(0.5,0)}]	(0,2.5) node {$\eta^{\tens n}$}
							(1,1) node {$\mf s_\sigma$}
							(3,1) node {$l_{\tens_n}$};
			\end{tikzpicture}
		\end{align*}
		 The identities above follow from the composite of \eqref{lax monoidal structure} with $(j_n)_J$; the symmetry axiom for $J$ (see \exref{example: monoidal enriched profunctors}); the symmetry axiom for $d$; \eqref{lax monoidal structure} composed with $(j_n)_J$ again; the naturality of $\sigma$.
		 
		 We conclude that $l$, as a lax monoidal $\V$-functor, has a unique lift along the forgetful double functor $\map U{\sMonProf\V w}{\MonProf\V w}$. That $\eta$, as a cell in $\sMonProf\V w$, defines this lift as the pointwise left Kan extension of $d$ along $J$ follows readily from the fact that $U$ restricts to the identity on cells. This completes the proof.
	\end{proof}
  
  As a corollary to the previous proposition we can now recover Proposition 2.3 of \cite{Getzler09}. Given a symmetric monoidal $\V$-functor $\map jAB$ recall that the representable $\V$-profunctor $\hmap{j_*}AB$ is monoidal, by \propref{lifting restrictions}. It satisfies the left Beck-Chevalley condition if the transformation \mbox{$\nat{\bigpars{(j_*)_\tens}^*}{\tens^*_A \hc Tj_*}{B(j, \tens_B)}$} of profunctors $A \brar TB$, that is induced by the compositions
  \begin{align*}
  	A\bigpars{x, (y_1 &\tens \dotsb \tens y_n)} \tens B(jy_1, z_1) \tens \dotsb \tens B(jy_n, z_n) \\
  	&\xrar{\id \tens \tens_B} A\bigpars{x, (y_1 \tens \dotsb \tens y_n)} \tens B\bigpars{(jy_1 \tens \dotsb \tens jy_n), (z_1 \tens \dotsb \tens z_n)} \\
  	&\xrar{j \tens \lambda_{\inv{j_\tens}}} B\bigpars{jx, j(y_1 \tens \dotsb \tens y_n)} \tens B\bigpars{j(y_1 \tens \dotsb \tens y_n), (z_1 \tens \dotsb \tens z_n)} \\
  	&\xrar\mu B\bigpars{jx, (z_1 \tens \dotsb \tens z_n)},
  \end{align*}
  is invertible. To see what this means in the case of $\V = \Set$, we first notice that the composite profunctor $\hmap{\tens^*_A \hc Tj_*}A{TB}$, see \exref{example: profunctors}, is given by the sets \mbox{$(\tens^* \hc Tj_*)\bigpars{x, (z_1, \dotsc, z_n)}$} forming equivalence classes of sequences of morphisms
  \begin{displaymath}
  	\bigpars{x \xrar p (y_1 \tens \dotsb \tens y_n), jy_1 \xrar{u_1} z_1, \dotsc, jy_n \xrar{u_n} z_n},
  \end{displaymath}
  where $p \in A$ and $u_i \in B$, under the smallest equivalence relation $\sim$ that relates the pairs $(p, u_1, \dotsc, u_n) \sim (p', u'_1, \dotsc, u'_n)$ for which there exist maps $\map{r_i}{y_i}{y'_i}$ in $A$ that make the diagrams
  \begin{displaymath}
  	\begin{tikzpicture}[textbaseline]
  		\matrix(m)[math2em, row sep=1.0em]{& (y_1 \tens \dotsb \tens y_n) \\ x & \\ & (y'_1 \tens \dotsb \tens y'_n) \\};
  		\path[map]	(m-1-2) edge node[right] {$(r_1 \tens \dotsb \tens r_n)$} (m-3-2)
  								(m-2-1) edge node[above left] {$p$} (m-1-2)
  												edge node[below left] {$p'$} (m-3-2);
  	\end{tikzpicture} \qquad\qquad \text{and} \qquad\qquad \begin{tikzpicture}[textbaseline]
  		\matrix(m)[math2em, row sep=1.0em]{jy_i & \\ & z_i \\ jy'_i & \\};
  		\path[map]	(m-1-1) edge node[above right] {$u_i$} (m-2-2)
  												edge node[left] {$jr_i$} (m-3-1)
  								(m-3-1) edge node[below right] {$u'_i$} (m-2-2);
  	\end{tikzpicture}
  \end{displaymath}
  commute, the latter for each $i = 1, \dotsc, n$. In terms of $\sim$, the profunctor $j_*$ satisfies the left Beck-Chevalley condition if every morphism $\map v{jx}{(z_1 \tens \dotsb \tens z_n)}$ in $B$ can be represented as
  \begin{displaymath}
  	\begin{tikzpicture}
  		\matrix(m)[math2em]{jx & (z_1 \tens \dotsb \tens z_n) \\ j(y_1 \tens \dotsb \tens y_n) & (jy_1 \tens \dotsb \tens jy_n), \\};
  		\path[map]	(m-1-1) edge node[above] {$v$} (m-1-2)
  												edge node[left] {$jp$} (m-2-1)
  								(m-2-1) edge node[below] {$\inv{j_\tens}$} (m-2-2)
  								(m-2-2) edge node[right] {$(u_1 \tens \dotsb \tens u_n)$} (m-1-2);
  	\end{tikzpicture}
  \end{displaymath}
  where $p \in A$ and $u_i \in B$, and that any two such representations $(p, u_1, \dotsc, u_n)$ and $(p', u'_1, \dotsc, u'_n)$ are related under $\sim$.
  
  In the corollary below we have denoted by $\V\textup -\mathsf{sMonCat}$ the $2$-category of symmetric monoidal $\V$-categories, symmetric monoidal $\V$-functors and monoidal $\V$-natural transformations.
  \begin{corollary}[Getzler] \label{Getzler}
  	Let $\V$ be a closed symmetric monoidal category that is both complete and cocomplete. Consider symmetric monoidal $\V$-categories $A$, $B$ and $M$, as well as a symmetric monoidal $\V$-functor $\map jAB$. If $\hmap{j_*}AB$ satisfies the left Beck-Chevalley condition and $M$ is cocomplete, such that its tensor product preserves $\V$-weighted colimits in each variable, then the functor
  	\begin{displaymath}
  		\map{\lambda_j}{\V\textup -\mathsf{sMonCat}(B, M)}{\V\textup -\mathsf{sMonCat}(A,M)},
  	\end{displaymath}
  	that is given by precomposition with $j$, admits a left adjoint that is given by pointwise left Kan extension along $j$.
  \end{corollary}
  \begin{proof}
  	It is well known that, for any morphism $\map jAB$ in a $2$-category $\C$, the functor $\map{\lambda_j}{\C(B, M)}{\C(A,M)}$, given by precomposition with $j$, has a left adjoint precisely if all left Kan extensions along $j$ exist; see e.g.\ Section X.3 of \cite{MacLane98} for the case $\C = \Cat$. It is also well known that all pointwise left Kan extensions into the $\V$-category $M$ (in the sense of the final assertion of \propref{pointwise Kan extensions in terms of weighted limits}) exist whenever it is cocomplete; see Proposition 4.33 of \cite{Kelly82}. By the assumptions on $M$ and $j_*$ we may apply the previous proposition to these pointwise left Kan extensions, which shows that they are lifted to the double category $\sMonProf\V{ps}$. Therefore, to conclude the proof, it suffices to recall that (pointwise) left Kan extensions along $j_*$, in $\sMonProf\V{ps}$, correspond to left Kan extensions along $j$ in $\V\textup -\mathsf{sMonCat}$, by (the horizontal dual of) \propref{right Kan extensions along conjoints as right Kan extensions in V(K)}.
	\end{proof}
  
  \section{Application: free bicommutative Hopf monoids}
  In this last section we use \corref{Getzler} above to obtain a left adjoint to the forgetful functor
  \begin{displaymath}
  	\mathsf{ubicHopf}(M) \to \mathsf{ucocComon}(M),
  \end{displaymath}
  from the category of `unbiased' bicommutative Hopf monoids, in a suitable symmetric monoidal category $M$, to the category of unbiased cocommutative comonoids in $M$. The idea is to regard these categories as categories of algebras of `PROPs'. Informally, a PROP is an algebraic structure that defines a type of algebra that lives in a symmetric monoidal category, and whose operations may have multiple inputs and multiple outputs; the following formal definition is the unbiased variant of the original definition, that was given in Section 24 of \cite{MacLane65}.
  
  \begin{definition}
		A \emph{PROP} $\mb P$ is a symmetric strict monoidal category $\mb P$ whose monoid of objects equals $(\mb N, +, 0)$. A \emph{morphism} $\map\phi{\mb P}{\mb Q}$ of PROPs is a symmetric strict monoidal functor that restricts to the identity on objects.
  \end{definition}
  
  Unraveling this, a PROP $\mb P$ consists of a category $\mb P$ with $\ob \mb P = \mb N$, equipped with functors $\map{\tens_k}{\mb P^k}{\mb P}$, one for each $k \geq 0$, and symmetries $\map{\mf s_\sigma}nn$, that are functorial in $\sigma \in \Sigma_n$, satisfying the following axioms. Firstly the tensor products restrict to addition $(n_1, \dotsc, n_k) \mapsto n_1 + \dotsb + n_k$ on objects, with $0$ as unit, and they are strictly associative and unital in the sense that
  \begin{displaymath}
  	\bigpars{(\xi_{11} \tens \dotsb \tens \xi_{1l_1}) \tens \dotsb \tens (\xi_{k1} \tens \dotsb \tens \xi_{kl_k})} = (\xi_{11} \tens \dotsb \tens \xi_{kl_k}),
  \end{displaymath}
  for any double sequence of morphisms $\map{\xi_{ij}}{m_{ij}}{n_{ij}}$ in $\mb P$, and $(\xi) = \xi$ for each single morphism $\map\xi mn$. Secondly $(\mf s_{\sigma_1} \tens \dotsb \tens \mf s_{\sigma_k}) = \mf s_{(\sigma_1, \dotsc, \sigma_k)}$, where $(\sigma_1, \dotsc, \sigma_k)$ denotes the disjoint union of the permutations $\sigma_i \in \Sigma_{n_i}$, while the symmetry
  \begin{displaymath}
  	\map{\mf s_{\tau_{(n_1, \dotsc, n_k)}}}{n_1 + \dotsb + n_k}{n_{\tau 1} + \dotsb + n_{\tau k}},	
  \end{displaymath}
  for the block permutation $\tau_{(n_1, \dotsc, n_k)}$ that is induced by $\tau \in \Sigma_k$ (see \exref{example: symmetric monoidal enriched category}), is required to be natural in each $n_1, \dotsc, n_k$.
  
  \begin{example}
  	We denote by $\mb F$ the PROP that is (isomorphic to) the skeletal category of finite sets. More precisely, we think of a morphism $\map\xi mn$ in $\mb F$ as being a function $\map\xi{\brks m}{\brks n}$ between the ordinals $\brks m = \set{0 < \dotsb < m-1}$ and \mbox{$\brks n = \set{0 < \dotsb < n-1}$}. The tensor product $(\xi_1 \tens \dotsb \tens \xi_k)$ is given by disjoint union:
  	\begin{displaymath}
  		\brks{m_1 + \dotsb + m_k} \iso \brks{m_1} \djunion \dotsb \djunion \brks{m_k} \xrar{\xi_1 \djunion \dotsb \djunion \xi_k} \brks{n_1} \djunion \dotsb \djunion \brks{n_k} \iso \brks{n_1 + \dotsb + n_k},
  	\end{displaymath}
  	where the bijections are unique such that they preserve order; here the order on the disjoint union $\brks{m_1} \djunion \dotsb \djunion \brks{m_k}$ is induced by the orders of each $\brks{m_i}$ and by the rule that, for each pair $i < j$, all elements of $\brks{m_i}$ are less than those of $\brks{m_j}$. Finally, the symmetries $\map{\mf s_\sigma}nn$ are simply given by the permutations of $\brks n$.
  \end{example}
  
  \begin{example} \label{example: PROP H}
  	We denote by $\mb H$ the PROP that is (isomorphic to) the skeletal category of free, finitely generated abelian groups. Precisely, a morphism $\map\xi mn$ in $\mb H$ is a homomorphism $\map \xi{\mb Z^m}{\mb Z^n}$, which we identify with an $n \times m$-matrix with coefficients in $\mb Z$. The monoidal structure on $\mb H$ is given by disjoint sum, so that the tensor product of morphisms $\map{\xi_i}{m_i}{n_i}$ is identified with the block matrix
  	\begin{displaymath}
  		\begin{pmatrix}
				 \xi_1 & 0 & \hdots & 0 \\
				 0 & \xi_2 & \ddots & \vdots \\
				 \vdots & \ddots & \ddots & 0 \\
				 0 & \hdots & 0 & \xi_k \\
			\end{pmatrix}.
  	\end{displaymath}
  	The symmetries $\map{\mf s_\sigma}nn$ are given by permuting the generators of $\mb Z^n$.
  	
  	The opposite $\op{\mb F}$ of the PROP $\mb F$, which is again a PROP, can be embedded into $\mb H$ by mapping each morphism $m \to n$, corresponding to a function $\map\xi{\brks n}{\brks m}$, to the $n \times m$-matrix $(\xi_{ij})$ that is given by $\xi_{ij} = 1$ if $\xi i = j$ and $0$ otherwise. Thus the image of $\op{\mb F}$ in $\mb H$ is the subPROP consisting of all matrices that contain precisely one non-zero entry in each row, whose coefficient is $1$. This gives a morphism of PROPs $\map j{\op{\mb F}}{\mb H}$.
  \end{example}
  
  \begin{definition} \label{definition: algebra of PROP}
  	Let $\mb P$ be a PROP and $M = (M, \tens, \mf a, \mf i)$ a symmetric monoidal category. An \emph{algebra} $A$ of $\mb P$ in $M$ is a symmetric monoidal functor $\map A{\mb P}M$. A \emph{morphism} $\map fAB$ of such algebras is a monoidal transformation $\nat fAB$. In other words, the category of algebras of $\mb P$ in $M$ is defined as
  	\begin{displaymath}
  		\PAlg{\mb P, M} = \sMonCat(\mb P, M),
  	\end{displaymath}
  	where $\sMonCat$ denotes the $2$-category of symmetric monoidal categories, symmetric monoidal functors and monoidal transformations.
  \end{definition}
  
  	In describing algebras of PROPs it is useful to abbreviate $A^{\tens n} = (\overbrace{A \tens \dotsb \tens A}^{\text{$n$ times}})$.
  \begin{example}
  	An algebra of the PROP $\mb F$ in $M$ is essentially an \emph{unbiased commutative monoid}; that is, an object $A$ of $M$ equipped with a family $\mu$ of morphisms $\map{\mu_n}{A^{\tens n}}A$, one for each $n \geq 0$, that make the following diagrams commute, where on the right $\sigma$ is any permutation in $\Sigma_n$.
  	\begin{displaymath}
  		\begin{tikzpicture}[baseline]
  			\matrix(m)[math2em, column sep=5.5em]{(	A^{\tens n_1} \tens \dotsb \tens A^{\tens n_k}) & A^{\tens k} \\
  				A^{\tens(n_1 + \dotsb + n_k)} & A \\};
  			\path[map]	(m-1-1) edge node[above] {$(\mu_{n_1} \tens \dotsb \tens \mu_{n_k})$} (m-1-2)
  													edge node[left] {$\mf a$} (m-2-1)
  									(m-1-2) edge node[right] {$\mu_k$} (m-2-2)
  									(m-2-1) edge node[below] {$\mu_{n_1 + \dotsb + n_k}$} (m-2-2);
  		\end{tikzpicture} \qquad \begin{tikzpicture}[baseline]
  			\matrix(m)[math2em]{A & (A) \\ & A \\};
  			\path[map]	(m-1-1) edge node[above] {$\mf i$} (m-1-2)
  													edge node[below left] {$\id$} (m-2-2)
  									(m-1-2)	edge node[right] {$\mu_1$} (m-2-2);
  		\end{tikzpicture} \qquad \begin{tikzpicture}[baseline]
  			\matrix(m)[math2em]{A^{\tens n} & \\ A^{\tens n} & A \\};
  			\path[map]	(m-1-1) edge node[above right] {$\mu_n$} (m-2-2)
  													edge node[left] {$\mf s_\sigma$} (m-2-1)
  									(m-2-1) edge node[below] {$\mu_n$} (m-2-2);
  		\end{tikzpicture}
  	\end{displaymath}
  	To be precise, every algebra $\map A{\mb F}M$ induces the structure of an unbiased commutative monoid on $A(1)$, using the maps
  	\begin{displaymath}
  		A(1)^{\tens n} \xrar{A_\tens} A(n) \xrar{A\pars{\map !{\brks n}{\brks 1}}} A(1)
  	\end{displaymath}
  	where $\map !{\brks n}{\brks 1}$ denotes the unique map into the terminal set $\brks 1$. That this family satisfies the axioms above is easy to check. Moreover, it is not hard to show that the assignment $A \mapsto A(1)$ extends to an equivalence of categories
  	\begin{displaymath}
  		\PAlg{\mb F, M} \simeq \mathsf{ucMon}(M),
  	\end{displaymath}
  	where $\mathsf{ucMon}(M)$ denotes the category of unbiased commutative monoids in $M$, in which a morphism $\map f{(A, \mu_A)}{(B, \mu_B)}$ is simply a map $\map fAB$ in $M$ such that $\mu_B \of f^{\tens n} = f \of \mu_A$. In proving this it is helpful to notice that each morphism $\map \xi mn$ in $\mb F$ factors as
  	\begin{displaymath}
  		\begin{tikzpicture}
  			\matrix(m)[math2em, column sep=1.5em]{m & & n, \\ & m_1 + \dotsb + m_n & \\};
  			\path[map]	(m-1-1) edge node[above] {$\xi$} (m-1-3)
  													edge node[below left] {$\mf s_\sigma$} (m-2-2)
  									(m-2-2) edge node[below right] {$(! \tens \dotsb \tens !)$} (m-1-3);
  		\end{tikzpicture}
  	\end{displaymath}
  	where $\sigma \in \Sigma_m$ is unique up to composition with disjoint unions $(\tau_1, \dotsc, \tau_n)$ of symmetries $\tau_i \in \Sigma_{m_i}$.
  	
  	Dual to the above, in the notion of an \emph{unbiased cocommutative comonoid} $A$ the structure is given by a family of maps $\map{\delta_n}A{A^{\tens n}}$. Analogous to the equivalence above we have $\PAlg{\op{\mb F}, M} \simeq \mathsf{ucocComon}(M)$.
  \end{example}
  \begin{example}
  	An \emph{unbiased bicommutative bimonoid} in $M$ is a triple $A = (A, \mu, \delta)$ where $(A, \mu)$ forms an unbiased commutative monoid and $(A, \delta)$ forms an unbiased cocommutative comonoid, such that the compatibility diagrams on the left below commute. Here the isomorphism moves each copy of $A$ in $(A^{\tens n})^{\tens m}$, that is indexed by $(i, j) \in \brks m \times \brks n$, to the copy in $(A^{\tens m})^{\tens n}$ that is indexed by $(j, i) \in \brks n \times \brks m$.
  	\begin{displaymath}
  		\begin{tikzpicture}[baseline]
  			\matrix(m)[math, column sep=0.75em, yshift=3.5em]{(A^{\tens n})^{\tens m} & (A^{\tens m})^{\tens n} \\};
  			\matrix(n)[math2em, yshift=-1.72em]{A^{\tens m} & & A^{\tens n} \\ & A & \\};
  			\path[desc]	(m-1-1) edge node {$\iso$} (m-1-2);
  			\path[map]	(m-1-2) edge node[right, yshift=2pt] {$(\mu_m)^{\tens n}$} (n-1-3)
  									(n-1-1) edge node[left, yshift=2pt] {$(\delta_n)^{\tens m}$} (m-1-1)
  													edge node[below left] {$\mu_m$} (n-2-2)
  									(n-2-2) edge node[below right] {$\delta_n$} (n-1-3);
  		\end{tikzpicture} \qquad \qquad \begin{tikzpicture}[baseline]
	  		\matrix(m)[math2em]{(A \tens A) & & (A \tens A) \\ A & () & A \\ (A \tens A) & & (A \tens A) \\};
	  		\path[map]	(m-1-1) edge node[above] {$(S \tens \id)$} (m-1-3)
	  								(m-1-3) edge node[right] {$\mu_2$} (m-2-3)
	  								(m-2-1) edge node[left] {$\delta_2$} (m-1-1)
	  												edge node[below] {$\mu_0$} (m-2-2)
	  												edge node[left] {$\delta_2$} (m-3-1)
	  								(m-2-2) edge node[below] {$\delta_0$} (m-2-3)
	  								(m-3-1) edge node[below] {$(\id \tens S)$} (m-3-3)
	  								(m-3-3) edge node[right] {$\mu_2$} (m-2-3);
	  	\end{tikzpicture}
  	\end{displaymath}
  	A morphism $A \to B$ of unbiased bicommutative bimonoids is a map $A \to B$ in $M$ that is simultaneously a map of monoids and comonoids. Furthermore, an unbiased bicommutative bimonoid $A$ is called an \emph{unbiased bicommutative Hopf monoid} whenever there exists a map $\map SAA$ that makes the diagram on the right above commute. Such a map, which is necessarily unique (see e.g.\ Lemma 37 of \cite{Porst13}), is called the \emph{antipode} of $A$.
  	
  	The algebras of the PROP $\mb H$ are essentially the unbiased bicommutative Hopf monoids in $M$: each algebra $\map A{\mb H}M$ induces the structure of such a Hopf monoid on $A(1)$ using the maps
  	\begin{flalign*}
  		&& \mu_n &= \bigbrks{A(1)^{\tens n} \xrar{A_\tens} A(n) \xrar{A(1 \hdots 1)} A(1)}, & \\
  		&& \delta_n &= \bigbrks{A(1) \xrar{A(1 \hdots 1)^{\textup t}}A(n) \xrar{\inv{A_\tens}} A(1)^{\tens n}} & \\
  		\text{and} && S &= \bigbrks{A(1) \xrar{A(-1)} A(1)},
  	\end{flalign*}
  	where $(1 \hdots 1)^{\textup t}$ denotes the transpose of the $1 \times n$-matrix $(1 \hdots 1)$. That these maps satisfy the bicommutative bimonoid axioms and the Hopf monoid axiom follows directly from various matrix equations. For example, the top of the diagram on the right above is the $A$-image of the matrix equation
\begin{displaymath}
	\begin{pmatrix}1 & 1 \\\end{pmatrix} \begin{pmatrix}-1 & 0 \\ 0 & 1 \\\end{pmatrix} \begin{pmatrix}1 \\ 1 \\\end{pmatrix} = \begin{pmatrix}0 \\\end{pmatrix}.
\end{displaymath}
		One can show that the assignment $A \mapsto A(1)$ extends to an equivalence of categories
		\begin{displaymath}
			\PAlg{\mb H, M} \simeq \mathsf{ubicHopf}(M),
		\end{displaymath}
		where $\mathsf{ubicHopf}(M)$ denotes the category of unbiased bicommutative Hopf monoids in $M$. In fact it is shown in \cite{Wadsley08} that, in the case that $M$ is symmetric strict monoidal, the category of symmetric strict monoidal functors $\mb H \to M$ is isomorphic to the category of biased bicommutative Hopf monoids; it is straightforward to modify (parts of) the proof given there into a proof for the equivalence above.
  \end{example}
  
  As promised, we finish by showing that unbiased bicommutative Hopf monoids that are freely generated by unbiased cocomutative comonoids exist. For a much more detailed and thorough treatment of free, and cofree, non-commutative Hopf monoids we refer the reader to \cite{Porst13}.
  \begin{theorem}
  	Let $M$ be a cocomplete symmetric monoidal category whose tensor product preserves colimits in each variable. The forgetful functor
  	\begin{displaymath}
  		\mathsf{ubicHopf}(M) \to \mathsf{ucocComon}(M),
  	\end{displaymath}
  	of unbiased bicommutative Hopf monoids to unbiased cocommutative comonoids, admits a left adjoint.
  \end{theorem}
  \begin{proof}
  	We may just as well prove that the composite of the forgetful functor with the equivalences $\PAlg{\mb H, M} \simeq \mathsf{ubicHopf}(M)$ and $\mathsf{ucocComon}(M) \simeq \PAlg{\op{\mb F}, M}$ admits a left adjoint. In light of \defref{definition: algebra of PROP}, this composite is simply the functor
  	\begin{displaymath}
  		 \map{\lambda_j}{\mathsf{sMonCat}(\mb H, M)}{\mathsf{sMonCat}(\op{\mb F},M)}	
  	\end{displaymath}
  	that is given by precomposition with the embedding of PROPs $\map j{\op{\mb F}}{\mb H}$. Hence the proof follows if we can show that we may apply \corref{Getzler}, which means that we have to show that the companion $\hmap{j_*}{\op{\mb F}}{\mb H}$ satisfies the left Beck-Chevalley condition.
  	
  	Thus we consider a morphism $\xi$ in $\mb H$ that is of the form $\map\xi m{n_1 + \dotsb + n_k}$; following the discussion preceding \corref{Getzler}, we have to show that it can be uniquely represented as a composite $\xi = (\xi_1 \tens \dotsb \tens \xi_k) \of \zeta$, where the $\map{\xi_i}{l_i}{n_i}$ are maps in $\mb H$ and $\map\zeta m{l_1 + \dotsb + l_k}$ is a map in $\op{\mb F}$. Recalling the definitions of $\mb H$ and $\map j{\op{\mb F}}{\mb H}$ (see \exref{example: PROP H}), this means that $\xi$ is a $(n_1 + \dotsb + n_k) \times m$-matrix with coefficents in $\mb Z$, the $\xi_i$ are such $n_i \times l_i$-matrices, and $\zeta$ is an $(l_1 + \dotsb + l_k) \times m$-matrix that contains precisely one non-zero entry in each row, whose coefficent is $1$. It is easy to show that such a representation exists: we take $l_i = m$ and let $\xi_i$ be the blocks of $\xi$ as shown below. As the $km \times m$-matrix $\zeta$ we take the block matrix of identity matrices $I_m$ of dimension $m$, as on the right. 
		\begin{displaymath}
			\xi = \begin{pmatrix} \xi_1 \\ \xi_2 \\ \vdots \\ \xi_k \end{pmatrix} =	\begin{pmatrix}
					 \xi_1 & 0 & \hdots & 0 \\
					 0 & \xi_2 & \ddots & \vdots \\
					 \vdots & \ddots & \ddots & 0 \\
					 0 & \hdots & 0 & \xi_k \\
				\end{pmatrix} \begin{pmatrix} I_m \\ I_m \\ \vdots \\ I_m \end{pmatrix}
		\end{displaymath}
		
		It remains to prove that the representation $\xi = (\xi_1 \tens \dotsb \tens \xi_k) \of \zeta$ is unique, in the sense explained before \corref{Getzler}. Hence consider another such representation $\xi = (\xi'_1 \tens \dotsb \tens \xi'_k) \of \zeta'$, where each $\xi'_i$ is a $n_i \times l'_i$-matrix and $\zeta'$ is a $(l'_1 + \dotsb + l'_k)\times m$\ndash matrix in $\op{\mb F} \subset \mb H$; it suffices to show that there exist $l'_i \times m$-matrices $\chi_i$ in $\op{\mb F}$ satisfying $\zeta' = (\chi_1 \tens \dotsb \tens \chi_k) \of \zeta$ as well as $\xi_i = \xi'_i \of \chi_i$, for each $i = 1, \dotsc, k$. Using the same trick, we take the $\chi_i$ to be defined by the identity
		\begin{displaymath}
			\zeta' = \begin{pmatrix} \chi_1 \\ \chi_2 \\ \vdots \\ \chi_k \end{pmatrix} =	\begin{pmatrix}
					 \chi_1 & 0 & \hdots & 0 \\
					 0 & \chi_2 & \ddots & \vdots \\
					 \vdots & \ddots & \ddots & 0 \\
					 0 & \hdots & 0 & \chi_k \\
				\end{pmatrix} \begin{pmatrix} I_m \\ I_m \\ \vdots \\ I_m \end{pmatrix}.
		\end{displaymath}
		Cleary $\zeta' \in \op{\mb F}$ implies that each $\chi_i \in \op{\mb F}$, while  $\xi_i = \xi'_i \of \chi_i$ follows by postcomposing both sides of the identity above with $(\xi'_1 \tens \dotsb \tens \xi'_k)$.
  \end{proof}

\end{document}